\documentclass[a4paper,a4wide,12pt]{article} 
\usepackage[utf8]{inputenc}
\usepackage{amsmath,amsthm, amssymb}
\usepackage{lmodern}
\usepackage{enumitem} 
\usepackage{chngcntr}
\usepackage{dsfont}
\usepackage{geometry}
\usepackage{color}
\usepackage{graphicx}
\usepackage{float}
\newtheorem{theorem}{Theorem}[section]
\newtheorem{lemma}[theorem]{Lemma}
\newtheorem{prop}[theorem]{Proposition}
\newtheorem{problem}[theorem]{Problem}
\theoremstyle{definition}
\newtheorem{defn}[theorem]{Definition}
\newtheorem{cor}[theorem]{Corollary}
\newtheorem{rem}[theorem]{Remark}

\DeclareMathOperator*\ox{e_{1}}
\DeclareMathOperator*\oy{e_{2}}
\DeclareMathOperator*\ang{angle}

\DeclareMathOperator*\dist{dist}
\DeclareMathOperator*\hdim{dim_{H}}
\DeclareMathOperator*\s{\Sigma}
\DeclareMathOperator*\Om{\Omega}
\DeclareMathOperator*\diam{diam}

\DeclareMathOperator*\lip{Lip}

\DeclareMathOperator*\ess{ess\,sup}

\newcommand*\diff{\mathop{}\, d}
\newcommand\numberthis{\addtocounter{equation}{1}\tag{\theequation}}

\newcommand\restr[2]{\ensuremath{\left.#1\right|_{#2}}}

\numberwithin{equation}{section}
\numberwithin{figure}{section}

\title{Regularity for the planar optimal $p$-compliance problem}

\author{Bohdan Bulanyi\footnote{LJLL UMR 7598, Universit\'e de Paris, France. e-mail: bulanyi@math.univ-paris-diderot.fr} \and Antoine Lemenant\footnote{IECL UMR 7502, Universit\'e de Lorraine, Nancy, France. e-mail: antoine.lemenant@univ-lorraine.fr}}

\date{\today}

\begin{document}

\maketitle


\begin{abstract}
In this paper we prove a partial $C^{1,\alpha}$ regularity result in dimension $N=2$ for the  optimal $p$-compliance problem, extending for $p\not = 2$ some of the results obtained by A. Chambolle, J. Lamboley, A. Lemenant, E. Stepanov (2017). Because of the lack of good monotonicity estimates for the $p$-energy when $p\not = 2$, we employ an alternative technique based on a compactness argument leading to a $p$-energy decay at any flat point. We finally obtain that every optimal set has no loop, is Ahlfors regular, and  is $C^{1,\alpha}$ at $\mathcal{H}^{1}$-a.e. point for every $p \in (1 ,+\infty)$. 
\end{abstract}

\tableofcontents

\section{Introduction}

For an open set $U\subset \mathbb{R}^{2}$ and $p\in (1,+\infty)$ denote by $W^{1,p}_{0}(U)$ the closure of $C^{\infty}_{0}(U)$ in the Sobolev space $W^{1,p}(U)$, where $C^{\infty}_{0}(U)$ is the space of functions in $C^{\infty}(U)$ with compact support in $U$. Let $\Om$ be an open and bounded subset of $\mathbb{R}^{2}$, and let  $p \in (1, +\infty)$. For each $u \in W^{1,p}_{0}(\Om)$, we define
\[
E_{p}(u)= \frac 1p \int_{\Om} |\nabla u|^{p}\diff x - \int_{\Om} fu \diff x .
\]

Thanks to the Sobolev inequalities (see \cite[Theorem 7.10]{PDE}) the functional $E_{p}$ is finite on $W^{1,p}_{0}(\Omega)$ when $f \in L^{q_{0}}(\Om)$ with $q_{0}=q_{0}(p)$ such that 
\begin{equation}
q_{0}=\frac{2p}{3p-2}\,\ \text{if $1<p<2$}, \qquad q_{0}>1 \,\ \text{if $p=2$}, \qquad q_{0}=1 \,\ \text{if} \,\ p>2. \label{(0.1)}
\end{equation}

 It is classical that for each closed proper subset $\Sigma$ of $\overline{\Omega}$ the functional $E_{p}$ admits a unique minimizer $u_{\s}$ over $W^{1,p}_{0}(\Omega\backslash \Sigma)$, which is the solution of the Dirichlet problem
\begin{equation} \label{(1.1)}
\begin{cases}
-\Delta_{p}u & =\,\ f \,\ \text{in}\,\ \Omega\backslash \Sigma \\
\qquad u & =\,\ 0 \,\ \text{on}\,\ \Sigma \cup \partial\Omega
\end{cases}
\end{equation}
in the weak sense, which means that
\begin{equation} \label{(1.2)}
\int_{\Om} |\nabla u_{\s}|^{p-2} \nabla u_{\s} \nabla \varphi \diff x= \int_{\Om}f\varphi \diff x
\end{equation}
for all $\varphi \in W^{1,p}_{0}(\Om \backslash \s)$.

Following \cite{opt}, we can interpret $\Omega$ as a membrane which is attached along $\s \cup \partial \Om$ to some fixed base (where $\s$ can be interpreted as a ``glue line") and subjected to a given force $f$. Then $u_{\s}$ is the displacement of the membrane. The rigidity of the membrane is measured through the $p$-compliance functional, which is defined as 
\[
C_{p}(\s)=-E_{p}(u_{\s})=\frac {1}{p^{\prime}} \int_{\Om} |\nabla u_{\s}|^{p}\diff x = \frac {1}{p^{\prime}} \int_{\Om}f u_{\s}\diff x.
\]
We study the following shape optimization problem.
\begin{problem} \label{problemMain}
Given $\lambda>0$, find a set $\s \subset \overline{\Om}$ minimizing the functional $\mathcal{F}_{\lambda,p}$ defined by
\[
\mathcal{F}_{\lambda,p}({\s}^{\prime})=C_{p}({\s}^{\prime})+ \lambda \mathcal{H}^{1}({\s}^{\prime})
\]
among all sets $\Sigma^{\prime} \in \mathcal{K}(\Omega)$, where $\mathcal{K}(\Omega)$ is the class of all closed connected proper subsets of $\overline{\Om}$.
\end{problem}

The physical interpretation of this problem may be the following: we are trying to find the best location $\Sigma$ for the glue to put on the membrane $\Omega$ in order to maximize the rigidity of the latter, subject to the force $f$, while the penalization by $\lambda \mathcal{H}^{1}$ takes into account the quantity (or cost) of the glue.

Without loss of generality, we assume that the force field $f$ is nonzero, because otherwise for any $\Sigma \in \mathcal{K}(\Omega)$ we would have $C_{p}(\Sigma)=0$ and then every solution of Problem~\ref{problemMain} would be either a point $x \in \overline{\Omega}$ or the empty set.

In this paper, we prove some regularity properties about minimizers of Problem~\ref{problemMain}. In particular, we prove that a minimizer  has no loop (Theorem~\ref{thm: 8.1}), is Ahlfors regular (Theorem~\ref{thm: 5.3}) and, furthermore, we establish some $C^{1,\alpha}$ regularity properties.

Most of our results will hold under some integrability condition on the second member~$f$. Namely we define 
\begin{eqnarray}
q_1=
\left\{
\begin{array}{cl}
\frac{2p}{2p-1} & \text{ if }\,\ 2\leq p<+\infty \\
\frac{2p}{3p-3} & \text{ if }\,\ 1<p < 2,
\end{array}
\right. \label{qrestrict}
\end{eqnarray}
and we notice that $q_1\geq q_0$.  As  will be shown later (see Lemma~\ref{lemma 3.3}), asking $f\in L^{q_1}(\Omega)$ for $2\leq p<+\infty$ is natural, since it seems to be the right exponent which implies an estimate of the type $\int_{B_{r}(x_{0})} |\nabla u|^p \diff x \leq C r$ for the solution  $u$ of the Dirichlet problem
\[
-\Delta_p u=f \,\ \text{in} \,\ B_{r}(x_{0}),\,\ u \in W^{1,p}_{0}(B_{r}(x_{0})),
\]
which is the kind of estimate that we are looking for to establish regularity properties on a minimizer $\s$ of Problem~\ref{problemMain}.

The main regularity result of this paper is the following.

\begin{theorem} \label{thm: 9.16} Let $\Om \subset \mathbb{R}^{2}$ be an open bounded set,\,\ $ p \in (1,+\infty),\, f \in L^{q}(\Omega)$ with    $q> q_1$, where $q_{1}$ is defined in \eqref{qrestrict}.   
Let $\s \subset \overline{\Omega}$ be a minimizer of Problem~\ref{problemMain}. Then there is a constant $\alpha \in (0,1)$ such that for $\mathcal{H}^{1}$-a.e. point $x \in \s \cap \Omega $ one can find a radius $r_{0}>0$, depending on $x$, such that $\s \cap \overline{B}_{r_{0}}(x)$ is a $C^{1,\alpha}$ regular curve.
\end{theorem}
 Notice that Theorem~\ref{thm: 9.16} is interesting only in the case when $\diam(\Sigma)>0$, which happens to be true at least for some small enough values of $\lambda$ (see Proposition~\ref{prop: 2.16}). Furthermore, we have proved that every minimizer $\Sigma$ of Problem~\ref{problemMain} cannot contain quadruple points in $\Omega$ (see Proposition~\ref{remark about quadruple points}), i.e., there is no point $x \in \Sigma \cap \Omega$ such that for some sufficiently small radius $r>0$ the set $\Sigma \cap \overline{B}_{r}(x)$ is a union of four distinct $C^{1}$ arcs, each of which meets at point $x$ exactly one of the other three at an angle of $180^{\circ}$ degrees, and each of the other two at an angle of $90^{\circ}$ degrees.


 Problem~\ref{problemMain}  was studied earlier in the particular case $p=2$ in \cite{opt} for which a full regularity result was proved. 
 It is worth mentioning that our result generalizes some of the results of  \cite{opt} for $p\not =2$, but contains also better results in the special case $p=2$ as well. Indeed, our integrability condition $q>q_1$ on the second member $f$ for the particular case $p=2$ yields $q>\frac{4}{3}$ for the $\varepsilon$-regularity result to hold, which is slightly better than the one  in \cite{opt} for which $q>2$ was required. According to our Ahlfors-regularity result (see Theorem~\ref{thm: 5.3}), it holds under the mild integrability assumption $q=\frac{2p}{2p-1}$ and is proved up to the boundary (for a Lipschitz domain $\Omega$), which for the particular case $p=2$ generalizes the earlier result in \cite{opt}. We shall explain later more in detail the main technical differences between the case $p\not =2$ with respect to the case $p=2$. 
 
 Now let us emphasize that our partial $C^{1,\alpha}$ regularity result is internal; therefore  in Theorem~\ref{thm: 9.16} we do not require any regularity for $\partial \Omega$. On the other hand, as mentioned earlier, to prove our Ahlfors-regularity result, we required $\Omega$ to be a Lipschitz domain. We do not know whether the restriction on Lipschitz domains is needed to prove the Ahlfors regularity of minimizers of Problem~\ref{problemMain} which have at least two points. However, according to the proof of Theorem~\ref{thm: 5.3}, for each open set $\Omega^{\prime}\subset\subset \Omega$, there exist $C_{0}=C_{0}(p,q_{0},\|f\|_{(2p)^{\prime}},\lambda)>0$ and $r_{0}=r_{0}(\Omega^{\prime}, \Omega)>0$ such that if $\Sigma$  is a minimizer of Problem~\ref{problemMain}, then $\mathcal{H}^{1}(\Sigma\cap B_{r}(x))\leq C_{0}r$ whenever $x \in \Sigma\cap \overline{\Omega^{\prime}}$ and $0<r\leq r_{0}$.

 In the limit $p\to +\infty$, Problem~\ref{problemMain} in some sense converges to the so-called \emph{average distance problem} (see \cite[Theorem 3]{Butazzo-Santambrogio}) which was also widely studied in the literature and for which it is known that minimizers may not be $C^1$ regular (see \cite{MR3165284}). Our result can therefore be considered as making a link between  $p=2$ and $p=+\infty$, although it actually works for any $p\in (1,+\infty)$.

A constrained variant of the same problem was also studied in \cite{Butazzo-Santambrogio, MR3063566,MR3195349} for $p\not =2$ in dimension 2 and greater, but focusing on different type of questions. In particular, no regularity results were available before with $p\not =2$.

As a matter of fact, even if the present paper is restricted to dimension 2 only, the same problem can be defined in higher dimension, provided that $p\in (N-1, +\infty)$, still with a penalization with the one dimensional Hausdorff measure. This    instance of the problem in higher dimensions seems to be very original, leading to a free-boundary  type  problem  with a high  co-dimensional free boundary set $\Sigma$.  Due to the low dimension of the ``free-boundary'' in dimension $N>2$, most of the usual competitors are no more valid and some new ideas and new tools have to be used. 
 
 The present paper can therefore be seen as a preliminary step toward the regularity in any dimensions, focusing on the particular case of dimension 2. This approach is pertinent because   in dimension 2 only, the ``free boundary'' $\Sigma$ is of codimension 1, thus many standard arguments and competitors are  available. Let us highlight three places where we have taken the advantage of working in dimension 2, which does not extend in a trivial manner in higher dimensions. Firstly, in our proof of Ahlfors-regularity, we use (in the ``internal case") the set $(\Sigma\backslash B_{r}(x))\cup \partial B_r(x)$ as a competitor for $\Sigma$. But in dimension $N>2$ we cannot effectively use such a competitor, because $\partial B_{r}(x)$ has infinite $\mathcal{H}^{1}$-measure. Secondly, in the proof of Lemma~\ref{lem: 7.4},  we use a reflection technique to estimate a $p$-harmonic function in $B_{1}\backslash ((-1,1)\times \{0\})$ that vanishes on $(-1,1)\times \{0\}$, which is no more valid for a $p$-harmonic function in $B_{1}\backslash ((-1,1)\times \{0\}^{N-1})$ that vanishes on $(-1,1)\times \{0\}^{N-1}$ if $N>2$. Thirdly, in the density estimate in Proposition~\ref{prop: 9.8}, when $\Sigma$ is $\varepsilon r$-close, in a ball $\overline{B}_{r}(x_{0})$ and in the Hausdorff distance, to a diameter $[a,b]$ of $\overline{B}_{r}(x_{0})$, we  use as a competitor the set $\Sigma^{\prime}=(\Sigma\setminus B_{r}(x_{0})) \cup [a,b] \cup W$, where 
 \[
 W=\partial B_{r}(x_{0})\cap \{y: \dist(y, [a,b])\leq \varepsilon r\}.
 \]
 But in dimension $N> 2$ we cannot effectively use the above competitor because it has infinite $\mathcal{H}^{1}$-measure. 
 
However, we believe that some techniques developed in this paper could be useful to prove a  similar result in higher dimensions as well. This will be the purpose of a forthcoming work. 

Nevertheless, even in dimension 2, we have to face several technical new difficulties with $p\not =2$ compared to the  work for $p=2$ in \cite{opt}   that we shall try to explain now.

 One of the most difficulty is the lack of good monotonicity estimates for the $p$-energy. Indeed, the monotonicity of energy is one of the main tool in the case $p=2$ in  \cite{opt} which does not work anymore for $p\not =2$. A big part of the work in \cite{opt} relies on blow-up techniques from the Mumford-Shah functional which cannot be used anymore in our context, without a good monotonicity formula. This is why, even if we expect the minimizer, as for $p=2$, to be a finite union of $C^{1,\alpha}$ curves, we prove only  $C^{1,\alpha}$ regularity at $\mathcal{H}^{1}$-a.e. point.

\noindent{\bf Comments about the proof.}  In the proof of $C^{1,\alpha}$ regularity, as many other free boundary or free discontinuity problems, one of the main point is to prove a decay estimate on the local energy around a flat point. In other words we need to prove that the ``normalized" energy
  $$r\mapsto \frac{1}{r}\int_{B_r(x_{0})}|\nabla u_\Sigma|^p \diff x$$
  converges to zero sufficiently fast at a point $x_{0}\in \Sigma$, like a power of the radius, and this is where our proof differs from the case $p=2$.

    In the case $p=2$, the decay on the ``normalized" energy is obtained using a so-called \emph{monotonicity formula} that was inspired by the one of A. Bonnet on the Mumford-Shah functional \cite{Bonnoet}. This monotonicity formula is also the key tool in the classification of blow-up limits.

  For $p\not = 2$, an analogous monotonicity formula can still be established for the $p$-energy,  but  the resulting power of $r$ in that monotonicity formula is not large enough for our purposes and thus cannot be  used to prove $C^{1,\alpha}$ estimates. Consequently, we also miss a great tool which prevents us to establish the classification of blow-up limits. As the $p$-monotonicity is not strong enough to get $C^{1,\alpha}$ regularity we therefore use another strategy, arguing by contradiction and compactness: we know that $\int_{B_{r}}|\nabla u|^p \diff x$ behaves like  $Cr^{2}$ for $r\in (0,1/2]$ if $u$ is a $p$-harmonic function in $B_{1}\backslash P$ vanishing on $P\cap B_{1}$, where $P$ is an affine line passing through the origin, thus by compactness $\int_{B_{r}} |\nabla u|^{p}\diff x$ still has a similar behavior when $u$ is a $p$-harmonic function in $B_{1}\backslash \Sigma$ vanishing on $\Sigma \cap B_{1}$, when $\Sigma$ locally stays $\varepsilon$-close to a line.  
 
 Actually, as the compliance is a min-max type problem,  the true quantity to control  is not exactly $\int_{B_r(x_{0})}|\nabla u_\Sigma|^p \diff x$, but rather this other variant, as already defined and denoted by $\omega_\Sigma(x_0,r)$ in \cite{opt},

\begin{equation}
\omega_\Sigma(x_0,r)=\sup_{ \s^{\prime} \in \mathcal{K}(\Omega); \s^{\prime}\Delta \s \subset \overline{B}_{r}(x_{0}) } \frac{1}{r} \int_{B_{r}(x_{0})}|\nabla u_{\s^{\prime}}|^{p}\diff x.  \notag
\end{equation}

It can be shown that the quantity $\omega_\Sigma(x_0,r)$ controls, in many circumstances,  the square of the flatness, leading to some $C^{1,\alpha}$ estimates when $\omega_\Sigma(x_{0},r)$ decays fast enough. 

In \cite{opt} the decay of the above quantity was still obtained by use of the monotonicity formula, applied to the function $u_{\Sigma'}$, where $\Sigma'$ is a maximizer in the definition of $\omega_\Sigma(x_0,r)$.

As a consequence of our compactness argument, which provides a decay only for a closed connected set $\Sigma'$ staying $\tau$-close to a line, we need to introduce and work  with the following slightly more complicated quantity

\begin{equation}
w^{\tau}_{\s}(x_{0},r)=\sup_{\substack{\s^{\prime} \in \mathcal{K}(\Omega),\, \s^{\prime}\Delta \s \subset \overline{B}_{r}(x_{0}), \\\mathcal{H}^{1}(\Sigma^{\prime})\leq 100 \mathcal{H}^{1}(\Sigma),\,  \beta_{\s^{\prime}}(x_{0},r)\leq \tau}}  \frac{1}{r} \int_{B_{r}(x_{0})}|\nabla u_{\s^{\prime}}|^{p}\diff x, \notag
\end{equation}
where $\beta_{\Sigma^{\prime}}(x_{0},r)$ is the flatness defined by 

\begin{equation*}
\beta_{\Sigma^{\prime}}(x_{0},r)=\inf_{P\ni x_{0}} \frac{1}{r}d_{H}(\Sigma^{\prime}\cap \overline{B}_{r}(x_{0}), P\cap \overline{B}_{r}(x_{0})),
\end{equation*}
(the infimum being taken over the set of all affine lines $P$ passing through $x_{0}$),  where $d_{H}$ is the Hausdorff distance that for any nonempty sets $A,\, B\subset \mathbb{R}^{2}$ is defined by
\[
d_{H}(A,B)=\max\biggl\{\sup_{x \in A} \dist(x, B),\, \sup_{x\in B}\dist(x,A)\biggr\}.
\]
We also agree that for a nonempty set $A\subset \mathbb{R}^{2}$, $d_{H}(A,\emptyset)=d_{H}(\emptyset,A)=+\infty$ and that $d_{H}(\emptyset, \emptyset)=0$. Notice that the assumption $\mathcal{H}^{1}(\Sigma^{\prime}) \leq 100\mathcal{H}^{1}(\Sigma)$ in the definition of $w^{\tau}_{\Sigma}(x_{0},r)$ is rather optional, however, it guarantees  that if $\Sigma^{\prime}$ is a maximizer in the definition of $w^{\tau}_{\Sigma}(x_{0},r)$, then $\Sigma^{\prime}$ is arcwise connected.
 
We indeed obtain a decay of $\omega^\tau_\Sigma(x_0,r)$ provided that $\beta_\Sigma(x_{0},r)$ stays under control, which finally leads to the desired $C^{1,\alpha}$ result, and the same kind of estimate is also used to prove the absence of loops.

\section{Preliminaries}

\subsection{Definitions}

\begin{defn} \label{def. 2.1} \textit{ Let $U$ be a bounded open set in $\mathbb{R}^{2}$ and let $p \in (1,+\infty)$. We say that $u \in W^{1,p}(U)$ is a weak solution of the $p$-Laplace equation in $U$, if
\[
\int_{U} |\nabla u|^{p-2}\nabla u \nabla \varphi \diff x=0
\]
for each $\varphi \in W^{1,p}_{0}(U)$.}  
\end{defn}

We recall the following basic result for weak solutions (see \cite[Theorem 2.7]{Lindqvist}).
\begin{theorem} \label{thm: 2.2}
Let $U$ be a bounded open set in $\mathbb{R}^{2}$ and let $u \in W^{1,p}(U)$. The following two assertions are equivalent.
\begin{enumerate}[label=(\roman*)]
\item $u$ is minimizing:
\[
\int_{U} |\nabla u|^{p}\diff x \leq \int_{U}|\nabla v|^{p}\diff x, \,\ \text{when}\,\ v-u \in W^{1,p}_{0}(U);
\]
\item the first variation vanishes:
\[
\int_{U} |\nabla u|^{p-2}\nabla u \nabla \zeta \diff x=0,\,\ \text{when}\,\ \zeta \in W^{1,p}_{0}(U).
\]
\end{enumerate}
\end{theorem}
Now we introduce the notion of the Bessel capacity (see e.g. \cite{Potential}, \cite{Ziemer}) which is crucial in the investigation of the pointwise behavior of Sobolev functions and in describing the appropriate class of negligible sets with respect to the appropriate Lebesgue measure.
\begin{defn} \label{def 2.3} \textit{ For $p\in (1,+\infty)$, the Bessel $(1,p)$-capacity of a set $E\subset \mathbb{R}^{2}$ is defined as
\[
{\rm Cap}_{p}(E)=\inf \{\|f\|^{p}_{p} :\, g*f \geq 1\,\ \text{on}\,\ E,\,\ f \geq 0\},
\]
where the Bessel kernel $g$ is defined as that function whose Fourier transform is}
\[
\hat{g}(\xi)=(2\pi)^{-1}(1+|\xi|^{2})^{-1/2}.
\]
\end{defn}
We say that a property holds $p$-quasi everywhere (abbreviated as $p$-q.e.) if it holds except on a set $A$ where ${\rm Cap}_{p}(A)=0$. 

It is worth mentioning that by \cite[Corollary 2.6.8]{Potential} for each $p\in (1,+\infty)$ the notion of the Bessel capacity ${\rm Cap}_{p}$ is equivalent to the following 
 \[
 \widetilde{{\rm Cap}_{p}}(E)=\inf_{u \in W^{1,p}(\mathbb{R}^{2})}\biggl\{\int_{\mathbb{R}^{2}}|\nabla u|^{p}\diff x + \int_{\mathbb{R}^{2}}|u|^{p}\diff x : u \geq 1 \,\ \text{on some neighborhood of $E$}\biggr\}
 \]
 in the sense that there is a constant $C=C(p)>0$ such that for any set $E\subset \mathbb{R}^{2}$ one has
 \[
 \frac{1}{C}\widetilde{{\rm Cap}_{p}}(E) \leq {\rm Cap}_{p}(E) \leq C\widetilde{{\rm Cap}_{p}}(E).
 \]

The next theorems and propositions are stated here for convenience.

\begin{theorem} \label{thm: 2.4} If $p \in (1, 2]$, then ${\rm Cap}_p(E)=0$ if $\mathcal{H}^{2-p}(E)<+\infty$. Conversely, if ${\rm Cap}_p(E)=0$, then $\mathcal{H}^{2-p+\varepsilon}(E)=0$ for every $\varepsilon>0$.
\end{theorem}

\begin{proof} For the proof of the fact that  ${\rm Cap}_p(E)=0$ if $\mathcal{H}^{2-p}(E)<+\infty$ we refer the reader to \cite[Theorem 5.1.9]{Potential}. The fact that ${\rm Cap}_{p}(E)=0$ implies $\mathcal{H}^{2-p+\varepsilon}(E)=0$ for every $\varepsilon>0$ is the direct consequence of \cite[Theorem 5.1.13]{Potential}. 
\end{proof}

\begin{rem} \label{rem: 2.6} Let $p \in (2,+\infty)$. Then there is a constant $C=C(p)>0$ such that  if $E \not = \emptyset $, then ${\rm Cap}_p(E)\geq C$. In fact, one can take $C=  {\rm Cap}_{p}(\{(0,0)\})$ which is positive by \cite[Proposition 2.6.1 (a)]{Potential} and use the fact that the Bessel $(1,p)$-capacity is an invariant under translations and is nondecreasing with respect to set inclusion.
\end{rem}

Recall that for all $E\subset \mathbb{R}^{2}$ the number
\[
\hdim(E)=\sup\{s \in \mathbb{R}^{+}: \mathcal{H}^{s}(E)=+\infty\}=\inf\{t\in \mathbb{R}^{+}: \mathcal{H}^{t}(E)=0\}
\]
is called the Hausdorff dimension of $E$.
\begin{cor} \label{cor: 2.5} \textit{ Let $p \in (1,+\infty)$ and let $M \subset \mathbb{R}^{2}$ be a set with $\dim_{H}(M)=1$. Then ${\rm Cap}_p(M)>0$.}
\end{cor}
\begin{proof}[Proof of Corollary \ref{cor: 2.5}] If $p>2$ by Remark~\ref{rem: 2.6} and since $\hdim(M)=1$, ${\rm Cap}_{p}(M)>0$. Assume by contradiction that ${\rm Cap}_{p}(M)=0$ for some $p \in (1,2]$. Taking $\varepsilon= (p-1)/2$  so that $2-p+\varepsilon<1$, by Theorem~\ref{thm: 2.4} we get $\mathcal{H}^{2-p+\varepsilon}(M)=0$, but this leads to a contradiction with the fact that $\hdim(M)=1$.
\end{proof}
\begin{defn} Let the function $u$ be defined $p$-q.e. on $\mathbb{R}^{2}$ or on some open subset. Then $u$ is said to be $p$-quasi continuous if for every $\varepsilon>0$ there is an open set $A$ with ${\rm Cap}_{p}(A)<\varepsilon$ such that the restriction of $u$ to the complement of $A$ is continuous in the induced topology.
\end{defn}
\begin{theorem} Let $Y\subset \mathbb{R}^{2}$ be an open set and $p \in (1,+\infty)$. Then for each $u \in W^{1,p}(Y)$ there exists a $p$-quasi continuous function $\widetilde{u} \in W^{1,p}(Y)$, which is uniquely defined up to a set of ${\rm Cap}_{p}$-capacity zero and $u=\widetilde{u}$ a.e. in $Y$.
\end{theorem}

\begin{proof} Let  $x_{0} \in Y$ and let $\{\varphi_{i}: i \in \mathbb{N}, \, i \geq 1\}$ be a sequence of $C^{\infty}_{0}(Y)$ functions such that $\varphi_{i}=1$ in $Y_{i}=\{x \in Y: \dist(x, \partial Y)>\frac{1}{i}\} \cap B_{i}(x_{0})$. Observe that $u\varphi_{i}$ belongs to $W^{1,p}(\mathbb{R}^{2})$ and $u\varphi_{i}=u$ in $Y_{i}$. Then by \cite[Proposition 6.1.2]{Potential} there exist $p$-quasi continuous functions $v_{i} \in W^{1,p}(\mathbb{R}^{2})$ such that $v_{i}=u\varphi_{i}$ a.e. in $\mathbb{R}^{2}$. Notice that if $j>i$, then $v_{i}$ and $v_{j}$ coincide a.e. in $Y_{i}$, but this implies (see \cite[Theorem 6.1.4]{Potential}) that they coincide $p$-q.e. in $Y_{i}$. Now fix an arbitrary $\varepsilon>0$ and let $V_{i} \subset \mathbb{R}^{2}$ be such that $v_{i}$ restricted to $\mathbb{R}^{2} \backslash V_{i}$ is continuous and ${\rm Cap}_{p}(V_{i})<2^{-i}\varepsilon.$ Set $\widetilde{u}(x)=v_{i}(x)$ for every $x \in Y$, where $i \in \mathbb{N},\, i\geq 1$ is the smallest number with $x \in B_{i}(x_{0})$ and $\dist(x, \partial Y)>\frac{1}{i}$. We deduce that $\widetilde{u}=u$ a.e. in $Y$, $\widetilde{u}$ restricted to $Y\backslash \bigcup_{i}V_{i}$ is continuous and using \cite[Proposition 2.3.6]{Potential}, we get
\[
{\rm Cap}_{p}\biggl(\bigcup_{i} V_{i}\biggr) \leq \sum_{i}{\rm Cap}_{p}(V_{i}) \leq \varepsilon.
\]
Thus $\widetilde{u}$ is a $p$-quasi continuous representative for $u$, which by \cite[Theorem 6.1.4]{Potential} is uniquely defined up to a set of ${\rm Cap}_{p}$-capacity zero. This concludes the proof.
\end{proof}

\begin{rem} \label{rem: 2.8} Notice that $u \in W^{1,p}(\mathbb{R}^{2})$ belongs to $W^{1,p}_{0}(Y)$ if and only if its $p$-quasi continuous representative $\widetilde{u}$ vanishes $p$-q.e. on $\mathbb{R}^{2}\backslash Y $ (see \cite[Theorem 4]{BAGBY} and \cite[Lemma~4]{Hedberg}). Thus, if $Y^{\prime}$ is an open subset of $Y$ and $u \in W^{1,p}_{0}(Y)$ such that $\widetilde{u}=0$ $p$-q.e. on $Y\backslash Y^{\prime}$, then the restriction of $u$ to $ Y^{\prime}$ belongs to $W^{1,p}_{0}(Y^{\prime})$ and conversely, if we extend a function $u \in W^{1,p}_{0}(Y^{\prime})$ by zero in $Y\backslash Y^{\prime}$, then $u \in W^{1,p}_{0}(Y)$. Note that if $\Sigma \subset \overline{Y}$ and ${\rm Cap}_{p}(\s)=0$, then $W^{1,p}_{0}(Y) = W^{1,p}_{0}(Y \backslash \s)$. Indeed, $u \in W^{1,p}_{0}(Y)$ if and only if $u \in W^{1,p}(\mathbb{R}^{2})$ and $\widetilde{u}=0$ $p$-q.e. on $\mathbb{R}^{2} \backslash Y$ that is equivalent to say $u\in W^{1,p}(\mathbb{R}^{2})$ and $\widetilde{u}=0$ $p$-q.e. on $(\mathbb{R}^{2} \backslash Y) \cup \s$ or $u \in W^{1,p}_{0}(Y \backslash \s)$. In the sequel we shall always identify $u \in W^{1,p}(Y)$ with its $p$-quasi continuous representative $\widetilde{u}$.
\end{rem}

\begin{prop} \label{prop: 2.7} Let $D\subset  \mathbb{R}^{2}$ be a bounded extension domain and let $u \in W^{1,p}(D)$. Consider $E=\overline{D} \cap \{x: u(x)=0\}$. If ${\rm Cap}_p(E)>0$, then there exists $C=C(p,D)>0$ such that 
\[
\int_{D}|u|^{p}\diff x \leq C ({\rm Cap}_p(E))^{-1} \int_{D} |\nabla u|^{p}\diff x.
\]
\end{prop}
\begin{proof} For the proof we refer to \cite[Corollary 4.5.3, p. 195]{Ziemer}.
\end{proof}
 
Finally, since in this paper the notion of the Hausdorff distance is used, we recall the following well-known fact. If $X$ is a compact set in $\mathbb{R}^{2}$ and $(K_{n})_{n}$ is a sequence of compact subsets of $X$, then $K_{n}$ converge to $K$ in the Hausdorff distance if and only if the following two properties hold (this is also known as convergence in the sense of Kuratowski):
\begin{align*}
&\text{any}\,\ x \in K\,\ \text{is the limit of a sequence}\,\ (x_{n})_{n} \,\ \text{with}  \,\ x_{n}\in  K_{n}; \tag{P.1} \label{Property 1}\\
&\text{if} \,\ x_{n} \in K_{n}, \,\ \text{any limit point of}\,\ (x_{n})_{n} \,\ \text{belongs to} \,\ K. \tag{P.2} \label{Property 2}
\end{align*}


\subsection{Lower bound for capacities}
\begin{lemma} \label{lem: 2.10} Let $\Sigma$ be a set in $\mathbb{R}^{2}$ such that  
$\Sigma \cap \partial B_{r} \not = \emptyset \,\ \text{for every}\,\ r \in [1/2, 1].
$
If $p \in (1,2]$, then there is a constant $C=C(p)>0$ such that
\[
{\rm Cap}_p([0,1/2]\times \{0\}) \leq C {\rm Cap}_p(\Sigma).
\]
\end{lemma}

\begin{proof} Let us associate every point $x$ in $\Sigma \cap (\overline{B}_{1}\backslash B_{1/2})$ with the point $\Phi(x)=(|x|,0)$ in $[1/2, 1]\times \{0\}$. Since $\s \cap \partial B_{r} \not =\emptyset$ for every $r \in [1/2,1]$, we have that 
	\[
	[1/2,1]\times\{0\} =\Phi(\Sigma\cap (\overline{B}_{1}\backslash B_{1/2}))
	\]
	and hence
\begin{equation} \label{2.1}
{\rm Cap}_p([1/2, 1]\times \{0\}) = {\rm Cap}_p(\Phi(\Sigma\cap (\overline{B}_{1}\backslash B_{1/2}))).
\end{equation}
Since $\Phi$ is a 1-Lipschitz map, by the behavior of ${\rm Cap}_p$-capacity with respect to a Lipschitz map (see e.g. \cite[Theorem 5.2.1]{Potential}), there is a constant $C=C(p)>0$ such that 
\begin{equation} \label{2.2}
{\rm Cap}_p(\Phi(\Sigma\cap (\overline{B}_{1}\backslash B_{1/2}))) \leq C {\rm Cap}_p(\Sigma\cap (\overline{B}_{1}\backslash B_{1/2})).
\end{equation}
Thus, using (\ref{2.1}), (\ref{2.2}) and the facts that ${\rm Cap}_p$-capacity is an invariant under translations and is nondecreasing with respect to set inclusion, we recover the desired inequality.
\end{proof}

\begin{cor} \label{cor: 2.11} \textit { Let $\Sigma \subset \mathbb{R}^{ 2}$,\ $\xi \in \mathbb{R}^{ 2}$ and $r>0$ be such that $\Sigma \cap \partial B_{s}(\xi) \not = \emptyset$ for every $s\in [r,2r]$. Let $p \in (1,+\infty)$ and $u \in W^{1,p}(B_{2r}(\xi))$ satisfy $u=0\,\ p$-q.e. on $\Sigma \cap \overline{B}_{2r}(\xi)$. Then there is a constant $C>0$, which depends only on $p$, such that }
\[
\int_{B_{2r}(\xi)} |u|^{p} \diff x \leq C r^{p} \int_{B_{2r}(\xi)} |\nabla u|^{p} \diff x.
\]
\textit{Proof of Corollary \ref{cor: 2.11}.} Let us define $v(y)=u(\xi+2ry),\, y \in B_{1}$. Then $v \in W^{1,p}(B_{1})$,\ $v=0\,\ p$-q.e. on $(\frac{1}{2r}(\Sigma - \xi)) \cap \overline{B}_{1}$ and $(\frac{1}{2r}(\Sigma-\xi))\cap \partial B_{s}\not = \emptyset$ for every $s \in [1/2, 1]$. Next, if $p\in (1,2]$, by Lemma~\ref{lem: 2.10} and by Proposition~\ref{prop: 2.7}, for some $C=C(p)>0$ we get 
\[
\int_{B_{1}}|v|^{p}\diff y \leq C({\rm Cap}_p([0,1/2]\times \{0\}))^{-1} \int_{B_{1}} |\nabla v|^{p}\diff y. 
\]
If $p\in (2, +\infty)$, by Remark~\ref{rem: 2.6}, ${\rm Cap}_{p}((\frac{1}{2r}(\Sigma-\xi))\cap B_{1})\geq {\rm Cap}_{p}(\{(0,0)\})$. Next, using Proposition~\ref{prop: 2.7}, we get
\[
\int_{B_{1}}|v|^{p}\diff y \leq C({\rm Cap}_{p}(\{(0,0)\}))^{-1} \int_{B_{1}} |\nabla v|^{p}\diff y.
\]
Then, changing the variables, we recover the desired inequality.
\qed
\end{cor}


\subsection{Uniform boundedness of potentials}
 In this short subsection we establish a boundedness result, uniformly with respect to $\s$ for the potential $u_{\s}$. Let us emphasize that the estimate \eqref{2.3} will never be used in the sequel, but we find it interesting enough to keep it in the present paper. On the other hand, the estimate \eqref{2.4} will be used several times.  Let $\Om$ be a bounded open set in $ \mathbb{R}^{2}$ and let $p \in (1,+\infty)$. If $f \in L^{q_{0}}(\Om)$, where $q_{0}$ is the exponent defined in (\ref{(0.1)}) and $\s$ is a closed proper subset of $\overline{\Om}$, then it is well known that there is a unique function $u_{\s}$ that minimizes $E_{p}$ over $W^{1,p}_{0}(\Om\backslash \s)$. Let us extend $u_{\s}$ by zero outside $\Om \backslash \s$ to an element that belongs to $W^{1,p}(\mathbb{R}^{ 2})$. We shall use the same notation for this extension as for $u_{\s}$. 
\begin{prop} \label{prop: 2.12} Let $f \in L^{q_{0}}(\Om)$ with $q_{0}$  defined in (\ref{(0.1)}).  Then there is a constant $C>0$, possibly depending only on $p$ and $q_{0}$, such that
\begin{equation}
\int_{\Om}|\nabla u_{\Sigma}|^{p}\diff x \leq C |\Om|^{\alpha} \|f\|^{\beta}_{L^{q_{0}}(\Om)}, \label{2.4}
\end{equation}
where 
\begin{equation} \label{(2.5)}
(\alpha, \beta)=
\begin{cases}
(0, p^{\prime}) & \text{if} \,\ 1<p<2\\
(\frac{2}{q^{\prime}_{0}}, 2) & \text{if} \,\  p=2\\
(\frac{p-2}{2(p-1)}, p^{\prime}) & \text{if} \,\ 2<p<+\infty.
\end{cases}
\end{equation}
Moreover, if $f\in L^{q}(\Omega)$ with $q>\frac{2}{p}$ if $p \in (1,2]$ and $q=1$ if $2<p<+\infty$, then there is a constant $C=C(p,q,\|f\|_{L^{q}(\Omega)},|\Om|)>0$ such that 
\begin{equation}
\|u_{\s}\|_{L^{\infty}(\mathbb{R}^{2})} \leq C. \label{2.3}
\end{equation} 
\end{prop}
\begin{proof} The estimate (\ref{2.3}) follows from Lemma~\ref{lem: A.2} applied for $U=\Omega\backslash \Sigma$ and from the fact that the constant $C$ in (\ref{A.5}) is increasing with respect to $|U|$. Now let $f \in L^{q_{0}}(\Om)$. Using $u_{\s}$ as the test function in (\ref{(1.2)}), we get
\begin{align*}
\int_{\Om} |\nabla u_{\s}|^{p}\diff x &= \int_{\Om} fu_{\s} \diff x\\
& \leq \|f\|_{L^{q_{0}}(\Om)} \|u_{\s}\|_{L^{q_{0}^{\prime}}(\Om)}. \numberthis \label{2.5}
\end{align*}
Next, recalling that by the Sobolev inequalities (see \cite[Theorem 7.10]{PDE}) there exists $C=C(p)>0$ such that 
\begin{equation}  \label{2.6}
\|u_{\s}\|_{L^{q_{0}^{\prime}}(\Om)} \leq 
\begin{cases}                                
C \|\nabla u_{\s}\|_{L^{p}(\Om)} & \text{if} \,\ 1<p<2 \\ 
C |\Omega|^{\frac{1}{2}-\frac{1}{p}} \|\nabla u_{\s}\|_{L^{p}(\Om)}&  \text{if} \,\ 2<p<+\infty
\end{cases}
\end{equation}
and using (\ref{2.5}), we recover (\ref{2.4}) when $p\neq 2$. If $p=2$ and $q_{0}\in (1,2]$, setting $\varepsilon=\frac{4}{q^{\prime}_{0}+2}$ (note that $\frac{1}{q^{\prime}_{0}}=\frac{1}{2-\varepsilon}-\frac{1}{2}$ and $2-\varepsilon\geq 1$), we get
\begin{align*}
\|u_{\s}\|_{L^{q^{\prime}_{0}}(\Omega)} &\leq C \|\nabla u_{\s}\|_{L^{2-\varepsilon}(\Omega)} \,\ (\text{by the Sobolev inequality})\\
& \leq C|\Om|^{\frac{1}{q^{\prime}_{0}}} \|\nabla u_{\s}\|_{L^{2}(\Om)} \,\ (\text{by H\"{o}lder's inequality}),   \numberthis \label{2.7}
\end{align*}
where $C=C(q_{0})>0$. Using (\ref{2.7}) together with (\ref{2.5}), we obtain (\ref{2.4}) in the case when $p=2$ and $q_{0}\in (1,2]$. Finally, assume that $p=2$ and $q_{0}>2$. We observe that $1\leq q_{0}^{\prime}<2$. Then, using H\"{o}lder's inequality and (\ref{2.7}), we get
\[
\|u_{\Sigma}\|_{L^{q^{\prime}_{0}}(\Omega)}\leq |\Omega|^{\frac{1}{q^{\prime}_{0}}-\frac{1}{2}}\|u_{\Sigma}\|_{L^{2}(\Omega)}\leq C|\Omega|^{\frac{1}{q^{\prime}_{0}}-\frac{1}{2}}|\Omega|^{\frac{1}{2}}\|\nabla u_{\Sigma}\|_{L^{2}(\Omega)}=C|\Omega|^{\frac{1}{q^{\prime}_{0}}}\|\nabla u_{\Sigma}\|_{L^{2}(\Omega)}.
\]
The last estimate, together with (\ref{2.5}), yields (\ref{2.4}) in the case when $p=2$ and $q_{0}>2$. This completes the proof of Proposition~\ref{prop: 2.12}.
\end{proof}


\subsection{Existence}
\begin{theorem} \label{thm: 2.13}  Let $\Om$ be an open and bounded set in $\mathbb{R}^{2}$ and $p\in (1,+\infty)$, and let $f \in L^{q_{0}}(\Om)$, with $q_{0}$ defined in (\ref{(0.1)}). Let $({\s}_{n})_{n}$ be a sequence of closed connected proper subsets of $\overline{\Om}$, converging to a closed connected proper subset $\s$ of $\overline{\Om}$ with respect to the Hausdorff distance. Then
\[
u_{\s_{n}} \underset{n \to +\infty}{\longrightarrow} u_{\s}\,\ \text{strongly in}\,\ W^{1,p}(\Om).
\]
\end{theorem}
\begin{proof} For a proof, see \cite{sverak} for the case $p=2$ and \cite{Bucur} for the general case.
	\end{proof}
\begin{rem} \label{rem: 2.14} As in \cite{Bucur} we recall that a sequence $(\Om_{n})_{n}$ of open subsets of a fixed ball $B$ $\gamma_{p}$-converges to $\Om$ if for any $f \in W^{-1, p^{\prime}}(B)$,  where $W^{-1, p^{\prime}}(B)$ is the dual space of $W^{1,p}_{0}(B)$, the solutions of the Dirichlet problem
\[
-\Delta_{p} u_{n} = f \,\ \text{in} \,\ {\Om}_{n}, \,\ u_{n} \in W^{1,p}_{0}({\Om}_{n})
\]
converge strongly in $W^{1,p}_{0}(B)$, as $n\to +\infty$, to the solution of the corresponding problem in $\Om$. It can be shown that the $\gamma_{p}$-convergence is equivalent to the convergence in the sense of Mosco of the associated Sobolev spaces (see \cite{Bucur}).
\end{rem}

\begin{prop} \label{prop: 2.15}
Problem~\ref{problemMain} admits a minimizer.
\end{prop}
\begin{proof} Let $({\s}_{n})_{n}$ be a minimizing sequence for Problem~\ref{problemMain}. We can assume that $\Sigma_{n}\neq \emptyset$ and $C_{p}(\Sigma_{n})+\lambda \mathcal{H}^{1}(\Sigma_{n})\leq C_{p}(\emptyset)$ for all $n\in \mathbb{N}$ or at least for a subsequence still denoted by $n$, because otherwise the empty set would be a minimizer. Then, using Blaschke's theorem (see \cite[Theorem 6.1]{APD}), we can find a compact connected proper subset $\s$ of $\overline{\Om}$ such that up to a subsequence, still denoted by the same index, ${\s}_{n}$ converges to $\s$ with respect to the Hausdorff distance as $n \to +\infty$. Then, by Theorem~\ref{thm: 2.13}, $u_{{\s_{n}}}$ converges to $u_{\s}$ strongly in $W^{1,p}_{0}(\Om)$ and thanks to the lower semicontinuity of $\mathcal{H}^{1}$ with respect to the topology generated by the Hausdorff distance, we deduce that $\s$ is a minimizer of Problem~\ref{problemMain}.
\end{proof}

Before starting the study of the regularity and qualitative properties satisfied by a minimizer, we verify that, at least for some range of  values of $\lambda$, a minimizer $\Sigma$ is actually not trivial. This is the purpose of the following proposition.

\begin{prop} \label{prop: 2.16}  Let $\Om \subset \mathbb{R}^{2}$ be open and bounded. Let $p \in (1,+\infty)$ and $f \in L^{q_{0}}(\Om)$, $f\neq 0$, with $q_{0}$ defined in (\ref{(0.1)}).  Then there exists  $\lambda_{0} =\lambda_{0} (p, f, \Om)>0$ such that   if   $\lambda \in (0, \lambda_{0}]$, then every solution $\s$ of Problem~\ref{problemMain} has positive $\mathcal{H}^{1}$-measure.
\end{prop}

\begin{proof} Case 1:\,\ $p \in (1,2]$. By Theorem~\ref{thm: 2.4}, for all point $x \in \overline{\Om}$ one has ${\rm Cap}_{p}(\{x\})=0$ and this implies that $W^{1,p}_{0}(\Omega)=W^{1,p}_{0}(\Omega\backslash \{x\})$ (see Remark~\ref{rem: 2.8}). We claim that there is a closed connected set $\s_0 \subset \overline{\Om}$ such that $0<\mathcal{H}^{1}(\s_0) <+\infty$ and $C_{p}(\s_0)<C_{p}(\emptyset)$. Otherwise, for any closed connected set  $\s$, since the functional $C_{p}(\cdot)$ is nonincreasing with respect to set inclusion, we would have that $C_{p}(\s)=C_{p}(\emptyset)$, that thanks to the uniqueness of $u_{\emptyset}$ and to the fact that $u_{\s} \in W^{1,p}_{0}(\Om)$, implies that $u_{\s}= u_{\emptyset}$. Thus, $u_{\emptyset}=u_{\s}=0$ $p$-q.e. on $\s$ and varying $\s$ in $\overline{\Om}$ we deduce that $u_{\emptyset}=0$ as an element of $W^{1,p}_{0}(\Omega)$. Then, by using the weak formulation of the $p$-Poisson equation which defines $u_{\emptyset}$, we get
\[
0=\int_{\Om}|\nabla u_{\emptyset}|^{p-2}\nabla u_{\emptyset} \nabla \varphi \diff y = \int_{\Omega}f \varphi \diff y \,\ \text{for all} \,\ \varphi \in C^{\infty}_{0}(\Omega),
\]
but this implies that $f=0$ and leads to a contradiction. Thus, taking $\lambda_{0}=\frac{C_{p}(\emptyset)-C_{p}(\s_0)}{2\mathcal{H}^{1}(\s_0)}$, for any $\lambda \in (0, \lambda_{0}]$ we get $C_{p}(\s_0)+\lambda \mathcal{H}^{1}(\s_0)<C_{p}(\emptyset)$ and therefore each minimizer of Problem~1.1 defined for such $\lambda$ should have positive $\mathcal{H}^{1}$-measure.

Case 2:\,\ $2<p<+\infty$. In this case the empty set will not be a minimizer of Problem~\ref{problemMain}. In fact, assume by contradiction that there exists $\lambda>0$ such that the empty set is a minimizer of Problem~\ref{problemMain}. Then for an arbitrary point $x_{0} \in \Om$, we have that $C_{p}(\{x_{0}\})=C_{p}(\emptyset)$, since $\emptyset$ is a minimizer and $C_{p}(\cdot)$ is nonincreasing. But by the uniqueness of $u_{\emptyset}$ and since $u_{\{x_{0}\}} \in W^{1,p}_{0}(\Om)$, the fact that $C_{p}(\{x_{0}\})=C_{p}(\emptyset)$ implies that $u_{\{x_{0}\}}=u_{\emptyset}$. Recalling that by the embedding theorem of Morrey, $W^{1,p}_{0}(\Om) \subset C^{0, \alpha}(\Omega)$, where $\alpha=1-2/p$, we get $u_{\{x_{0}\}}(x_{0})=u_{\emptyset}(x_{0})=0$. Varying $x_{0}$ in $\Omega$ we deduce that $u_{\emptyset}=0$, that, as in Case 1, contradicts the fact that $f \neq 0$ in $L^{q_{0}}(\Om)$.  Thus any minimizer $\s$ contains at least one point.

Next, let us consider the minimization problem $$\widetilde{(P)} \,\ \min_{x \in \overline{\Om}} C_{p}(\{x\}).$$ 

It is easy to check that a minimizer for $\widetilde{(P)}$ exists.  Indeed, taking a minimizing sequence $(x_{n})_{n}$, since $\overline{\Om}$ is compact, there exists $\overline{x} \in \overline{\Om}$ such that $x_{n}\to \overline{x}$ and then, by Theorem~\ref{thm: 2.13}, $C_{p}(\overline{x})= \min_{x \in \overline{\Om}} C_{p}(\{x\})$. We claim that $\overline{x} \in \Om$ and, actually, it belongs to a connected open component $U$ of $\Omega$ such that $\partial U \subset \partial \Om$ and $\restr f U \neq 0$ in $L^{q_{0}}(U)$. Indeed, if $\overline{x}$ would lie on $\partial \Om$, then $C_{p}(\{\overline{x}\})=C_{p}(\emptyset)$ and since $\overline{x}$ is a minimizer for $\widetilde{(P)}$ and $C_{p}(\cdot)$ is nonincreasing, $C_{p}(\emptyset)=C_{p}(\{x_{0}\})$ for all $x_{0} \in \Om$ that as before would contradict the fact that $f \neq 0$ in $L^{q_{0}}(\Om)$. Now, assume that $\restr f U=0$ in $L^{q_{0}}(U)$. Since $U$ is an open connected component of $\Om$, $\partial U \subset \partial \Om$, we have that $u_{\emptyset} \in W^{1,p}_{0}(U)$ and using the weak formulation of the $p$-Poisson equation which defines $u_{\emptyset}$, we get
\[
\int_{U}|\nabla u_{\emptyset}|^{p}\diff y = \int_{U}f u_{\emptyset} \diff y =0
\]
and hence $u_{\emptyset}= 0$ on $U$. Thus, $u_{\emptyset} \in W^{1,p}_{0}(\Om \backslash \{\overline{x}\})$ and since $C_{p}(\{\overline{x}\}) \leq C_{p}(\emptyset)$, we deduce that $C_{p}(\{\overline{x}\})=C_{p}(\emptyset)$, but this, as before, contradicts the fact that $f \neq 0$ in $L^{q_{0}}(\Om)$. Finally, we claim that there exists a closed connected set $\s_0 \subset \overline{U}$ such that $\overline{x} \in \s_0$, $0<\mathcal{H}^{1}(\s_0)<+\infty$ and $C_{p}(\s_0)<C_{p}(\{\overline{x}\})$. Because otherwise, we would have for all such $\s$ that $C_{p}(\s)=C_{p}(\{\overline{x}\})$ that would lead to the fact that $u_{\{\overline{x}\}}=0$ $p$-q.e. on $\s$ and since U is arcwise connected, because open and connected, varying $\s$ in $U$, one would obtain $u_{\{\overline{x}\}}=0$ in $U$, but this would contradict the fact that $\restr f U \neq 0 $ in $L^{q_{0}}(U)$. Thus, taking $\lambda_{0}= \frac{C_{p}(\{\overline{x}\}) - C_{p}(\s_0)}{2\mathcal{H}^{1}(\s_0)}$, for any $\lambda \in (0, \lambda_{0}]$ we get $C_{p}(\s_0) + \lambda \mathcal{H}^{1}(\s_0)<C_{p}(\{\overline{x}\})$. This shows that each minimizer of Problem~\ref{problemMain} defined for such $\lambda$ should have positive $\mathcal{H}^{1}$-measure.
\end{proof}



\subsection{Dual formulation}
\begin{prop} \label{prop: 3.1} Let $\Om \subset\mathbb{R}^{2}$ be open and bounded. Let $p\in (1,+\infty)$ and $f \in L^{q_{0}}(\Om)$ with $q_{0}$ defined in (\ref{(0.1)}). Then Problem~\ref{problemMain}  is equivalent to the minimization problem 
\begin{equation} \label{3.1}
(\mathcal{P}^{*}) \qquad \min_{(\sigma, \Sigma) \in \mathcal{B}} \frac{1}{p^{\prime}} \displaystyle\int_{\Omega} |\sigma|^{p^{\prime}} \diff x +  \lambda \mathcal{H}^{1}(\Sigma)
\end{equation}
where 
\begin{align*}
\mathcal{B}:=\{ (\sigma, \Sigma): \Sigma \in \mathcal{K}(\Om)\,\ \text{and}\,\ \sigma \in L^{p^{\prime}}(\Omega; \mathbb{R}^{ 2}), -div(\sigma)=f \,\ \text{in} \,\ \mathcal{D}^{\prime}(\Omega\backslash \Sigma) \}
\end{align*}
in the sense that the minimum value of the latter is equal to that of Problem~1.1, and once $(\overline{\sigma},\overline{\Sigma})\in \mathcal{B}$ is a minimizer for $(\mathcal{P}^{*})$, then $\overline{\Sigma}$ solves Problem~1.1. Moreover, for a given closed proper subset $\s$ of  $\overline{\Om}$, the choice $\sigma= |\nabla u_{\s}|^{p-2}\nabla u_{\s}$ solves
\[
\min_{\sigma \in L^{p^{\prime}}(\Omega; \mathbb{R}^{2})}\biggl\{ \frac{1}{p^{\prime}} \int_{\Omega} |\sigma|^{p^{\prime}} \diff x : - div(\sigma)=f \,\ \text{in}\,\ \mathcal{D}^{\prime}(\Om\backslash \s) \biggr\}.
\]
\end{prop}
\begin{proof} The proof is the direct consequence of Lemma A.3 and the uniqueness of $u_{\Sigma}$ and the minimizer $\sigma$.
\end{proof}

\section{Ahlfors regularity}
We recall that a set $\s \subset \mathbb{R}^{2}$ is said to be Ahlfors regular  of dimension 1, if there exist some constants $c>0$, $r_{0}>0$ and $C>0$ such that for every $r \in(0, r_{0})$ and for every $x \in \s$ the following holds
\[
cr\leq \mathcal{H}^{1}(\s \cap B_{r}(x)) \leq Cr.\numberthis \label{5.1}
\]
 The notion of Ahlfors regularity is a quantitative and scale-invariant version of having Hausdorff dimension one. It is known that Ahlfors regularity of a closed connected set $\s$ implies \textit{uniform rectifiability} of $\s$, which provides several useful analytical properties of $\s$, see for example \cite{David-Semmes}.

Note that for a closed connected nonempty set $\s$ the lower bound in (\ref{5.1}) is trivial: indeed, for all $x \in \s$ and for all $r\in (0,\diam(\s)/2)$ we have: $\s \cap \partial B_{r}(x) \not = \emptyset$, and then
\begin{equation}
\mathcal{H}^{1}(\s \cap B_{r}(x)) \geq r.  \label{A}
\end{equation}
In order to prove the Ahlfors regularity for such $\s$ it suffices to show that there is $r_{0}>0$, independent of $x$, such that the upper bound in (\ref{5.1}) holds for all $x \in \s$ and for all $r \in (0, r_{0})$.

Before starting to prove the Ahlfors regularity of $\s$, let us focus  on the following basic question: to which class $L^{q}(U)$ should the function $f$ belong so that the solution $u$ of the Dirichlet problem
\[
-\Delta_{p}u=f  \,\ \text{in} \,\ U  \subset  r[-a,a]\times [-b, b], \,\ u \in W^{1,p}_{0}(U)
\] 
satisfies $\int_{U}|\nabla u|^{p}\diff x \leq C r $, where  $C=C(a,b, p, q_{0},q,\|f\|_{q}) $ with $q_{0}$ defined (\ref{(0.1)})? Using Proposition~\ref{prop: 2.12}, we can state that it is enough to take $q=\frac{2p}{2p-1}=(2p)^{\prime}$, as explained in the following lemma, which will also appear in the proof of Theorem~\ref{thm: 5.3}. 

\begin{lemma} \label{lemma 3.3}Let $a,b,r>0$ and $U \subset r[-a,a]\times [-b, b]$ be an open set. Let $p \in (1,+\infty)$ and $f \in L^{(2p)^{\prime}}(U)$, and let $u$ be the weak solution of the Dirichlet problem:
\begin{equation*}
-\Delta_{p} u = f \,\ \text{in} \,\ U,\,\ u \in W^{1,p}_{0}(U)
\end{equation*}
which means that 
\begin{equation}
\int_{U}|\nabla u|^{p-2}\nabla u \nabla \varphi \diff x= \int_{U}f\varphi \diff x \,\ \text{for all} \,\ \varphi \in W^{1,p}_{0}(U). \label{(3.5)}
\end{equation}
Then there exists a constant $C=C(a, b, p, q_{0}, \|f\|_{(2p)^{\prime}})>0$, where $q_{0}$ is defined in (\ref{(0.1)}), such that 
\begin{equation}
\int_{U}|\nabla u|^{p}\diff x \leq C r. \label{(3.6)}
\end{equation}
\end{lemma}

\begin{proof} Assume that $f \in L^{q}(U)$ with $q\geq q_{0}$, where $q_{0}$ is defined in (\ref{(0.1)}). Then $u$ is well defined. By (\ref{2.4}) with $u_{\s}$ replaced by $u$ and $\Om$ by $U$, there exists $C=C(p, q_{0})>0$ such that
\begin{equation*} 
\int_{U}|\nabla u|^{p}\diff x \leq C |U|^{\alpha}\|f\|_{L^{q_{0}}(U)}^{\beta},
\end{equation*}
where $(\alpha, \beta)$ is defined in (\ref{(2.5)}). Using H\"{o}lder's inequality and  the fact that $U$ is a subset of $r[-a,a]\times[-b, b]$, we get
\begin{equation*}
\int_{U}|\nabla u|^{p}\diff x \leq C (4 a b r^{2})^{\alpha + \beta\bigl(\frac{1}{q_{0}}-\frac{1}{q}\bigr)} \|f\|^{\beta}_{L^{q}(U)}.
\end{equation*}
Thus, in order for the estimate (\ref{(3.6)}) to hold, one should take the exponent $q$ such that $2(\alpha+\beta(\frac{1}{q_{0}}-\frac{1}{q}))=1$. Having carefully performed the calculations, one gets $q=\frac{2p}{2p-1}$.
\end{proof}

To prove that $\s$ is Ahlfors regular ``near''  $\partial \Om$, we shall assume some Lipschitz regularity on $\Om$. Here is a precise definition.

\begin{defn} \label{defn 3.2} A bounded domain $\Omega \subset \mathbb{R}^{2}$ and its boundary $\partial \Om$ are locally Lipschitz if there exists a radius $r_{\partial \Om}$ and a constant $\delta>0$ such that for every point $x \in \partial \Omega$ and every radius $s \in (0, r_{\partial \Om})$ up to a rotation of coordinates, it holds
\begin{align*}
\Omega \cap B_{s}(x) =\{(y_{1}, y_{2}) \in \mathbb{R}^{2}: y_{2}>\varphi(y_{1})\}\cap B_{s}(x)
\end{align*}
for some Lipschitz function $\varphi: \mathbb{R} \to \mathbb{R}$ satisfying $\|\nabla \varphi\|_{L^{\infty}(\mathbb{R})}\leq \delta$.
\end{defn}

One deduces that for every radius $s \in (0, r_{\partial \Om})$ in the above definition the set $\partial \Om \cap \overline{B}_{s}(x)$ up to a rotation of coordinates is contained in the double cone
\[
K_{\delta}=\{y \in \mathbb{R}^{2}: y=0 \,\ \text{or}\,\ \ang(y, \ox) \in [0, \arctan(\delta)]\cup [\pi- \arctan(\delta), \pi]\}.
\]

\begin{theorem} \label{thm: 5.3} Let $\Om \subset \mathbb{R}^{2}$ be a bounded domain with locally Lipschitz boundary (see Definition~\ref{defn 3.2}), $p\in (1,+\infty)$, and  $f\in L^{\frac{2p}{2p-1}}(\Om)$. Let  $\Sigma$ be a solution of Problem~\ref{problemMain} with $\diam(\s)>0$. Then $\s$   is Ahlfors regular.
\end{theorem}

\begin{rem} By Proposition~\ref{prop: 2.16}  we know that the assumption $\diam(\s)>0$ is fulfilled  at least when $\lambda \in (0, \lambda_{0}]$, where $\lambda_0=\lambda_0(p, f,\Omega)$.
\end{rem}

\begin{rem} \label{remark 3.1} Every closed and connected set $\Sigma \subset \mathbb{R}^{2}$ satisfying $\mathcal{H}^{1}(\Sigma)<+\infty$ is arcwise connected (see, for instance, \cite[Corollary 30.2, p. 186]{David}).
\end{rem}

\begin{proof}[Proof of Theorem \ref{thm: 5.3}] Let $r_{\partial \Om}$ and $\delta$ be positive constants as in Definition~\ref{defn 3.2}. We set $$r_{0}=\min\{r_{\partial \Om}/3\sqrt{1+\delta^{2}},\diam(\s)/2\}$$ and let $x \in \s$ and $r\in (0,r_{0})$. Consider the next two cases.\\
Case 1: $B_{r}(x) \subset \Om$. As mentioned in Remark~\ref{remark 3.1}, $\s$ is arcwise connected. Then the set
\begin{equation}
\Sigma_{r}=(\s \backslash B_{r}(x)) \cup \partial B_{r}(x) \label{3.6}
\end{equation}
is a closed arcwise connected proper subset of $\overline{\Omega}$, that is a competitor for $\s$.  Let us now recall that $(\sigma, \Sigma)=(|\nabla u_{\Sigma}|^{p-2}\nabla u_{\s}, \s)$ is a minimizer for  problem $(\mathcal{P}^{*})$ in the formulation~(\ref{3.1}). Consider the pair $(\sigma_{r}, \Sigma_{r})$, where
\begin{equation*}
\sigma_{r}=
\begin{cases}
|\nabla u_{\Sigma}|^{p-2} \nabla u_{\Sigma} \,\ \text{in $\Omega\backslash (\Sigma_{r} \cup B_{r}(x))$}, \\
|\nabla u|^{p-2} \nabla u \,\ \text{in} \,\ B_{r}(x),\,\ u\in W^{1,p}_{0}(B_{r}(x)) \,\ \text{solves} -\Delta_{p} u= f.
\end{cases}
\end{equation*}
Notice that for any function $\varphi \in C^{\infty}_{0}(\Omega \backslash \s_{r})$ the support of $\varphi$ is contained in the union of two disjoint open sets $\Omega \backslash (\Sigma_{r} \cup B_{r}(x))$ and $B_{r}(x)$, and then, we can represent $\varphi$ as $\varphi=\varphi_{1}+\varphi_{2}$ with $\varphi_{1}\in C^{\infty}_{0}(\Omega \backslash (\Sigma_{r} \cup B_{r}(x)))$ and $\varphi_{2} \in C^{\infty}_{0}(B_{r}(x))$ which are test functions for the weak formulations of the $p$-Poisson equations that define $u_{\s}$ and $u$ respectively. Thus, we deduce that
\begin{equation*}
\langle -div(\sigma_{r}), \varphi \rangle= \langle |\nabla u_{\s}|^{p-2} \nabla u_{\s}, \nabla \varphi_{1}\rangle + \langle |\nabla u|^{p-2}\nabla u, \nabla \varphi_{2} \rangle = \langle f, \varphi_{1} \rangle + \langle f, \varphi_{2} \rangle = \langle f, \varphi \rangle.
\end{equation*}
Therefore $(\sigma_{r}, \Sigma_{r})$ is a competitor for $(\sigma, \Sigma)$. By the optimality of $(\sigma, \Sigma)$,
\begin{equation*}
\begin{split}
\frac{1}{p^{\prime}} \int_{\Omega}|\nabla u_{\Sigma}|^{p}\diff y + \lambda \mathcal{H}^{1} (\Sigma)& \leq \frac{1}{p^{\prime}} \int_{\Omega} |\sigma_{r}|^{p^{\prime}}\diff y + \lambda \mathcal{H}^{1} (\Sigma_{r}) \\ &\leq \frac{1}{p^{\prime}} \int_{\Omega\backslash \overline{B}_{r}(x)} |\nabla u_{\Sigma}|^{p} \diff y + \frac{1}{p^{\prime}} \int_{B_{r}(x)} |\nabla u|^{p} \diff y \\ &  \,\ + \lambda \mathcal{H}^{1} (\Sigma\backslash B_{r}(x)) + \lambda\mathcal{H}^{1}(\partial B_{r}(x)).
\end{split}
\end{equation*}
Then
\begin{align*} \label{5.2}
\lambda \mathcal{H}^{1} (\Sigma\cap B_{r}(x)) \leq  2\lambda\pi r + \frac{1}{p^{\prime}} \int_{B_{r}(x)} |\nabla u|^{p} \diff y.
\end{align*}
So, recalling that by Lemma~\ref{lemma 3.3} one has $$\int_{B_{r}(x)}|\nabla u|^{p}\diff y \leq \widetilde{C}r,$$ where $\widetilde{C}=\widetilde{C}(p, q_{0}, \|f\|_{(2p)^{\prime}})>0$ with $q_{0}$ defined in (\ref{(0.1)}), we deduce that 
\begin{equation}
\mathcal{H}^{1}(\s \cap B_{r}(x)) \leq Cr, \label{3.7}
\end{equation}
where $C=C(p, q_{0}, \|f\|_{(2p)^{\prime}}, \lambda)>0$. 
\\
Case 2: $B_{r}(x) \cap \partial \Om \neq \emptyset$. In this case we use the fact that locally $\partial \Om$ is a graph of a $\delta$-Lipschitz function. Let $x_{\partial \Om}$ be an arbitrary projection of $x$ to $\partial \Om$. Recalling that $r<r_{\partial \Om}/3\sqrt{1+\delta^{2}}$, up to a rotation of coordinates one has
\begin{equation} \label{(3.7)}
\Om \cap B_{3\sqrt{1+\delta^{2}}r}(x_{\partial \Om})=\{(y_{1}, y_{2}) \in \mathbb{R}^{2}: y_{2}> \varphi(y_{1})\}\cap B_{3\sqrt{1+\delta^{2}}r}(x_{\partial \Om})
\end{equation}
for some Lipschitz function $\varphi: \mathbb{R} \to \mathbb{R}$ satisfying $\|\nabla \varphi\|_{L^{\infty}(\mathbb{R})}\leq \delta$. In addition, the set $\partial \Om \cap \overline{B}_{3\sqrt{1+\delta^{2}}r}(x_{\partial \Om})$ is contained in the double cone
\[
K_{\delta}=\{y \in \mathbb{R}^{2}: y=0 \,\ \text{or} \,\ \ang(y, \ox) \in [0, \arctan(\delta)]\cup [\pi-\arctan(\delta), \pi]\}.
\]
Notice that the ball $B_{2r}(x_{\partial \Om})$ in the $(y_{1}, y_{2})$ coordinates is represented as $B_{2r}(0)$. Let us define $\xi^{-}=\varphi(-2r)$ and $\xi^{+}=\varphi(2r)$. Now we need to distinguish between two further cases.\\
\text{Case 2a:} $\delta\in (0,1]$. Define the points $h^{-}$ and $h^{+}$ by $h^{-}=2r(\oy-\ox)$ and $h^{+}=2r(\ox+\oy)$.\\
\text{Case 2b:} $\delta>1$. Define $h^{-}$ and $h^{+}$ by $h^{-}=2r(\delta\oy - \ox)$ and $h^{+}=2r(\ox+\delta \oy)$.

At this point observe that the open rectangle $\mathcal{R}$ with vertices $-h^{+}, h^{-}, h^{+}$ and $-h^{-}$ contains the ball $B_{2r}(0)$. Furthermore, by (\ref{(3.7)}) and since $\partial \Om \cap \overline{B}_{3\sqrt{1+\delta^{2}}r}(x_{\partial \Om})\subset K_{\delta}$,  the union of the segments $$\gamma_{r}=[\xi^{-}, h^{-}]\cup [h^{-}, h^{+}] \cup [\xi^{+}, h^{+}]$$ is a curve lying in $\overline{\Omega}$ such that $\gamma_{r} \cup (\partial \Om \cap \mathcal{R})$ is a closed simple curve (i.e., homeomorphic image of $\mathcal{S}^{1}$ into $\mathbb{R}^{2}$) lying in $\overline{\Omega}$ and $\partial(\mathcal{R}\cap \Omega)=\gamma_{r} \cup (\partial \Om \cap \mathcal{R})$. Thus, it is clear that 
\[
\Sigma_{r}=(\s\backslash \mathcal{R}) \cup \gamma_{r}
\]
is closed arcwise connected proper subset of $\overline{\Omega}$, namely, it is a competitor for $\s$. Let us now recall that $(\sigma, \Sigma)=(|\nabla u_{\Sigma}|^{p-2}\nabla u_{\s}, \s)$ is a minimizer for the problem $(\mathcal{P}^{*})$ in the formulation (\ref{3.1}). Then, consider the pair $(\sigma_{r}, \Sigma_{r})$, where 
\begin{equation*}
\sigma_{r}=
\begin{cases}
|\nabla u_{\Sigma}|^{p-2} \nabla u_{\Sigma} \,\ \text{in $\Omega\backslash (\Sigma_{r} \cup \mathcal{R})$}, \\
|\nabla u|^{p-2} \nabla u \,\ \text{in} \,\ \mathcal{R} \cap \Omega,\,\ u\in W^{1,p}_{0}(\mathcal{R}\cap \Omega) \,\ \text{solves} -\Delta_{p} u= f.
\end{cases}
\end{equation*}
Observe that if $\varphi \in C^{\infty}_{0}(\Om \backslash \s_{r})$, then because $\gamma_{r} \cup (\partial \Omega \cap \mathcal{R})$ is a closed simple curve, the support of $\varphi$ is contained in the union of two open disjoint sets $\Omega \backslash (\Sigma_{r}\cup \mathcal{R})$ and  $\mathcal{R}\cap \Omega$, and then we can write $\varphi=\varphi_{1}+\varphi_{2},$ where $\varphi_{1} \in C^{\infty}_{0}(\Omega \backslash (\Sigma_{r}\cup \mathcal{R}))$ and $\varphi_{2} \in C^{\infty}_{0}(\mathcal{R}\cap \Omega)$. Thus, we have that
\[
\langle-div(\sigma_{r}), \varphi\rangle=\langle |\nabla u_{\s}|^{p-2} \nabla u_{\s}, \nabla \varphi_{1} \rangle + \langle |\nabla u|^{p-2}\nabla u, \nabla \varphi_{2} \rangle =\langle f, \varphi_{1}\rangle + \langle f, \varphi_{2} \rangle = \langle f , \varphi \rangle,
\]
where we have used that $\varphi_{1}$ and $\varphi_{2}$ are test functions for the weak formulations of the $p$-Poisson equations that define $u_{\s}$ and $u$ respectively. Therefore $(\sigma_{r}, \Sigma_{r})$ is a competitor for the minimizer $(\sigma, \Sigma)$. Moreover, since $\partial \Om \cap B_{r}(x) \neq \emptyset$, one has $|x-x_{\partial \Om}|<r$ and then $B_{r}(x)\subset B_{2r}(x_{\partial \Om})\subset \mathcal{R}$. Thus, by the optimality of $(\sigma, \Sigma)$,
\begin{equation*}
\begin{split}
\frac{1}{p^{\prime}} \int_{\Omega}|\nabla u_{\Sigma}|^{p}\diff z + \lambda \mathcal{H}^{1} (\Sigma)& \leq \frac{1}{p^{\prime}} \int_{\Omega} |\sigma_{r}|^{p^{\prime}}\diff z + \lambda \mathcal{H}^{1} (\Sigma_{r}) \\& \leq \frac{1}{p^{\prime}} \int_{\Omega\backslash \overline{\mathcal{R}}} |\nabla u_{\Sigma}|^{p} \diff z + \frac{1}{p^{\prime}} \int_{\Omega\cap \mathcal{R}} |\nabla u|^{p} \diff z \\ &  \,\ + \lambda \mathcal{H}^{1} (\Sigma\backslash B_{r}(x)) + \lambda\mathcal{H}^{1}(\gamma_{r}),
\end{split}
\end{equation*}
where we have used that $B_{r}(x) \subset \mathcal{R}$. Notice that $\mathcal{H}^{1}(\gamma_{r}) \leq 4r+8\max\{1, \delta\}r$. Then we deduce that
\begin{equation*}
\lambda\mathcal{H}^{1}(\s \cap B_{r}(x)) \leq 4\lambda r+ 8 \lambda \max\{1, \delta\}r + \frac{1}{p^{\prime}}\int_{\Omega\cap \mathcal{R}}|\nabla u|^{p}\diff z
\end{equation*}
and recalling that by Lemma~\ref{lemma 3.3}, $\int_{\Omega\cap \mathcal{R}}|\nabla u|^{p}\diff z \leq \widetilde{C} r$ for some positive constant $\widetilde{C}$ depending only on $\delta, p, q_{0}, \|f\|_{(2p)^{\prime}}$, we finally get the estimate
\[
\mathcal{H}^{1}(\s \cap B_{r}(x)) \leq Cr
\]
where $C=C(\delta, p, q_{0}, \|f\|_{(2p)^{\prime}}, \lambda)>0$. This together with (\ref{A}) and (\ref{3.7}) implies the Ahlfors regularity of $\Sigma$.
\end{proof}


\section{Decay for the potential $u_{\s}$}
In this section, we establish the desired decay for the potential $u_{\s}$ at those points around which $\s$ is flat.
\begin{lemma} \label{lem: 7.1} Let $\Omega$ be a bounded open set in $\mathbb{R}^{2}$ and $p \in (1,+\infty)$, and let $f\in L^{q_{0}}(\Omega)$ with $q_{0}$ defined in (\ref{(0.1)}). Let $\s$ and $\s^{\prime}$ be closed proper subsets of $\overline{\Omega}$ and $x_{0}\in \mathbb{R}^{2}$. We consider $0<r_{0}<r_{1}$ and assume that $\s^{\prime} \Delta \s \subset \overline{B}_{r_{0}}(x_{0})$. Then for any $\varphi \in \lip(\mathbb{R}^{2})$ such that $\varphi=1$ over $B^{c}_{r_{1}}(x_{0}),\,\ \varphi=0$ over $B_{r_{0}}(x_{0})$, and $\|\varphi\|_{\infty}\leq 1$ on $\mathbb{R}^{2}$, one has
\begin{equation*}
\begin{split}
E_{p}(u_{\s})-E_{p}(u_{\s^{\prime}})\leq \frac{2^{p-1}}{p}\int_{ B_{r_{1}}(x_{0}) } |\nabla u_{\s^{\prime}}|^{p}\diff x &+ \frac{2^{p-1}}{p}\int_{B_{r_{1}}(x_{0}) } | u_{\s^{\prime}}|^{p}|\nabla \varphi|^{p}\diff x \\+ & \int_{B_{r_{1}}(x_{0})}fu_{\s^{\prime}}(1-\varphi)\diff x.
\end{split}
\end{equation*}
\end{lemma}
\begin{proof} Since $u_{\s^{\prime}}\varphi\in W^{1,p}_{0}(\Omega\backslash \s)$ and $u_{\s}$ is a minimizer of $E_{p}$ over $ W^{1,p}_{0}(\Omega\backslash \s)$, then $E_{p}(u_{\s})\leq E_{p}(u_{\s^{\prime}}\varphi)$, and hence,
\begin{align*}
E_{p}(u_{\s})-E_{p}(u_{\s^{\prime}}) & \leq E_{p}(u_{\s^{\prime}}\varphi)-E_{p}(u_{\s^{\prime}}) = \frac{1}{p}\int_{\Omega}|\nabla u_{\s^{\prime}}\varphi +u_{\s^{\prime}}\nabla\varphi|^{p}\diff x \\
& \qquad \qquad -\int_{\Omega}fu_{\s^{\prime}}\varphi\diff x -\frac{1}{p}\int_{\Omega}|\nabla u_{\s^{\prime}}|^{p}\diff x +\int_{\Omega}fu_{\s^{\prime}}\diff x \\
& = \frac{1}{p}\int_{B_{r_{1}}(x_{0})}|\nabla u_{\s^{\prime}}\varphi+u_{\s^{\prime}}\nabla\varphi|^{p}\diff x + \frac{1}{p}\int_{ B^{c}_{r_{1}}(x_{0})}|\nabla u_{\s^{\prime}}|^{p}\diff x \\ 
& \qquad \qquad +\int_{ B_{r_{1}}(x_{0})}fu_{\s^{\prime}}(1-\varphi)\diff x-\frac{1}{p}\int_{\Omega}|\nabla u_{\s^{\prime}}|^{p}\diff x \\ 
& \leq \frac{2^{p-1}}{p}\int_{B_{r_{1}}(x_{0})}|\nabla u_{\s^{\prime}}|^{p}|\varphi|^{p}\diff x+\frac{2^{p-1}}{p}\int_{B_{r_{1}}(x_{0})}|u_{\s^{\prime}}|^{p}|\nabla\varphi|^{p}\diff x \\ 
& \qquad \qquad -\frac{1}{p}\int_{B_{r_{1}}(x_{0}) }|\nabla u_{\s^{\prime}}|^{p}\diff x+\int_{ B_{r_{1}}(x_{0})}fu_{\s^{\prime}}(1-\varphi)\diff x \\
& \leq \frac{2^{p-1}}{p} \int_{B_{r_{1}}(x_{0}) }|\nabla u_{\s^{\prime}}|^{p}\diff x +\frac{2^{p-1}}{p}\int_{ B_{r_{1}}(x_{0})}|u_{\s^{\prime}}|^{p}|\nabla\varphi|^{p}\diff x \\ 
& \qquad  \qquad +\int_{ B_{r_{1}}(x_{0})}fu_{\s^{\prime}}(1-\varphi)\diff x,
\end{align*}
which concludes the proof.
\end{proof}

\begin{lemma} \label{lem: 7.2} Let $\Omega$ be a bounded open set in $\mathbb{R}^2$ and $p\in (1,+\infty)$, and let $f\in L^{q}(\Omega)$ with $q\geq q_{0}$, where $q_{0}$ is defined in (\ref{(0.1)}). Let  $\s$ be a closed arcwise connected proper subset of $\overline{\Omega}$ and $x_{0}\in \mathbb{R}^2$, and let $0<2r_{0}\leq r_{1}\leq 1$ satisfy 
\begin{equation}
\s\cap B_{r_{0}}(x_{0})\not =\emptyset,\,\ \s\backslash B_{r_{1}}(x_{0}) \not = \emptyset. \label{7.1}
\end{equation}
Then for any $r\in [r_{0}, r_{1}/2]$, for any $\varphi \in \lip(\mathbb{R}^{2})$ such that $\|\varphi\|_{\infty}\leq 1$ and $\varphi=1$ over $B^{c}_{2r}(x_{0}), \,\ \varphi= 0$ over $B_{r}(x_{0})$ and $\|\nabla \varphi\|_{\infty}\leq 1/r$, the following assertions hold.
\begin{enumerate}[label=(\roman*)]
\item There exists $C>0$ depending only on $p$, such that:
\begin{equation}
\int_{B_{2r}(x_{0})}|u_{\s}|^{p}|\nabla \varphi|^{p}\diff x \leq C\int_{B_{2r}(x_{0})}|\nabla u_{\s}|^{p}\diff x. \label{7.2}
\end{equation}
\item There exists $C>0$ depending only on $p$, $q_{0}$, $q$ and $\|f\|_{q}$ such that
\begin{equation}
\int_{B_{2r}(x_{0})} fu_{\s}(1-\varphi)\diff x \leq C\int_{B_{2r}(x_{0})}|\nabla u_{\Sigma}|^{p}\diff x + C r^{2+p^{\prime}-\frac{2p^{\prime}}{q}}. \label{7.3}
\end{equation}
\end{enumerate}
\end{lemma}

\begin{proof} 
In this proof we write $u$ instead of $u_{\s}$ to lighten the notation. Due to (\ref{7.1}), $\Sigma  \cap \partial B_{s}(x_{0}) \neq \emptyset$ for all $s \in [r,2r]$ and then, since $u=0$\, $p$-q.e. on $\Sigma$ and $u \in W^{1,p}(B_{2r}(x_{0}))$, by Corollary~\ref{cor: 2.11}, there is a constant $C=C(p)>0$ such that
\begin{equation}
\int_{B_{2r}(x_{0})}|u|^{p}\diff x\leq Cr^{p} \int_{B_{2r}(x_{0})}|\nabla u |^{p}\diff x. \label{7.4}
\end{equation}
Therefore, 
\begin{align*}
\int_{B_{2r}(x_{0})} |u|^{p}|\nabla \varphi|^{p}\diff x & \leq \frac{1}{r^{p}} \int_{B_{2r}(x_{0})} |u |^{p} \diff x \\
& \leq C \int_{B_{2r}(x_{0})} |\nabla u|^{p} \diff x,
\end{align*}
which proves \eqref{7.2}.

Then let us prove  (\ref{7.3}). First, notice that due to (\ref{7.4}) and the fact that $2r\leq 1$, there is a constant  $C_{0}=C_{0}(p)>0$ such that
\begin{equation}
\|u\|_{W^{1,p}(B_{2r}(x_{0}))}\leq C_{0} \|\nabla u\|_{L^{p}(B_{2r}(x_{0}))}. \label {P1}
\end{equation}
Using the Sobolev embeddings (see \cite[Theorem 7.26]{PDE}) together with (\ref{P1}), we deduce that there is a constant $\widetilde{C}=\widetilde{C}(p, q_{0})>0$ such that
\begin{equation} \label{(E10)}
\|u\|_{L^{q^{\prime}_{0}}(B_{2r}(x_{0}))} \leq \widetilde{C} r^{\beta}\|\nabla u\|_{L^{p}(B_{2r}(x_{0}))},
\end{equation}
where
\begin{equation} \label{(E20)}
\beta=1-\frac{2}{p} \,\  \text{if} \,\ 2<p<+\infty , \,\ \,\ \beta=\frac{2}{q^{\prime}_{0}} \,\ \text{if} \,\ p=2, \,\ \,\ \beta=0 \,\ \text{if} \,\  1<p<2,
\end{equation}
and it is worth noting that in the case $2<p<+\infty$ we have used that $u(\xi)=0$ for some $\xi \in \Sigma  \cap B_{2r}(x_{0})$ yielding the following: for all $x \in B_{2r}(x_{0})$ we have 
\[
|u(x)|=|u(x)-u(\xi)|\leq C_{1} r^{1-\frac{2}{p}}\|u\|_{W^{1,p}(B_{2r}(x_{0}))}
\]
for some $C_{1}=C_{1}(p)>0$.
Thus, using the fact that $\|1-\varphi\|_{\infty}\leq 1$ and H\"{o}lder's inequality, we get
\begin{align*}
\int_{B_{2r}(x_{0})}fu(1-\varphi)\diff x& \leq \|f\|_{L^{q_{0}}(B_{2r}(x_{0}))}\|u\|_{L^{q^{\prime}_{0}}(B_{2r}(x_{0}))} \leq |B_{2r}(x_{0})|^{\frac{1}{q_{0}}-\frac{1}{q}}\|f\|_{L^{q}(\Omega)}\|u\|_{L^{q^{\prime}_{0}}(B_{2r}(x_{0}))}\\
&\leq Cr^{2(\frac{1}{q_{0}}-\frac{1}{q})+\beta}\|\nabla u\|_{L^{p}(B_{2r}(x_{0}))}  \,\ (\text{by using (\ref{(E10)}}))\\
& = Cr^{3-\frac{2}{p}-\frac{2}{q}}\|\nabla u\|_{L^{p}(B_{2r}(x_{0}))} \,\ (\text{by using (\ref{(0.1)}) and (\ref{(E20)}}))\\
&\leq C  r^{2+p^{\prime}-\frac{2p^{\prime}}{q}} +  C\|\nabla u\|_{L^{p}(B_{2r}(x_{0}))}^p \,\ (\text{by Young's inequality}), 
\end{align*}
where $C=C(p,q_{0},q, \|f\|_{L^{q}(\Omega)})>0$.   This achieves the proof of Lemma~\ref{lem: 7.2}. \end{proof}

The following corollary  follows directly from Lemma~\ref{lem: 7.1} and Lemma~\ref{lem: 7.2}, so we omit the proof.

\begin{cor}\label{cor: 7.3} \textit{  Let $\Omega$ be a bounded open set in $\mathbb{R}^2$ and $p\in (1,+\infty)$, and let $f\in L^{q}(\Omega)$ with $q\geq q_{0}$, where $q_{0}$ is defined in (\ref{(0.1)}). Let $\s$ and $\s^{\prime}$ be closed arcwise connected proper subsets of $\overline{\Omega}$, and let $x_{0}\in \mathbb{R}^2$. Suppose that $0<2r_{0}\leq r_{1}\leq 1$, $\s^{\prime} \Delta \s \subset \overline{B}_{r_{0}}(x_{0})$ and}
\[
\Sigma^{\prime}\cap B_{r_{0}}(x_{0})\not =\emptyset,\,\ \Sigma^{\prime}\backslash B_{r_{1}}(x_{0}) \not = \emptyset.
\]
\textit{Then for any $r \in [r_{0}, r_{1}/2]$ we have:
\begin{equation}
E_{p}(u_{\s}) - E_{p}(u_{\Sigma^{\prime}}) \leq C \int_{B_{2r}(x_{0})} |\nabla u_{\Sigma^{\prime}}|^{p}\diff x + Cr^{2+p^{\prime}-\frac{2p^{\prime}}{q}}\label{7.5}
\end{equation}
for some constant $C>0$ depending only on $p$, $q_{0}$, $q$ and $\|f\|_{q}$.}
\end{cor}

We now start to prove some decay estimates on the $p$-energy. We begin with  the simple case of a weak solution of the $p$-Laplace equation vanishing on a line, for which we can argue by reflection.

\begin{lemma} \label{lem: 7.4}
Let $p\in (1,+\infty)$.   Then there is a constant $C=C(p)>0$ such that for all   $u\in W^{1,p}(B_{1}),\,\ u=0$ $p$-q.e. on $[-1,1]\times \{0\}$ being a weak solution of the $p$-Laplace equation in $B_{1}\backslash ([-1,1]\times \{0\})$,
\begin{equation*}
\ess_{B_{1/2}} |\nabla u|^{p} \leq C \int_{B_{1}} |\nabla u|^{p} \diff x. 
\end{equation*}
\end{lemma}

\begin{proof} Consider the restrictions of $u$ on $B^{+} = B_{1} \cap \{x_{2} \geq 0\}$ and on $B^{-} =B_{1} \cap \{x_{2} \leq 0\}$ and extend them on $B_{1}$ using the Schwarz reflection. We show that each of the obtained functions is a weak solution of the corresponding $p$-Laplace equation in $B_{1}$. Thus we define
\begin{equation*}
\begin{split}
\tilde{u}(x_{1},x_{2})&=
\begin{cases}
u(x_{1},x_{2}) &  \text{if} \,\ (x_{1},x_{2}) \in B^{+}\\
-u(x_{1},-x_{2}) & \text{if} \,\ (x_{1},x_{2}) \in B^{-}
\end{cases} 
\\
\overline{u}(x_{1},x_{2})&=
\begin{cases}
-u(x_{1},-x_{2}) & \text{if} \,\ (x_{1},x_{2}) \in B^{+}\\
u(x_{1},x_{2}) & \text{if} \,\ (x_{1},x_{2}) \in B^{-}.
\end{cases}
\end{split}
\end{equation*}
It is clear that $\tilde{u}, \overline{u} \in W^{1,p}(B_{1})$ and $\tilde{u}=\overline{u}=0$ $ p$-q.e. on $[-1,1]\times \{0\}$. We claim that $\tilde{u}$ and $\overline{u}$ are weak solutions in $B_{1}$. Indeed, denoting $B_{1} \cap \{x_{2}>0\}$ by $int(B^{+})$ and $B_{1}\cap \{x_{2}<0\}$ by $int(B^{-})$, for an arbitrary test function $\varphi \in C^{\infty}_{0}(B_{1})$ we have
\begin{align*}
\int_{B_{1}}|\nabla \tilde{u}|^{p-2} \nabla \tilde{u} \nabla \varphi\diff x &= \int_{int(B^{+})} |\nabla u|^{p-2} \nabla u \nabla \varphi \diff x\\
+\int_{int(B^{-})}|\nabla u(& x_{1},- x_{2})|^{p-2}\langle(-\partial_{1}u(x_{1}, -x_{2}), \partial_{2}u(x_{1},-x_{2})), \nabla \varphi(x_{1},x_{2})\rangle \diff x_{1} \diff x_{2} \\
&= \int_{int(B^{+})} |\nabla u|^{p-2} \nabla u \nabla \varphi \diff x\\
-\int_{int(B^{+})}|\nabla u(& x_{1}, x_{2})|^{p-2}\langle \nabla u(x_{1},x_{2}), (\partial_{1}\varphi(x_{1},-x_{2}),-\partial_{2}\varphi(x_{1},-x_{2})) \rangle \diff x_{1} \diff x_{2} \\
&= \int_{int(B^{+})}|\nabla u|^{p-2} \nabla u \nabla \psi\diff x, \numberthis \label{7.7} 
\end{align*}
where   $\psi(x_{1},x_{2})=\varphi(x_{1},x_{2})-\varphi(x_{1},-x_{2}),\,\ (x_{1},x_{2}) \in int(B^{+})$. Since $\left.\tilde{u}\right|_{int(B^{+})}\equiv \left.u\right|_{int(B^{+})}$  is a weak solution in $int(B^{+})$ and since $\psi \in W^{1,p}_{0}(int(B^{+}))$, using (\ref{7.7}), we get that 
\begin{equation*}
\int_{B_{1}}|\nabla \tilde{u}|^{p-2}\nabla \tilde{u}\nabla \varphi\diff x=0. 
\end{equation*}
As $\varphi \in C^{\infty}_{0}(B_{1})$ was arbitrarily chosen, we deduce that $\tilde{u}$ is a weak solution in $B_{1}$. The proof of the fact that $\overline{u}$ is a weak solution in $B_{1}$ is similar. Thus by  \cite[Proposition 3.3]{DIBENEDETTO1983827} there is $C=C(p)>0$ such that
\begin{align*}
\ess_{B_{1/2}} |\nabla \tilde{u}|^{p} & \leq C \int_{B_{1}} |\nabla \tilde{u}|^{p} \diff x \\
\ess_{B_{1/2}} |\nabla \overline{u}|^{p} & \leq C \int_{B_{1}} |\nabla \overline{u}|^{p} \diff x.
\end{align*}
Therefore, 
\begin{align*}
\ess_{B_{1/2}} |\nabla u|^{p} & \leq \ess_{B_{1/2}} |\nabla \tilde{u}|^{p} + \ess_{B_{1/2}} |\nabla \overline{u}|^{p} \\
& \leq C \bigl(\int_{B_{1}} |\nabla \tilde{u}|^{p} \diff x +  \int_{B_{1}} |\nabla \overline{u}|^{p} \diff x \bigr) \\
& \leq 2C \int_{B_{1}} |\nabla u|^{p} \diff x.
\end{align*}
This completes the proof of Lemma~\ref{lem: 7.4}.
\end{proof}

\begin{cor} \label{cor: 7.5} Let $u$ be a weak solution of the $p$-Laplace equation in $B_{1}\backslash ([-1,1]\times \{0\})$ and let $u=0$ $p$-q.e. on $[-1,1]\times \{0\}$. Then $u$ is Lipschitz continuous on $B_{1/2}$.
\end{cor}

\begin{cor} \label{cor: 7.6} \textit{ There is a constant $C_{0}=C_{0}(p)>2$ such that if $u$ is a weak solution of the $p$-Laplace equation in $B_{1}\backslash ([-1,1]\times \{0\})$ and $ u=0$ $p$-q.e. on $[-1,1]\times \{0\}$, then}
\begin{equation*}
\int_{B_{r}} |\nabla u|^{p}\diff x \leq C_{0}r^{2} \int_{B_{1}} |\nabla u|^{p} \diff x\,\ \text{for all $r \in (0, 1/2]$}. 
\end{equation*}
\end{cor}

\begin{proof}[Proof of Corollary \ref{cor: 7.6}] By Lemma~\ref{lem: 7.4} we know that for some $C=C(p)>1$,
\begin{equation*}
\ess_{B_{1/2}} |\nabla u|^{p} \leq C \int_{B_{1}} |\nabla u|^{p} \diff x. 
\end{equation*}
We deduce that for $r \leq 1/2$,
$$\int_{B_{r}}|\nabla u|^p\diff x\leq \left(\ess_{B_{1/2}} |\nabla u|^{p}\right) \pi r^2\leq \pi C r^2 \int_{B_1}|\nabla u|^p \diff x.$$
\end{proof}

 Next we use a compactness argument to derive a similar estimate for a weak solution of the $p$-Laplace equation vanishing on a set $\Sigma$ which is close enough to a line, in the Hausdorff distance.

\begin{lemma} \label{lem: 7.7} Let $p\in (1,+\infty)$ and let $C_{0}$ be a constant as in Corollary~\ref{cor: 7.6}. Then for every $\varrho \in (0,1/2]$ there is $\varrho_{0} \in (0, \varrho)$ such that the following holds. Let $\s \subset \mathbb{R}^{2}$ be a closed set such that $(\s \cap B_{r}(x_{0})) \cup \partial B_{r}(x_{0})$ is connected and there is an affine line $P$, passing through $x_{0}$, such that $d_{H}(\s\cap \overline{B}_{r}(x_{0}), P\cap \overline{B}_{r}(x_{0}))\leq \varrho_{0}r$. Then for any weak solution $u$ of the $p$-Laplace equation in $B_{r}(x_{0})\backslash\s$, vanishing $p$-q.e. on $\s\cap \overline{B}_{r}(x_{0})$, the following estimate holds
\begin{equation*}
\int_{B_{\varrho r}(x_{0})} |\nabla u|^{p}\diff x\leq (C_{0} \varrho)^{2}\int_{B_{r}(x_{0})} |\nabla u|^{p} \diff x.
\end{equation*}
\end{lemma}

\begin{proof} Since the $p$-Laplacian is invariant under scalings, rotations and translations, it is not restrictive to assume that $B_{r}(x_{0})=B_{1}$ and $P\cap \overline{B}_{r}(x_{0})=[-1,1]\times \{0\}$. For the sake of contradiction, suppose that  for some $\varrho \in (0,1/2]$ there exist sequences $(\varepsilon_{n})_{n},$ $(\s_{n})_{n}$ and $(u_{n})_{n}$ such that $\varepsilon_{n}\downarrow 0$ as $n \to +\infty$; $\s_{n}$ is closed, $(\s_{n}\cap B_{1})\cup \partial B_{1}$ is connected, $d_{H}(\s_{n}\cap \overline{B}_{1}, [-1,1]\times \{0\})\leq \varepsilon_{n}$ and hence
\begin{equation}
d_{H}({\s}_{n}\cap \overline{B}_{1}, [-1,1]\times \{0\})\to 0 \,\ \text{as}\,\ n\to +\infty; \label{7.9}
\end{equation}
$u_{n}$ is a weak solution in $B_{1}\backslash \s_{n},\,\ u_{n}=0$ $p$-q.e. on $\s_{n}\cap \overline{B}_{1}$ and  
\begin{equation}
\int_{B_{\varrho}}|\nabla u_{n}|^{p}\diff x>(C_{0}\varrho)^{2}\int_{B_{1}}|\nabla u_{n}|^{p}\diff x. \label{7.10} 
\end{equation} 
Thus for any $n$ we can define
\begin{equation}
v_{n}(x)=\frac{u_{n}(x)}{\bigl(\displaystyle\int_{B_{1}}|\nabla u_{n}|^{p}\diff x\bigr)^{\frac{1}{p}}}. \label{7.11}
\end{equation}
Clearly $v_{n}=0$ $p$-q.e. on $\s_{n}\cap \overline{B}_{1}$ and 
\begin{equation}
\int_{B_{1}}|\nabla v_{n}|^{p}\diff x=1. \label{7.12}
\end{equation}
By (\ref{7.9}) and by the fact that $({\s}_{n} \cap B_{1})\cup \partial B_{1}$ is connected, there is a constant $\widetilde{C}>0$ (independent of $n$) such that for any $n$ large enough we have
\begin{equation*}
{\rm Cap}_{p}({\s}_{n}\cap \overline{B}_{1})\geq \widetilde{C}.
\end{equation*}
Then, using the above estimate together with Proposition \ref{prop: 2.7} and with (\ref{7.12}), we conclude that the sequence $(v_{n})_{n}$ is bounded in $W^{1,p}(B_{1})$. Hence, up to a subsequence still denoted by the same index, we have
\begin{align}
v_{n} & \rightharpoonup v \,\ \text{in} \,\ W^{1,p}(B_{1}) \label{7.13} \\ 
v_{n} & \to v \,\ \text{in} \,\ L^{p}(B_{1}), \label{7.14}
\end{align}
for some $v \in W^{1,p}(B_{1})$.

Let us now show that $v=0$ $p$-q.e. on $[-1,1]\times \{0\}$. For any $t\in (0,1)$ we fix $\psi \in C^{1}_{0}(B_{1}),\,\ \psi=1$ on $\overline{B}_{t}$ and $0\leq \psi \leq 1$. Since $(\s_{n}\cap B_{1})\cup \partial B_{1}$ is connected for all $n \in \mathbb{N}$ and $d_{H}(\s_{n}\cap \overline{B}_{1}, [-1,1]\times\{0\}) \to 0$, as $n\to +\infty$, it follows (see \cite{Bucur}) that the sequence of Sobolev spaces $W^{1,p}_{0}(B_{1}\backslash \s_{n})$ converges in the sense of Mosco to $W^{1,p}_{0}(B_{1}\backslash ([-1,1]\times~\{0\}))$. Note that by (\ref{7.13}), $v_{n}\psi \rightharpoonup v\psi$\,\ in\,\ $W^{1,p}(\mathbb{R}^{2})$ and using the definition of limit in the sense of Mosco, we deduce that $v\psi \in W^{1,p}_{0}(B_{1}\backslash ([-1,1]\times \{0\}))$. This implies that $v=0$ $p$-q.e. on $[-t,t]\times \{0\}$. As $t\in (0,1)$ was arbitrarily chosen, we deduce that $v=0$ $p$-q.e. on $[-1,1]\times \{0\}$.

We claim that  $v$ is a weak solution of the $p$-Laplace equation in $B_1\setminus ([-1,1]\times \{0\})$. Notice that, in contrary to the linear case, it is not so clear how  to pass to the limit in the weak formulation using only the weak convergence of $\nabla v_n$ to $\nabla v$ in $L^p(B_{1})$. But one can argue exactly as in  the proof of \cite[Proposition 3.7]{Bucur} to get that $|\nabla v_n|^{p-2}\nabla v_n$ weakly converges to $|\nabla v|^{p-2}\nabla v$ in $L^{p'}(B_{1})$, and this is enough to pass to the limit in the weak formulation. We refer to \cite{Bucur} for further details.

We now  want to prove the strong convergence of $\nabla v_n$ to $\nabla v$ in $L^p(B_{1/2})$. Since for all $n$ we have that $\int_{B_{1}}|\nabla v_{n}|^p \diff x=1$, we may assume that the sequence $|\nabla v_n|^p \diff x$ of probability measures over $B_{1}$  weakly* converges (in the duality with $C_{0}(B_{1})$) to some finite Borel measure $\mu$ over $B_1$. Then we select some $t_{0}\in (1/2,3/4)$  such that $\mu(\partial B_{t_{0}})=0$. Such $t_{0}$ exists, since otherwise $\mu(\partial B_{t})>0$ for all $t\in (1/2, 3/4)$ and therefore we can find a positive integer number $j$ and an uncountable set of indices $A \subset (1/2, 3/4)$ such that for all $t \in A$ we have that $\mu(\partial B_{t})>1/j$ that leads to a contradiction with the fact that $\mu(B_{1})<+\infty$.

From the weak convergence of $\nabla v_n$ in $L^p$ we only need to prove that  $\|\nabla v_{n}\|_{L^{p}(B_{t_{0}};\mathbb{R}^{2})}$ tends to $\|\nabla v\|_{L^{p}(B_{t_{0}};\mathbb{R}^{2})}$. We already have, still by weak convergence,
$$\int_{B_{t_{0}}}|\nabla v|^{p}\diff x \leq \liminf_{n\to +\infty} \int_{B_{t_{0}}}|\nabla v_n|^{p}\diff x,$$
thus it remains to prove the reverse inequality, with a limsup. For this purpose we shall use the minimality of $v_n$.

Let $\chi$ be smooth cut-off function equal to 1 on $B_{t_0}$ and zero outside of $B_{3/4}$, and consider the function $\chi v \in W^{1,p}_{0}(B_{1}\backslash ([-1,1]\times \{0\}))$. By the definition of convergence in the sense of Mosco, it follows that there is a sequence $(\tilde{v}_{n}) \subset W^{1,p}_{0}(B_{1}\backslash \s_{n})$ converging to $\chi v$ strongly in $W^{1,p}(B_{1})$.

Now, fix an arbitrary $\delta \in (0,t_0-1/2)$, and let $\eta_\delta \in C^\infty_c(B_{t_0})$ be smooth cut-off function satisfying
$$\eta_\delta = 1  \text{ on } B_{t_0-\delta}, \quad \quad |\nabla \eta_\delta|\leq \frac{C}{\delta}.$$

Then we define

$$w_{n}=  \eta_\delta \tilde{v}_{n} +(1-\eta_\delta)v_{n}.$$

In particular, $w_n =0$ $p$-q.e. on $\Sigma_n$ and $w_n=v_n$ outside $B_{t_0}$. By the minimality of $v_n$ (see Theorem~\ref{thm: 2.2}) we infer
\begin{eqnarray}
\int_{B_{t_0}}|\nabla v_n|^p \, \diff x \leq \int_{B_{t_0}} |\nabla w_n|^p \,\diff x. \label{ttt}
\end{eqnarray}

Recalling that for any $\varepsilon>0$ there is a constant $c_{\varepsilon}>0$ such that for all nonnegative real numbers $a,b,$
\[
(a+b)^{p} \leq c_{\varepsilon} a^{p} + (1+\varepsilon)b^{p},
\]
computing $\nabla w_{n}$ and using \eqref{ttt} we obtain the following chain of estimates
\begin{align*}
\int_{B_{t_0}}|\nabla v_{n}|^{p}\diff x & \leq \int_{B_{t_0}} |\eta_\delta \nabla \tilde{v}_{n}+ (1-\eta_\delta)\nabla v_{n} +\nabla \eta_\delta (\tilde{v}_{n}-v_{n})|^{p} \diff x\\
& \leq c(\varepsilon)\int_{B_{t_0}}|\nabla \eta_\delta|^{p}|\tilde{v}_{n}-v_{n}|^{p}\diff x \\
& \qquad \qquad \qquad \qquad + (1+\varepsilon) \int_{B_{t_0}}|\eta_\delta \nabla \tilde{v}_{n} + (1-\eta_\delta) \nabla v_{n}|^{p}\diff x\\
& \leq c(\varepsilon,\delta)\int_{B_{t_0}}|\tilde{v}_{n}-v_{n}|^{p}\diff x + (1+\varepsilon)\int_{B_{t_0}}(1- \eta_\delta)|\nabla v_{n}|^{p}\diff x\\ 
& \qquad  \qquad \qquad \qquad + (1+\varepsilon)\int_{B_{t_0}}\eta_\delta |\nabla \tilde{v}_{n}|^{p}\diff x.
\end{align*}
Notice that since $|\nabla v_{n}|^{p}\diff x$ weakly* converges to $\mu$ (in the duality with $C_{0}(B_{1})$) and since $\mu(\partial B_{t_{0}})=0$, we obtain that $\int_{B_{t_{0}}}|\nabla v_{n}|^{p}\diff x$ tends to $\mu(B_{t_{0}})$ as $n\to +\infty$ (it is easy to see by taking sequences $(g_{m})_{m},\, (h_{m})_{m} \subset C_{0}(B_{1})$ such that $g_{m}\downarrow \mathds{1}_{\overline{B}_{t_{0}}}$, $h_{m} \uparrow \mathds{1}_{B_{t_{0}}}$ and by using the definition of the weak* convergence of measures).
Passing to the limsup, from the strong convergence in $L^p(B_{t_0})$ of both $v_n$ and $\tilde v_n$ to $v$  we get 
$$\int_{B_{t_0}}|\tilde{v}_{n}-v_{n}|^{p}\diff x \to 0,$$
thus
$$\limsup_{n\to +\infty} \int_{B_{t_0}}|\nabla v_{n}|^{p}\diff x \leq (1+\varepsilon)\mu(\overline{B}_{t_0}\setminus B_{t_0-\delta}) +(1+\varepsilon) \int_{B_{t_0}}   |\nabla v|^p \diff x.$$
Letting now $\delta$ tend to $0+$ and using the fact that $\mu(\partial B_{t_0})=0$ we get
$$\limsup_{n\to + \infty} \int_{B_{t_0}}|\nabla v_{n}|^{p}\diff x \leq (1+\varepsilon) \int_{B_{t_0}}   |\nabla v|^p \diff x,$$
and we finally conclude by letting $\varepsilon$ tend to $0+$ to get
$$\limsup_{n\to + \infty} \int_{B_{t_0}}|\nabla v_{n}|^{p}\diff x \leq   \int_{B_{t_0}}   |\nabla v|^p \diff x,$$
which proves the strong convergence of $\nabla v_n$ to $\nabla v$ in $L^p(B_{t_0})$.

Using (\ref{7.10}) and (\ref{7.11}) and passing to the limit we therefore arrive at
\begin{equation}
\int_{B_{\varrho}}|\nabla v|^{p} \diff x \geq (C_{0}\varrho)^{2}. \label{7.23}
\end{equation}
On the other hand, by Corollary~\ref{cor: 7.6} applied with $u=v$ and by (\ref{7.12}) and (\ref{7.13}) we get
\[
\int_{B_{\varrho}}|\nabla v|^{p}\diff x \leq C_{0}\varrho^{2},
\]
which leads to a contradiction with (\ref{7.23}), since $\varrho>0$ and $C_{0}>2$, concluding the proof. 
\end{proof}
Now we would like to treat the second member $f$. For that purpose we shall use the following lemma  (see \cite[Lemma 2.2]{ABC}), which will allow us to control the difference between the potential $u_{\Sigma}$ and its Dirichlet replacement on a ball with a crack. 
\begin{lemma}[\cite{ABC}] \label{lem: 7.10} Let $U$ be an open set in $\mathbb{R}^2, N\geq 2$, and let $u_{1}, u_{2} \in W^{1,p}(U)$. If $2\leq p<+\infty$, then:
\begin{equation}
\int_{U}|\nabla u_{1}-\nabla u_{2}|^{p}\diff x \leq c_{0}\int_{U}\langle |\nabla u_{1}|^{p-2}\nabla u_{1}-|\nabla u_{2}|^{p-2}\nabla u_{2}, \nabla u_{1}-\nabla u_{2}\rangle \diff x, \label{7.25}
\end{equation}
where $c_{0}$ depends only on $p$.\\ If $1<p< 2$, then:
\begin{equation} \label{7.26}
\begin{split}
\bigl(\int_{U}|\nabla u_{1}-\nabla u_{2}|^{p}\diff x\bigr)^{\frac{2}{p}} & \leq K(u_{1},u_{2}) \int_{U}\langle |\nabla u_{1}|^{p-2}\nabla u_{1}-|\nabla u_{2}|^{p-2}\nabla u_{2}, \nabla u_{1}-\nabla u_{2}\rangle\diff x,
\end{split}
\end{equation}
where $K(u_{1},u_{2})$ stands for:
\begin{align*}
K(u_{1},u_{2})=2\bigl(\int_{U}|\nabla u_{1}|^{p}\diff x + \int_{U}|\nabla u_{2}|^{p}\diff x\bigr)^{\frac{2-p}{p}}.
\end{align*}
\end{lemma}

Now we can control  the difference between a weak solution of the $p$-Poisson equation and its Dirichlet replacement on a ball with a crack.

\begin{lemma} \label{lem: 7.11}  Let $p\in (1, +\infty)$ and $f \in L^{q}(B_{r_{1}}(x_{0}))$ with $q> q_{0}$, where $q_{0}$ is defined in~(\ref{(0.1)}), and let $\s$ be a closed arcwise connected set in $\mathbb{R}^{2}$ and $0<2 r_{0}\leq r_{1}\leq 1$ satisfy
\begin{equation*}
\s \cap B_{r_{0}}(x_{0})\neq \emptyset,\,\ \s \backslash B_{r_{1}}(x_{0}) \neq \emptyset \,\  \text{and} \,\ B_{r_{1}}(x_{0})\backslash \Sigma \neq \emptyset. 
\end{equation*}
Let $u \in W^{1,p}(B_{r_{1}}(x_{0})),$ $u=0$ $p$-q.e. on $\s\cap \overline{B}_{r_{1}}(x_{0})$ be the solution of the $p$-Poisson equation $-\Delta_{p}v = f \,\ \text{in}\,\ B_{r_{1}}(x_{0})\backslash \s$
in the weak sense, which means that 
\begin{equation}
\int_{B_{r_{1}}(x_{0})}|\nabla u|^{p-2}\nabla u \nabla \varphi \diff x = \int_{B_{r_{1}}(x_{0})} f\varphi \diff x \,\ \text{for all}\,\ \varphi \in W^{1,p}_{0}(B_{r_{1}}(x_{0})\backslash \s). \label{wf p-Poisson}
\end{equation}
Let $w \in W^{1,p}(B_{r_{1}}(x_{0})),$ $w=0$ $p$-q.e. on $\s\cap \overline{B}_{r_{1}}(x_{0})$ be the solution of the following $p$-Laplace equation
\begin{equation*}
\begin{cases}
-\Delta_{p} v &=\,\  0  \,\ \text{in}\,\ B_{r_{1}}(x_{0})\backslash \s \\
\qquad v & =\,\ u \ \text{on}\,\ \partial B_{r_{1}}(x_{0})\cup (\s\cap B_{r_{1}}(x_{0})),
\end{cases}
\end{equation*}
in the weak sense, which means that $w - u \in W^{1,p}_{0}(B_{r_{1}}(x_{0})\backslash \s)$ and
\begin{equation}
\int_{B_{r_{1}}(x_{0})}|\nabla w|^{p-2}\nabla w \nabla \varphi \diff x = 0 \,\ \text{for all}\,\ \varphi \in W^{1,p}_{0}(B_{r_{1}}(x_{0})\backslash \s). \label{wf p-Laplace}
\end{equation}
If $2\leq p<+\infty$, then:
\begin{equation}
\int_{B_{r_{1}}(x_{0})}|\nabla u-\nabla w|^{p}\diff x \leq Cr^{2+p^{\prime}-\frac{2p^{\prime}}{q}}_{1}, \label{Estimate 4.27}
\end{equation}
where $C=C(p, q_{0}, q, \|f\|_{q})>0$.\\
If $1<p <2$, then:
\begin{equation}
\int_{B_{r_{1}}(x_{0})}|\nabla u- \nabla w|^{p}\diff x \leq C (K(u, u))^{p} (r^{p-1}_{1})^{2+p^{\prime}-\frac{2p^{\prime}}{q}}, \label{7.31}
\end{equation}
where $C=C(p, q_{0}, q, \|f\|_{q})>0$ and $K(\cdot, \cdot)$ as in Lemma~\ref{lem: 7.10} with $U=B_{r_{1}}(x_{0})$.

\end{lemma}

\begin{proof} Every ball in this proof is centered at $x_{0}$. For convenience, let us define $z= u- w$. Since $z=0$ $p$-q.e. on $\Sigma \cap \overline{B}_{r_{1}}$, by Corollary~\ref{cor: 2.11} and the fact that $r_{1}\leq 1$, there is a constant $C=C(p)>0$ such that
\begin{equation}
\|z\|_{W^{1,p}(B_{r_{1}})}\leq C \|\nabla z\|_{L^{p}(B_{r_{1}})}. \label {P}
\end{equation}
Then, using the Sobolev embeddings (see \cite[Theorem 7.26]{PDE}) together with (\ref{P}), we deduce that there is a constant $\widetilde{C}=\widetilde{C}(p, q_{0})>0$ such that
\begin{equation} \label{S}
\|z\|_{L^{q^{\prime}_{0}}(B_{r_{1}})} \leq \widetilde{C} r^{\alpha}_{1}\|\nabla z\|_{L^{p}(B_{r_{1}})},
\end{equation}
where
\begin{equation} \label{Ex S}
\alpha=1-\frac{2}{p} \,\  \text{if} \,\ 2<p<+\infty , \,\ \,\ \alpha=\frac{2}{q^{\prime}_{0}} \,\ \text{if} \,\ p=2, \,\ \,\ \alpha=0 \,\ \text{if} \,\  1<p<2,
\end{equation}
in particular, in the case $2<p<+\infty$ we have used that $z(\xi)=0$ for some $\xi \in \s \cap B_{r_{1}}$ yielding the following: for all $x \in B_{r_{1}}$ one has $|z(x)|=|z(x)-z(\xi)|\leq C_{0}(2r_{1})^{1-\frac{2}{p}}\|z\|_{W^{1,p}(B_{r_{1}})}$ for some $C_{0}=C_{0}(p)>0$. Let us consider the next two cases. \\
\text{Case 1:} $2\leq p<+\infty$. Using (\ref{7.25}) and the fact that $z$ is a test function for (\ref{wf p-Poisson}) and (\ref{wf p-Laplace}), we get 
\begin{align*}
\int_{B_{r_{1}}}|\nabla z|^{p}\diff x & \leq c_{0}\int_{B_{r_{1}}}\langle|\nabla u|^{p-2}\nabla u -|\nabla w|^{p-2}\nabla w, \nabla z \rangle\diff x \\
& = c_{0} \int_{B_{r_{1}}}f z \diff x
\end{align*}
where $c_{0}=c_{0}(p)>0$. Applying H\"older's inequality to the right-hand side of the latter formula and using (\ref{S}), we obtain
\begin{align*}
\int_{B_{r_{1}}}|\nabla z|^{p}\diff x &\leq c_{0} \|f\|_{L^{q_{0}}(B_{r_{1}})}\|z\|_{L^{q^{\prime}_{0}}(B_{r_{1}})} \leq c_{0}|B_{r_{1}}|^{\frac{1}{q_{0}}-\frac{1}{q}}\|f\|_{L^{q}(B_{r_{1}})}\|z\|_{L^{q^{\prime}_{0}}(B_{r_{1}})}\\
& \leq C r^{2(\frac{1}{q_{0}}-\frac{1}{q})+\alpha}_{1} \Bigl(\int_{B_{r_{1}}}|\nabla z|^{p}\diff x\Bigr)^{\frac{1}{p}}
\end{align*}
for some $C=(p, q_{0}, q, \|f\|_{q})>0$. Therefore,
\begin{align*}
\int_{B_{r_{1}}}|\nabla z|^{p}\diff x & \leq C^{p^{\prime}}r^{2p^{\prime}(\frac{1}{q_{0}}-\frac{1}{q})+p^{\prime}\alpha}_{1}
\end{align*}
and carefully calculating $(2p^{\prime}(\frac{1}{q_{0}}-\frac{1}{q})+p^{\prime}\alpha)$ where $\alpha$ is defined in (\ref{Ex S}), one gets (\ref{Estimate 4.27}). \\
\text{Case 2:} Let $1<p< 2$. Using (\ref{7.26}), and the fact that $z$ is a test function for (\ref{wf p-Poisson}) and (\ref{wf p-Laplace}), we get
\begin{align*}
\bigl(\int_{B_{r_{1}}}|\nabla z|^{p}\diff x\bigr)^{\frac{2}{p}} & \leq K(u, w) \int_{B_{r_{1}}} \langle |\nabla u|^{p-2}\nabla u-|\nabla w|^{p-2}\nabla w, \nabla z \rangle \diff x\\
& = K(u,w)\int_{B_{r_{1}}}f z \diff x.
\end{align*}
Next, by using H\"older's inequality and then (\ref{S}), we obtain
\begin{equation*}
\begin{split}
\bigl(\int_{B_{r_{1}}}|\nabla z|^{p}\diff x\bigr)^{\frac{2}{p}} & \leq K(u, w) \|f\|_{L^{q_{0}}(B_{r_{1}})} \|z\|_{L^{q^{\prime}_{0}}(B_{r_{1}})}\\&  \leq K(u,w)|B_{r_{1}}|^{\frac{1}{q_{0}}-\frac{1}{q}}\|f\|_{L^{q}(B_{r_{1}})}\|z\|_{L^{q^{\prime}_{0}}(B_{r_{1}})}\\
& \leq C  K(u, w)r^{\frac{2}{q_{0}}-\frac{2}{q}}_{1}\bigl(\int_{B_{r_{1}}}|\nabla z|^{p}\diff x\bigr)^{\frac{1}{p}}\\
& \leq C K(u, u) r^{\frac{2}{q_{0}}-\frac{2}{q}}_{1} \bigr(\int_{B_{r_{1}}}|\nabla z|^{p}\diff x\bigr)^{\frac{1}{p}},
\end{split}
\end{equation*}
for some $C=C(p, q_{0}, q, \|f\|_{q})>0$, where the last estimate comes from the fact that $w$ minimizes the energy $\int_{B_{r_{1}}}|\nabla v|^{p}\diff x$ among all $v$ satisfying $v-u \in W^{1,p}_{0}(B_{r_{1}}\backslash \Sigma)$ (see Theorem~{\ref{thm: 2.2}}) and $u$ is a competitor for $w$. Therefore,
\begin{align*}
\int_{B_{r_{1}}}|\nabla z|^{p}\diff x &\leq C^{p} (K(u,u))^{p} r^{\frac{2p}{q_{0}}-\frac{2p}{q}}_{1}\\
&= C^{p} (K(u,u))^{p} r_{1}^{3p-2-\frac{2p}{q}}=C^{p} (K(u,u))^{p} (r_{1}^{p-1})^{2+p^{\prime}-\frac{2p^{\prime}}{q}}
\end{align*}
{that yields (\ref{7.31}).}
\end{proof}

Gathering together Lemma~\ref{lem: 7.7} and Lemma~\ref{lem: 7.11} we arrive at the following decay estimate for $u_{\Sigma}$. Notice that in the following statement the definition of $\alpha(q)$ also depends on $p$, but we decided to not mention it explicitly to lighten the notation.

\begin{lemma} \label{lem: 7.11P} Let $p\in (1, +\infty)$ and $f \in L^{q}(\Omega)$ with $q> q_{0}$, where $q_{0}$ is defined in (\ref{(0.1)}). Then we can find $a\in (0,1/2)$, $\varepsilon_0 \in (0,a)$  and $C=C(p,q_{0},q,\|f\|_q, |\Omega|)>0$ such that the following holds.  Let $\s\subset \overline{\Omega}$ be a closed arcwise connected set. Let $0<2 r_{0}\leq r_{1}\leq 1$ satisfy $B_{r_{1}}(x_{0})\subset \Omega$,
	\begin{equation*}
	\s \cap B_{r_{0}}(x_{0})\neq \emptyset\,\ \text{and} \,\ \s \backslash B_{r_{1}}(x_{0}) \neq \emptyset, 
	\end{equation*}
	and  assume that there is an affine line $P$, passing through $x_{0}$, such that 
	\begin{eqnarray}
	d_{H}(\s\cap \overline{B}_{r_1}(x_{0}), P\cap \overline{B}_{r_1}(x_{0})) \leq \varepsilon_{0}r_1. \label{distance0}
	\end{eqnarray}
	Then
	\begin{eqnarray}
	\frac{1}{ar_1}\int_{B_{ar_1}(x_0)} |\nabla u_{\s}|^p \diff x \leq \frac{1}{2} \left( \frac{1}{r_1}\int_{B_{r_1}(x_0)} |\nabla u_{\s}|^p \diff x \right)+ C r_1^{\alpha(q)}
	\end{eqnarray}
	where 
	\begin{eqnarray}
	\alpha(q)=
	\left\{
	\begin{array}{cc}
	1+p^{\prime}-\frac{2p^{\prime}}{q} & \text{ if }  \; 2\leq p<+\infty \\
	3(p-1)-\frac{2p}{q} & \text{ if } \; 1<p <2.
	\end{array}
	\right. \label{defalphap}
	\end{eqnarray}
\end{lemma}

\begin{proof} Let $w \in W^{1,p}(B_{r_{1}}(x_{0})),\,\ w=0$ $p$-q.e. on $\s\cap \overline{B}_{r_{1}}(x_{0})$ be the Dirichlet replacement of $u_{\s}$, i.e., the solution of the following $p$-Laplace equation
	\begin{equation*}
	\begin{cases}
	-\Delta_{p} u &=\,\  0  \,\ \text{in}\,\ B_{r_{1}}(x_{0})\backslash \s \\
	\qquad u & =\,\ u_{\s} \ \text{on}\,\ \partial B_{r_{1}}(x_{0})\cup (\s\cap B_{r_{1}}(x_{0})),
	\end{cases}
	\end{equation*}
	in the weak sense, which means that $w - u_{\s} \in W^{1,p}_{0}(B_{r_{1}}(x_{0})\backslash \s)$ and
	\begin{equation}
	\int_{B_{r_{1}}(x_{0})}|\nabla w|^{p-2}\nabla w \nabla \varphi \diff x = 0 \,\ \text{for all}\,\ \varphi \in W^{1,p}_{0}(B_{r_{1}}(x_{0})\backslash \s). \label{7.29}
	\end{equation}
	Let $K(\cdot, \cdot)$ be as in Lemma~\ref{lem: 7.10} with $U=B_{r_{1}}(x_{0})$. Using  (\ref{2.4}) and H\"older's inequality, it is easy to see that
			\begin{equation} \label{estimate for K}
			K(u_{\Sigma}, u_{\Sigma})\leq C_{1}
			\end{equation}
    for some $C_{1}=C_{1}(p, q_{0}, q, \|f\|_{q}, |\Omega|)>0$. Then applying Lemma~\ref{lem: 7.11} and using (\ref{estimate for K}), we know that:
	\begin{equation}
	\int_{B_{r_{1}}(x_{0})}|\nabla u_{\s}-\nabla w|^{p}\diff x \leq Cr_1^{1+\alpha(q)}, \label{7.30}
	\end{equation}
	where $C=C(p, q_{0}, q, \|f\|_{q}, |\Omega|)>0$ and $\alpha(q)$ is defined in \eqref{defalphap}. 
	Now let $C_{0}=C_{0}(p)$ be the constant of Corollary~\ref{cor: 7.6}, and let $a=2^{-p}C_{0}^{-2}$. For every $p\in (1,+\infty)$ the constant $a$ is fixed.   We  can apply Lemma~\ref{lem: 7.7} with $r=r_{1}$ and $\varrho=a$ to the function $w$.  We  then obtain some $\varrho_0 \in (0,a)$ which defines our $\varepsilon_0:=\varrho_0$ such that under the condition \eqref{distance0} it holds	
	\begin{align*}
	\frac{1}{a}\int_{B_{ar_1}(x_{0})} |\nabla w|^{p}\diff x \leq C_{0}^{2}a \int_{B_{r_1}(x_{0})} |\nabla w|^{p}\diff x \leq 2^{-p}\int_{B_{r_1}(x_{0})} |\nabla w|^{p}\diff x.
	\end{align*} 
	Hereinafter in this proof, $C$ denotes a positive constant that can depend only on $p,\,q_{0},\,q$, $\|f\|_{q},\, |\Omega|$ and can be different from line to line. Now we use the elementary inequality $(c+d)^p\leq 2^{p-1}(c^p+d^p)$ to write 
	
	\begin{eqnarray}
	\frac{1}{a}\int_{B_{ar_1}(x_0)} |\nabla u_{\s}|^p \diff x &\leq & \frac{2^{p-1}}{a} \int_{B_{ar_1}(x_0)} |\nabla w|^p \diff x + \frac{2^{p-1}}{a} \int_{B_{ar_1}(x_0)} |\nabla u_{\s}-\nabla w|^p \diff x \notag \\
	&\leq & \frac{1}{2}  \int_{B_{r_1}(x_0)} |\nabla w|^p \diff x + \frac{2^{p-1}}{a} \int_{B_{r_1}(x_0)} |\nabla u_{\s}-\nabla w|^p \diff x \notag \\
	&\leq &  \frac{1}{2}  \int_{B_{r_1}(x_0)} |\nabla w|^p \diff x + Cr_1^{1+\alpha(q)} \notag \\
	&\leq & \frac{1}{2}  \int_{B_{r_1}(x_0)} |\nabla u_{\s}|^p \diff x + Cr_1^{1+\alpha(q)}, \notag
	\end{eqnarray}
	where we have used that $w$ minimizes the $p$-energy in $B_{r_1}(x_{0})$ with its own trace and $u_{\s}$ is a competitor. The proof of the lemma follows by dividing the resulting inequality by $r_1$.
\end{proof}

Finally, by iterating the last lemma in a sequence of balls $\{B_{a^{l}r_{1}}(x_{0})\}$, we obtain the following main decay behavior of the $p$-energy under flatness control.

\begin{lemma}\label{Good_decay} Let $p\in (1, +\infty)$, $f \in L^{q}(\Omega)$ with  $q> q_1$, where $q_1$ is defined in (\ref{qrestrict}). Then there exists   $\varepsilon_0 \in (0,1/2)$,  $C=C(p,q_{0},q,\|f\|_q, |\Omega|),\, \overline{r}, \, b>0$ such that the following holds.    Let $\s\subset \overline{\Omega}$ be a closed arcwise connected set. Assume that $0<2r_{0}\leq r_{1}\leq \overline{r}$,  $B_{r_{1}}(x_{0}) \subset \Omega$ and that for all  $r \in [r_{0}, r_{1}]$ there is an affine line $P=P(r)$,  passing through $x_{0}$, such that $d_{H}(\s\cap \overline{B}_{r}(x_{0}), P\cap \overline{B}_{r}(x_{0})) \leq \varepsilon_0 r$. Assume also that $\Sigma \setminus B_{r_1}(x_0)\not = \emptyset$. Then for every $r\in[r_{0}, r_{1}]$,
	\begin{equation}
	\int_{B_{r}(x_{0})}|\nabla u_\Sigma|^{p}\diff x \leq C \bigl(\frac{r}{r_{1}}\big)^{1+b}  \int_{B_{r_{1}}(x_{0})}|\nabla u_\Sigma|^{p}\diff x + Cr^{1+b}. \label{7.24}
	\end{equation}
\end{lemma}
\begin{proof} Let  $a\in (0,1/2)$, $\varepsilon_0 \in (0,a)$  and $C=C(p,q_{0},q,\|f\|_q, |\Omega|)>0$ be the constants given by Lemma~\ref{lem: 7.11P}. Under the assumptions of Lemma~\ref{Good_decay}, we can apply Lemma~\ref{lem: 7.11P} in all the balls $B_{a^lr_1}(x_0)$, $l\in \{0,...,k\}$ with $k$ for which $a^{k+1}r_1<r_0\leq a^{k}r_{1}$. Notice that the definition of $q_1$ and the assumption $q>q_1$ have been made in order to guarantee that $\alpha(q)>0$, where $\alpha(q)$ is defined in \eqref{defalphap}.
			Let us now define 
			$$b= \min\left(\frac{\alpha(q)}{2},  \frac{\ln(3/4)}{\ln(a)} \right)  \;,\quad \bar r =  \left(\frac{1}{4}\right)^{\frac{1}{b}}.$$
			We can easily check that for all $t \in (0, \bar r]$, 
			
			\begin{eqnarray}
			\frac{1}{2}t^{b} + t^{\alpha(q)} \leq (at)^{b},  \label{convex}
			\end{eqnarray}
			because since $0<2 b \leq \alpha(q)$ and $\bar r<1$,
			$$\frac{1}{2}t^{b} + t^{\alpha(q)} \leq \frac{1}{2}t^{b} + \bar r^{b} t^{b} \leq \frac{3}{4}t^{b} \leq (at)^{b}.$$
			
	Let us now define $\Psi(r)=\frac{1}{r}\int_{B_r(x_0)}|\nabla u_{\s}|^p \diff x$
	and  prove by induction  that for all $l\in \{0,...,k\}$, 
	\begin{eqnarray}
	\Psi(a^{l} r_1) &\leq& \frac{1}{2^{l}} \Psi(r_1) + C (a^{l} r_1)^{b}. \label{induction00}
	\end{eqnarray}
Clearly (\ref{induction00}) holds for $l= 0$, assume that (\ref{induction00}) holds for some $l\in \{0,...,k-1\}$. Then applying   Lemma~\ref{lem: 7.11P}, we get 
	$$\Psi(a^{l+1}r_1)\leq \frac{1}{2} \Psi(a^l r_1) + C (a^l r_1)^{\alpha(q)}.$$
	By the induction hypothesis it comes
	$$\Psi(a^{l+1}r_1)\leq \frac{1}{2}\Bigl(\frac{1}{2^{l}}\Psi( r_1)  + C(a^{l}r_{1})^{b}\Bigr)+ C(a^l r_1)^{\alpha(q)},$$
	and thanks to (\ref{convex}), we finally conclude that 
	$$\Psi(a^{l+1}r_1)\leq \frac{1}{2^{l+1}} \Psi( r_1)  +  C(a^{l+1} r_1)^{b},$$
	and \eqref{induction00} is proved. Now let $r\in [r_0,r_1]$ and $l\in \{0,...,k\}$ be such that $a^{l+1}r_1< r\leq a^lr_1$. Then 
	\begin{eqnarray}
	\Psi(r) \leq   \frac{1}{a}\Psi(a^l r_1)& \leq &  \frac{1}{a}\frac{1}{2^l}\Psi(r_1) + \frac{C}{a}(a^{l}r_1)^{b} \notag \\
	&\leq &  \frac{2}{a} (a^{l+1})^{b} \Psi(r_1)  + C' (a^{l+1}r_{1})^{b} \notag \\
	&\leq & C'' \left(\frac{r}{r_1}\right)^{b} \Psi(r_1) + C'' r^{b} \notag
	\end{eqnarray}
	where $C''=C''(a,p,q_{0},q,\|f\|_{q},|\Omega|)>0$. Notice that although $C''$ depends on $a$, however, for every $p\in (1,+\infty)$ we can fix $a$, and thus, we can assume that $C''$ can depend only on $p, q_{0}, q, \|f\|_{q}$ and $|\Omega|$.  This achieves the proof.
\end{proof}

\section{Absence of loops}

\begin{theorem} \label{thm: 8.1} Let $\Omega$ be a bounded open set in $\mathbb{R}^{2}$ and $p \in (1, +\infty)$, and let $f \in L^{q}(\Omega)$  with $q>q_1$, defined in \eqref{qrestrict}. Then every solution $\s$ of the penalized Problem~\ref{problemMain} contains no closed curves (i.e., homeomorphic images of a circumference $S^{1}$), hence $\mathbb{R}^{2}\backslash \s$ is connected.
\end{theorem}

The next lemma which was also used several times earlier in the literature, will be used in the proof of Theorem~\ref{thm: 8.1}.
\begin{lemma} \label{lem: 8.2} Let $\s$ be a closed connected set in $\mathbb{R}^{2}$, containing a simple closed curve $\Gamma$ and such that $\mathcal{H}^{1}(\s)<+\infty$. Then $\mathcal{H}^{1}$-a.e. point $x\in \Gamma$ is such that 
\begin{itemize}
\item ``noncut" : there is a sequence of (relatively) open sets $D_{n}\subset \s$ satisfying
\begin{enumerate}[label=(\roman*)]
\item $x \in D_{n}$ for all sufficiently large $n$;
\item $\s \backslash D_{n}$ are connected for all $n$;
\item $\diam D_{n} \searrow 0$ as $n\to +\infty$;
\item $D_{n}$ are connected for all $n$.
\end{enumerate}
\item ``flatness" : there exists the `` tangent" line $P$ to $\s$ at x in the sense that $x \in P$ and 
\[\frac{d_{H}(\s \cap \overline{B}_{r}(x), P\cap \overline{B}_{r}(x))}{r} \underset{r\to 0+}{\to 0}. \]
\end{itemize}
\end{lemma}

\begin{proof} By \cite[Lemma 5.6]{Paolini-Stepanov}, $\mathcal{H}^{1}$-a.e. point $x\in \Gamma$ is a noncut point for $\s$ (i.e., a point such that $\s\backslash \{x\}$ is connected). Then, by \cite[Lemma 6.1]{Buttazzo-Stepanov}, it follows that for every noncut point there are connected neighborhoods that can be cut leaving the set connected, so $(i)$-$(iv)$ are satisfied for a suitable sequence $D_{n}$. For the proof of the second assertion we refer to \cite[Proposition 2.2]{Approx}.
\end{proof}

\begin{proof}[Proof of Theorem \ref{thm: 8.1}] Assume by contradiction that for some $\lambda>0$ a minimizer $\s$ of $\mathcal{F}_{\lambda,p}$ over closed connected proper subsets of $\overline{\Omega}$ contains a simple closed curve $\Gamma\subset \s$. Notice that there is no a relatively open subset in $\Sigma$ contained in both $\Gamma$ and $\partial \Omega$, because otherwise by Lemma~\ref{lem: 8.2} there would be a relatively open subset $D\subset \Sigma$ such that $D \subset \partial \Omega$ and $\Sigma\backslash D$ would remain connected but observing that in this case $u_{\Sigma\backslash D}=u_{\Sigma}$ and $\mathcal{H}^{1}(D)>0$, we would obtain a contradiction with the optimality of $\Sigma$. Thus  by Lemma~\ref{lem: 8.2}, there is a point $x_{0} \in \Gamma \cap \Omega$ which is a noncut point for $\Sigma$ and such that $\s$ is differentiable at $x_{0}$. Therefore there exist the sets $D_{n} \subset \s$ and an affine line $P$ as in Lemma~\ref{lem: 8.2}. Let $\varepsilon_{0}, C, \overline{r}, b$   be the constants of Lemma~\ref{Good_decay} and let $B_{t_{0}}(x_{0})\subset \Omega$ with $t_{0}<\min\{\overline{r}, \diam(\Sigma)/2\}$.  We denote $r_{n}:=\diam D_{n}$ so that $D_{n}\subset \s\cap \overline{B}_{r_{n}}(x_{0})$. The flatness of $\s$ at $x_{0}$ implies that for any given $\varepsilon>0$ there is $t=t(\varepsilon) \in (0, t_{0}]$ such that 
\begin{equation}
d_{H}(\s\cap \overline{B}_{r}(x_{0}), P\cap \overline{B}_{r}(x_{0}))\leq \varepsilon r \,\ \text{for all}\,\ r\in (0,t]. \label{8.8}
\end{equation}
For every $n$ let us define $\s_{n}:=\s\backslash D_{n}$, which by Lemma~\ref{lem: 8.2} remains closed and connected. 
 We fix  $\varepsilon = \frac{\varepsilon_0}{2}$. Our aim is to  apply  Lemma~\ref{Good_decay} to $\Sigma_n$ but we have to control the Hausdorff distance between $\Sigma_n$ and a line in $\overline{B}_{r}(x_{0})$. We already know that $\Sigma$ is $\varepsilon  r$-close to $P$ in $\overline{B}_r(x_0)$ for all $r\leq t$.  Thus, if 
$r_n \leq \frac{\varepsilon_{0} r}{2} $ we can compute
\begin{align*}
& \, d_{H}(\Sigma_{n}\cap \overline{B}_{r}(x_{0}), P\cap \overline{B}_{r}(x_{0})) \\
\;\, \leq &\, d_{H}(\Sigma_{n}\cap \overline{B}_{r}(x_{0}), \s\cap \overline{B}_{r}(x_{0}))+ d_{H}(\Sigma\cap \overline{B}_{r}(x_{0}), P\cap \overline{B}_{r}(x_{0})) \\
\;\, \leq &\, r_{n} +  \frac{\varepsilon_{0}r}{2} \leq \frac{\varepsilon_{0}r}{2} +\frac{\varepsilon_{0}r}{2} =  \varepsilon_{0} r.
\end{align*}
We can therefore apply Lemma~\ref{Good_decay} to $\Sigma_n$, for the interval $ [\frac{2 r_{n}}{\varepsilon_{0}}, t]$, provided that $\frac{2r_{n}}{\varepsilon_{0}}\leq \frac{t}{2}$, which says that 
\begin{equation*} 
\int_{B_{r}(x_{0})}|\nabla u_{\Sigma_n}|^{p}\diff x \leq C \bigl(\frac{r}{t}\big)^{1+b}  \int_{B_{t}(x_{0})}|\nabla u_{\Sigma_n}|^{p}\diff x + Cr^{1+b}\,\ \text{for every}\,\  r\in\biggl[\frac{2 r_{n}}{\varepsilon_0}, t\biggr]. 
\end{equation*}
Hereinafter in this proof, $C$ denotes a positive constant that does not depend on $r_{n}$ and can be different
from line to line. Next, for $r=\frac{2 r_{n}}{\varepsilon_0}$, using also \eqref{2.4} it comes
\begin{equation}
\int_{B_{\frac{2 r_{n}}{\varepsilon_0}}(x_{0})}|\nabla u_{\Sigma_n}|^{p}\diff x \leq Cr_n^{1+b}, \notag
\end{equation}
for all $n$ such that $\frac{2 r_{n}}{\varepsilon_{0}}\leq \frac{t}{2}$. Remember that the exponent $b$ given by Lemma~\ref{Good_decay} is positive provided $q>q_1$, where 
\begin{eqnarray}
q_1=
\left\{
\begin{array}{cl}
\frac{2p}{2p-1} & \text{ if } 2 \leq p<+\infty \\
\frac{2p}{3p-3} & \text{ if } 1<p < 2
\end{array}
\right.   \notag
\end{eqnarray}
which is one of our assumptions.

Now by the fact that $\s$ is a minimizer and $\s_{n}$ is a competitor for $\s$ we get the following
\begin{align*} 
0&\leq \mathcal{F}_{\lambda,p}({\s}_{n})- \mathcal{F}_{\lambda,p}(\s)  \leq E_{p}(u_{\s})-E_{p}(u_{\s_{n}}) - \lambda r_{n}\\
& \leq C \int_{B_{2r_{n}}(x_{0})}|\nabla u_{\s_{n}}|^{p}\diff x + Cr^{2+p^{\prime}-\frac{2p'}{q}}_{n} -\lambda r_{n} \,\ (\text{by Corollary~\ref{cor: 7.3}})\\
&\leq C \int_{B_{\frac{2r_{n}}{\varepsilon_{0}}}(x_{0})}|\nabla u_{\Sigma_{n}}|^{p}\diff x  + Cr^{2+p^{\prime}-\frac{2p'}{q}}_{n} -\lambda r_{n}  \\
&\leq  Cr_n^{1+b} +Cr^{2+p^{\prime}-\frac{2p'}{q}}_{n} -\lambda r_{n}. 
\end{align*}
Notice that 

$$2+p'-\frac{2p'}{q}>1 \Leftrightarrow  q> \frac{2p}{2p-1},$$
which is always true under the assumption $q>q_1$. Therefore, letting $n$ tend to $+\infty$, we arrive to a contradiction.

This proves that every minimizer  $\s$ of Poblem 1.1 contains no closed curves. In order to prove the last assertion in Theorem~\ref{thm: 8.1}, we use theorem II.5 of \cite[§ 61]{topologie}, stating that if $D\subset \mathbb{R}^{2}$ is a bounded connected set with locally connected boundary, then there is a simple closed curve $S\subset \partial D$. If $\mathbb{R}^{2}\backslash \s$ were disconnected, then there would exist a bounded connected component $D$ of $\mathbb{R}^{2}\backslash \Sigma$ such that $\partial D \subset \s$, and hence $\Sigma$ would contain a simple closed curve, contrary to what we proved before.\end{proof}


\section{Proof of a $C^{1,\alpha}$ regularity}
In this section, we shall prove that every solution $\Sigma$ of Problem~\ref{problemMain} is  locally $C^{1,\alpha}$ regular at $\mathcal{H}^{1}$ a.e. point $x \in \Sigma \cap \Omega$.

Throughout this section, $\Omega$ will denote an open bounded subset in $\mathbb{R}^{2}$. Recall that $\mathcal{K}(\Omega)$ is the class of all closed connected proper subsets of $\overline{\Omega}$.

The factor $\lambda$ in the statement of Problem~\ref{problemMain} affects the shape of an optimal set minimizing the functional $\mathcal{F}_{\lambda,p}$ over $\mathcal{K}(\Omega)$, and according to Proposition~\ref{prop: 2.16}, we know that  there exists $\lambda_{0}=\lambda_{0}(p, f, \Omega)>0$  such that if $\lambda \in (0, \lambda_{0}]$, then each minimizer $\s$ of the functional $\mathcal{F}_{\lambda, p}$ over $\mathcal{K}(\Omega)$ has positive one-dimensional Hausdorff measure. Throughout this section, we shall assume that $\lambda=\lambda_{0}=1$ for simplicity. Of course, this is not restrictive regarding to the regularity theory.

\subsection{Control of the defect of minimality when $\Sigma$ is flat}
For any closed set $\Sigma \subset \mathbb{R}^{2}$, any point $x \in \mathbb{R}^{2}$ and any radius $r>0$ we denote by $\beta_{\s}(x,r)$ the flatness of $\s$ in $\overline{B}_{r}(x)$ defined through
\begin{equation*}
\beta_{\s}(x,r)=\inf_{P\ni x} \frac{1}{r}d_{H}(\s\cap \overline{B}_{r}(x), P\cap \overline{B}_{r}(x)),
\end{equation*}
where the infimum is taken over the set of all affine lines $P$ passing through $x$. Notice that if $\beta_{\s}(x,r)<+\infty$, then it is easy to prove that  the infimum above is actually the minimum. Furthermore, it is easy to see that in this case $\beta_{\Sigma}(x,r) \in [0, \sqrt{2}]$ and $\beta_{\Sigma}(x,r)=\sqrt{2}$ if and only if $\Sigma \cap \overline{B}_{r}(x)$ is a point on the circle $\partial B_{r}(x)$.

\begin{prop}\label{Inequality} Let $\s \subset \mathbb{R}^{2}$ be a closed set, $x \in \mathbb{R}^{2}$, $r>0$ and $\kappa \in (0,1)$. If $\beta_{\s}(x, \kappa r)<+\infty$, then 
\begin{equation}
\beta_{\s}(x, \kappa r)\leq \frac{2}{\kappa}\beta_{\s}(x,r). \label{9.1}
\end{equation}
\end{prop}

\begin{proof} Since $\beta_{\Sigma}(x, \kappa r)<+\infty$,  $\beta_{\Sigma}(x, \kappa r)$ and $\beta_{\Sigma}(x,r)$ belong to $[0,\sqrt{2}]$. Notice that if $\beta_{\Sigma}(x,r) \geq \frac{\kappa \sqrt{2}}{2}$, then (\ref{9.1}) becomes trivial. Now let $P$ be an affine line realizing the infimum in the definition of $\beta_{\s}(x,r)$. Then, because $\Sigma \cap \overline{B}_{\kappa r}(x) \subset \Sigma \cap \overline{B}_{r}(x)$, the following inequality holds
\begin{equation}
\max_{y \in \Sigma \cap \overline{B}_{r}(x)} \dist(y, P \cap \overline{B}_{r}(x)) \geq \max_{y \in  \Sigma \cap \overline{B}_{\kappa r}(x)} \dist(y,  P \cap \overline{B}_{\kappa r}(x)). \label{Ineq 7.2}
\end{equation}
Let $x_{0} \in P\cap \overline{B}_{\kappa r}(x)$ be a point such that $$r_{0}:=\dist(x_{0}, \Sigma \cap \overline{B}_{\kappa r}(x))=\max_{y\in P\cap \overline{B}_{\kappa r}(x)} \dist(y, \Sigma \cap \overline{B}_{\kappa r}(x)).$$ We now distinguish two cases.\\
Case 1: $r_{0}=0$. By (\ref{Ineq 7.2}) and by the definition of the Hausdorff distance, it follows that $ d_{H}(\Sigma \cap \overline{B}_{r}(x), P \cap \overline{B}_{r}(x)) \geq d_{H}(\Sigma \cap \overline{B}_{\kappa r}(x), P \cap \overline{B}_{\kappa r}(x))$. Thus
\begin{align*}
\frac{1}{\kappa}\beta_{\Sigma} (x, r) &=\frac{1}{\kappa r} d_{H}(\Sigma \cap \overline{B}_{r}(x), P \cap \overline{B}_{r}(x))\\ &\geq \frac{1}{\kappa r} d_{H}(\Sigma \cap \overline{B}_{\kappa r}(x), P \cap \overline{B}_{\kappa r}(x))\geq  \beta_{\Sigma}(x, \kappa r)
\end{align*}
and therefore in this case (\ref{9.1}) holds. \\
Case 2: $r_{0}>0$. Since $\beta_{\Sigma}(x, \kappa r) <+\infty$, namely $\Sigma \cap \overline{B}_{\kappa r}(x) \neq  \emptyset$, by the definitions of $x_{0}$ and $r_{0}$, we get that $r_{0}\leq |x_{0}-x|+\kappa r$, because $B_{\kappa r}(x) \subset B_{|x_{0}-x|+\kappa r}(x_{0})$. Then, there is a point $x_{1} \in \partial B_{r_{0}}(x_{0}) \cap P\cap \overline{B}_{\kappa r}(x)$, because otherwise $r_{0}$ would be greater than $|x_{0}-x|+\kappa r$. Setting $\widetilde{x}=x_{0}+\frac{1}{2}(x_{1}-x_{0}) \in P\cap \overline{B}_{\kappa r}(x)$, we observe the following: $|\widetilde{x}-x_{0}|=\frac{r_{0}}{2}$ and $B_{\frac{r_{0}}{2}}(\widetilde{x})\subset B_{\kappa r}(x) \cap B_{r_{0}}(x_{0})$. This, again by the definitions of $x_{0}$ and $r_{0}$, implies that $B_{\frac{r_{0}}{2}}(\widetilde{x})\cap \Sigma =\emptyset$ and therefore 
\begin{equation}
\max_{y \in P\cap \overline{B}_{r}(x)} \dist(y, \Sigma \cap \overline{B}_{r}(x)) \geq \frac{r_{0}}{2}. \label{Ineq 7.3}
\end{equation}
By (\ref{Ineq 7.2}), (\ref{Ineq 7.3}) and by the definition of the Hausdorff distance, we deduce the following inequality $$2d_{H}(\Sigma \cap \overline{B}_{r}(x), P\cap \overline{B}_{r}(x)) \geq d_{H}(\Sigma \cap \overline{B}_{\kappa r}(x), P \cap \overline{B}_{\kappa r}(x)),$$ leading to (\ref{9.1}).
\end{proof}

We  now introduce the following definition of the  local energy.

\begin{defn} \textit{Let $\s\in \mathcal{K}(\Omega)$ and let $\tau \in [0, \sqrt{2}]$. For any $x_{0} \in \mathbb{R}^{2}$ and any $r>0$ we define}
\begin{equation}
w^{\tau}_{\s}(x_{0},r)=\sup_{\substack{\s^{\prime} \in \mathcal{K}(\Omega),\, \s^{\prime}\Delta \s \subset \overline{B}_{r}(x_{0}) \\ \mathcal{H}^{1}(\Sigma^{\prime})\leq 100 \mathcal{H}^{1}(\Sigma),\, \beta_{\s^{\prime}}(x_{0},r)\leq \tau}} \frac{1}{r} \int_{B_{r}(x_{0})}|\nabla u_{\s^{\prime}}|^{p}\diff x. \label{9.5}
\end{equation}
\end{defn}

\begin{rem}\label{remark C} Let $\Sigma \subset \overline{\Omega}$ be closed and arcwise connected and let $\Sigma^{\prime}$ be an admissible set for the problem (\ref{9.5}). Assume that $\Sigma \backslash \overline{B}_{r}(x_{0})$ contains a sequence of points $(y_{n})_{n}$ converging to some point $y \in \partial B_{r}(x_{0})$. Then $y\in \Sigma^{\prime}$, since  $\Sigma \backslash \overline{B}_{r}(x_{0})=\Sigma^{\prime} \backslash \overline{B}_{r}(x_{0})$ and $\Sigma^{\prime}$ is closed.
\end{rem}

\begin{rem} \label{rem: 9.3} Assume that $\Sigma \subset \overline{\Omega}$ is closed and arcwise connected,  $\tau\in (0,\sqrt{2}]$ and $\beta_{\s}(x_{0}, r_{1})\leq\varepsilon$ with $\varepsilon \in (0, \frac{\tau}{2}]$. Then, for all   $r\in [\frac{2\varepsilon r_{1}}{\tau},r_{1}]$, there is a solution for problem~(\ref{9.5}). Indeed, using (\ref{9.1}), we deduce that $ \beta_{\s}(x_{0}, r) \leq \tau$ for all $r\in [\frac{2\varepsilon r_{1}}{\tau},r_{1}]$ and hence $\s$ is an admissible set in (\ref{9.5}). Thus, due to Proposition~\ref{prop: 2.12}, $w^{\tau}_{\Sigma}(x_{0},r)\in [0,+\infty)$. We can then conclude by use of the direct method in the  Calculus of Variations, standard compactness results and the fact that $\mathcal{H}^{1}$ is lower semicontinuous with respect to the topology generated by the Hausdorff distance.
\end{rem}

In order to establish a decay for $w^{\tau}_{\Sigma}$, we need the following geometrical result.

\begin{prop} \label{prop: 9.4} Let $\Sigma \subset \overline{\Omega}$ be closed and arcwise connected, $x \in \overline{\Om}$ and  $\tau \in (0,\frac{1}{2}]$, and let $\beta_{\Sigma}(x,r_{1})\leq \varepsilon$ for some $\varepsilon \in (0, \frac{\tau}{2}]$. In addition, assume that $\Sigma \backslash \overline{B}_{r_{1}}(x)\neq \emptyset$. If $r \in [\frac{2\varepsilon r_{1}}{\tau}, r_{1}]$, then for any closed arcwise connected set $\Sigma^{\prime}\subset \overline{\Omega}$ such that $\s^{\prime} \Delta \s \subset \overline{B}_{r}(x)$ and $\beta_{\s^{\prime}}(x, r) \leq \tau$ it holds
\begin{enumerate}[label=(\roman*)]
\item 
\begin{equation}
\beta_{\s^{\prime}}(x, r_{1}) \leq \frac{5\tau r}{r_{1}} + \varepsilon. \label{9.7}
\end{equation}
\item 
\begin{equation}
\beta_{\s^{\prime}}(x, s) \leq 6\tau\,\ \text{for all}\,\ s \in [r, r_{1}]. \label{9.8}
\end{equation}
\end{enumerate}
\end{prop}
\begin{proof} Every ball in this proof is centered at $x$.  Using  (\ref{9.1}), we deduce that 
\begin{equation}
\beta_{\s}(x,t) \leq \tau \,\ \text{for all}\,\ t \in [2\varepsilon r_{1}/\tau, r_{1}].\label{9.9}
\end{equation}
Let $P_{1}, \, P$ and $P^{\prime}$ realize the infimum, respectively, in the definitions of $\beta_{\Sigma}(x,r_{1}),\, \beta_{\Sigma}(x,r)$ and $\beta_{\Sigma^{\prime}}(x,r)$. By (\ref{9.9}),
\begin{equation}
d_{H}(\s \cap \overline{B}_{r}, P\cap \overline{B}_{r}) \leq \tau r. \label{9.10}
\end{equation}
On the other hand,
\begin{align*}
d_{H}(\Sigma^{\prime} \cap \overline{B}_{r_{1}}, P_{1}\cap \overline{B}_{r_{1}}) & \leq d_{H}(\Sigma^{\prime} \cap \overline{B}_{r_{1}}, \s \cap \overline{B}_{r_{1}}) + d_{H}(\s \cap \overline{B}_{r_{1}}, P_{1}\cap \overline{B}_{r_{1}}) \\
& \leq d_{H}(\Sigma^{\prime} \cap \overline{B}_{r}, \s \cap \overline{B}_{r}) + \varepsilon r_{1}, \numberthis \label{9.11}
\end{align*}
where the latter inequality comes because $\Sigma^{\prime} \Delta \s \subset \overline{B}_{r}$ and $\beta_{\s}(x,r_{1})\leq \varepsilon$. In addition, 
\begin{align*}
d_{H}(\Sigma^{\prime} \cap \overline{B}_{r}, \s \cap \overline{B}_{r}) & \leq d_{H}(\Sigma^{\prime} \cap \overline{B}_{r}, P^{\prime} \cap \overline{B}_{r}) + d_{H}(P \cap \overline{B}_{r}, P^{\prime} \cap \overline{B}_{r}) \\
& \qquad \qquad + d_{H}(\s \cap \overline{B}_{r}, P \cap \overline{B}_{r})\\
&\leq 2\tau r + d_{H}(P \cap \overline{B}_{r}, P^{\prime} \cap \overline{B}_{r}),\numberthis \label{Eq. 7.15}
\end{align*}
where we have used (\ref{9.10}) and the assumption $\beta_{\Sigma^{\prime}}(x,r) \leq \tau$. Notice that, since $\Sigma \cap B_{r}\neq \emptyset$, $\Sigma \backslash \overline{B}_{r_{1}} \neq \emptyset$ and $\Sigma$ is arcwise connected, there is a sequence $(x_{n})_{n}$ with $x_{n} \in \Sigma\backslash \overline{B}_{r}$ converging to some point $y \in \partial B_{r}$. By Remark~\ref{remark C}, $y\in \Sigma^{\prime}\cap \partial B_{r}$ and, defining $$W:=\partial B_{r}\cap \{z: \dist(z, P)\leq \beta_{\Sigma}(x,r) r\}\,\ \text{and}\,\ W^{\prime}:=\partial B_{r} \cap \{z: \dist(z, P^{\prime}) \leq \beta_{\Sigma^{\prime}}(x,r) r\},$$ it holds $y\in W\cap W^{\prime}$. This implies the following estimate
\begin{align*}
d_{H}(P\cap \overline{B}_{r}, P^{\prime} \cap \overline{B}_{r})& \leq (\arcsin(\beta_{\Sigma}(x,r))+ \arcsin(\beta_{\Sigma^{\prime}}(x,r)))r \\ & \leq 2\arcsin(\tau)r \leq 3\tau r, \numberthis  \label{9.12}
\end{align*}
where we have used (\ref{9.9}), the assumption $\beta_{\Sigma^{\prime}}(x,r) \leq \tau$ and the fact that $\arcsin(t)\leq \frac{3t}{2}$ if $t\in [0,\frac{1}{2}]$.
By (\ref{Eq. 7.15}) and (\ref{9.12}),
\[
d_{H}(\Sigma^{\prime} \cap \overline{B}_{r}, \s \cap \overline{B}_{r}) \leq 5\tau r. \numberthis\label{9.13}
\]
This together with (\ref{9.11}) gives the following
\begin{align*}
d_{H}(\Sigma^{\prime} \cap \overline{B}_{r_{1}}, P_{1} \cap \overline{B}_{r_{1}}) & \leq 5\tau r + \varepsilon r_{1}.
\end{align*}
Thus, we have proved $(i)$. Now let $s \in [r, r_{1}]$ and let $P_{s}$ be the line realizing the infimum in the definition of $\beta_{\s}(x,s)$. As in the proof of $(i)$ we get
\begin{align*}
d_{H}(\Sigma^{\prime} \cap \overline{B}_{s}, P_{s} \cap \overline{B}_{s}) & \leq d_{H}(\Sigma^{\prime} \cap \overline{B}_{s}, \s \cap \overline{B}_{s}) + d_{H}(\s \cap \overline{B}_{s}, P_{s} \cap \overline{B}_{s}) \\
& \leq d_{H}(\Sigma^{\prime} \cap \overline{B}_{r}, \s \cap \overline{B}_{r}) + d_{H}(\s \cap \overline{B}_{s}, P_{s} \cap \overline{B}_{s}).
\end{align*}
Then, by (\ref{9.9}) and (\ref{9.13}), we deduce that
\begin{align*}
d_{H}(\Sigma^{\prime} \cap \overline{B}_{s}, P_{s} \cap \overline{B}_{s}) & \leq 5\tau r + \tau s\\
& \leq 6 \tau s,
\end{align*}
thus concluding the proof.
\end{proof}

In the next proposition we establish a decay for $w^{\tau}_{\Sigma}(x,\cdot)$, provided that $\beta_{\Sigma}(x,\cdot)$ is small enough.

\begin{prop} \label{prop: 9.6} Let $p\in (1,+\infty)$ and $f \in L^{q}(\Omega)$ with $q> q_1$, where $q_1$ is defined in~\eqref{qrestrict}. Let $\varepsilon_0\in(0,1/2),\, b,\, \overline{r}, \, C>0$ be the constants of Lemma~\ref{Good_decay}. Let $\s \subset \overline{\Omega}$ be closed and connected, $\mathcal{H}^{1}(\Sigma)<+\infty$ and let $B_{r_{1}}(x_{0})\subset \Omega$ with $r_{1}\in (0, \min\{\overline{r},\diam(\Sigma)/2\})$. Suppose that $\tau \in (0, \frac{\varepsilon_{0}}{6}]$ and 
\begin{equation*}
\beta_{\s}(x_{0}, r_{1}) \leq \varepsilon
\end{equation*}
for some $\varepsilon \in (0, \frac{\tau}{20}]$. Then, for all $r \in [\frac{2\varepsilon r_{1}}{\tau}, \frac{r_{1}}{10}]$,
\begin{equation}
w^{\tau}_{\s}(x_{0}, r) \leq C \left(\frac{r}{r_{1}}\right)^{b} w^{\tau}_{\s}(x_{0}, r_{1}) + Cr^{b}. \label{9.14}
\end{equation}

\end{prop}

\begin{proof} Every ball in this proof is centered at $x_{0}$. By Remark~\ref{remark 3.1}, $\Sigma$ is arcwise connected. From Remark~\ref{rem: 9.3} it follows that there is $\Sigma_{r}\subset \overline{\Omega}$ realizing the supremum in the definition of $w^{\tau}_{\Sigma}(x_{0},r)$ which, by Remark~\ref{remark 3.1}, is arcwise connected. In addition, according to Proposition~\ref{prop: 9.4},
\[
\beta_{\Sigma_{r}}(x_{0}, r_{1}) \leq \tau \,\ \text{and} \,\ \beta_{\Sigma_{r}}(x_{0}, s) \leq \varepsilon_{0} \,\ \text{for all}\,\ s\in [r, r_{1}].
\]
This allows us to apply Lemma~\ref{Good_decay} to $u_{\s_r}$, which yields

\begin{align*}
w^{\tau}_{\s}(x_{0}, r) &=\frac{1}{r}\int_{B_{r}}|\nabla u_{\Sigma_{r}}|^{p}\diff x \\
&\leq C \bigl(\frac{r}{r_{1}}\bigr)^{b}\frac{1}{r_1}\int_{B_{r_{1}}}|\nabla u_{\Sigma_{r}}|^{p}\diff x + Cr^{b}\\ 
& \leq  C  \bigl(\frac{r}{r_{1}}\bigr)^{b} w^{\tau}_{\s}(x_{0},r_{1}) + Cr^{b}.
\end{align*}
 Notice that  to obtain the last inequality we have used the definition of $w^{\tau}_{\s}(x_{0},r_{1})$ and the fact that  $\beta_{\Sigma_{r}}(x_{0}, r_{1}) \leq \tau$. 
\end{proof}

Now we control a defect of minimality via $w^{\tau}_{\Sigma}$.

\begin{prop} \label{Defofminimality} Let $p\in (1,+\infty)$ and $f \in L^{q}(\Omega)$ with $q>q_{1}$, where $q_{1}$ is defined in~(\ref{qrestrict}), and let $\varepsilon_{0}\in (0,1/2),\ b,\, \overline{r}>0$ be the constants of Lemma~\ref{Good_decay}. Let $\s \subset \overline{\Omega}$ be closed and connected, $\mathcal{H}^{1}(\Sigma)<+\infty$ and $B_{r_{1}}(x_{0})\subset \Omega$ with $r_{1} \in (0, \min\{\overline{r},\diam(\s)/2\})$. Suppose that $\tau \in (0, \frac{\varepsilon_{0}}{6}]$  and 
			\begin{equation*}
			\beta_{\s}(x_{0}, r_{1}) \leq \varepsilon
			\end{equation*}
			for some $\varepsilon \in (0, \frac{\tau}{20}]$. Then there is a constant $C>0$, possibly depending only on $p, q_{0}, q,  \|f\|_{q}, |\Omega|$, such that if $r \in [\frac{2\varepsilon r_{1}}{\tau}, \frac{r_{1}}{10}]$, then for any closed connected set $\Sigma^{\prime} \subset \overline{\Omega}$ satisfying $\Sigma^{\prime} \Delta \Sigma \subset \overline{B}_{r}(x_{0})$,\,  $\mathcal{H}^{1}(\Sigma^{\prime}) \leq 100 \mathcal{H}^{1}(\Sigma)$ and $\beta_{\Sigma^{\prime}}(x_{0},r) \leq \tau$, 
			\begin{equation}
			E_{p}(u_{\s})-E_{p}(u_{\Sigma^{\prime}})\leq Cr \Bigl(\frac{r}{r_{1}}\Bigr)^{b} w^{\tau}_{\s}(x_{0},r_{1}) + Cr^{1+b}. \label{9.16}
			\end{equation}

\end{prop}
\begin{proof} Every ball in this proof is centered at $x_{0}$. By Remark~\ref{remark 3.1}, $\Sigma$ and $\Sigma^{\prime}$ are arcwise connected and by Corollary~\ref{cor: 7.3},
	\begin{equation}
	E_{p}(u_{\s})-E_{p}(u_{\s^{\prime}})\leq C\int_{B_{2r}}|\nabla u_{\s^{\prime}}|^{p}\diff x + Cr^{2+p^{\prime}-\frac{2p^{\prime}}{q}}, \label{9.17}
	\end{equation}
	where $C= C(p, q_{0}, q, \|f\|_{q})>0$. On the other hand, by Proposition~\ref{prop: 9.4}, 
	\[
	\beta_{\Sigma^{\prime}}(x_{0}, r_{1}) \leq \tau \,\ \text{and} \,\ \beta_{\Sigma^{\prime}}(x_{0}, s) \leq \varepsilon_{0} \,\ \text{for all}\,\ s\in [r, r_{1}].
	\]
	This allows to apply Lemma~\ref{Good_decay} to $u_{\Sigma^{\prime}}$ and obtain that
	\begin{equation} \label{Estimate 6.17}
	\int_{B_{2r}}|\nabla u_{\Sigma^{\prime}}|^{p}\diff x \leq C\Bigl(\frac{2r}{r_{1}}\Bigr)^{1+b}\int_{B_{r_{1}}}|\nabla u_{\Sigma^{\prime}}|^{p}\diff x + C(2r)^{1+b},
	\end{equation}
	where $C=C(p, q_{0}, q, \|f\|_{q}, |\Omega|)>0$. Hereinafter in this proof, $C$ denotes a positive constant that can depend only on $p,\, q_{0}$, $q,\, \|f\|_{q},\, |\Omega|$ and can be different from line to line. Using (\ref{9.17}), (\ref{Estimate 6.17}) and the fact that $r^{2+p^{\prime}-\frac{2p^{\prime}}{q}} < r^{1+b}$ (because $r<1$, $b < 1+p^{\prime}-2{p^{\prime}}/q$), we deduce the following chain of estimates
	\begin{align*}
	E_{p}(u_{\s})-E_{p}(u_{\s^{\prime}}) & \leq C\Bigl(\frac{r}{r_{1}}\Bigr)^{1+b}\int_{B_{r_{1}}}|\nabla u_{\Sigma^{\prime}}|^{p}\diff x + Cr^{1+b} \\
	& \leq Cr\Bigl(\frac{r}{r_{1}}\Bigr)^{b}\frac{1}{r_{1}}\int_{B_{r_{1}}}|\nabla u_{\s^{\prime}}|^{p}\diff x +Cr^{1+b} \\
	& \leq Cr\Bigl(\frac{r}{r_{1}}\Bigr)^{b}w^{\tau}_{\Sigma}(x_{0},r_{1}) + Cr^{1+b}, 
	\end{align*}
	where to obtain the last estimate we have used the  definition of $w^{\tau}_{\Sigma}(x_{0},r_{1})$ and the fact that $\beta_{\Sigma^{\prime}}(x_{0},r_{1})\leq \tau$.
\end{proof}

\subsection{Density control}

\begin{prop} \label{prop: 9.8} Let $p\in (1,+\infty)$ and $f \in L^{q}(\Omega)$ with $q > q_1$, where $q_1$ is defined in~(\ref{qrestrict}), and let $\varepsilon_{0} \in (0,1/2),\, b,\,  \overline{r}>0$ be the constants of Lemma~\ref{Good_decay} and $C>0$ be the constant of Proposition~\ref{Defofminimality}.  Let $\s \subset \overline{\Omega}$ be a solution of Problem~\ref{problemMain}, $\tau\in(0, \frac{\varepsilon_{0}}{6}]$, $x_{0} \in \Sigma$ and $0<r_{1}<\min\{\overline{r},\diam(\s)/2 \}$ be such that   $B_{r_{1}}(x_{0}) \subset \Omega$. Then the following assertions hold. 
\begin{enumerate}[label=(\roman*)]
\item If 
\begin{equation} \label{9.18}
\beta_{\s}(x_{0}, r_{1})\leq \varepsilon
\end{equation}
for some $\varepsilon \in (0, \frac{\tau^{2}}{400}]$,  then for all $r \in [\frac{\tau r_{1}}{4}, \frac{r_{1}}{10}]$, 
\begin{equation} \label{9.19}
\mathcal{H}^{1}(\s \cap B_{r}(x_{0})) \leq 2r + 5\beta_{\Sigma}(x_{0}, r) r+  Cr\bigl(\frac{r}{r_{1}}\bigr) ^{b}w^{\tau}_{\s}(x_{0}, r_{1})+ C r^{1+b}.
\end{equation}
\item Assume, in addition, that the estimate
\begin{equation} 
w^{\tau}_{\Sigma}(x_{0}, r_{1}) + r_1^{b} \leq \frac {\tau}{300C} \label{9.20}
\end{equation}
is valid. Then 
\begin{equation}
\mathcal{H}^{1}(\{s \in [\tau r_{1}/4, \tau r_{1}/2]: \# \s \cap \partial B_{s}(x_{0})= 2 \}) > \tau r_{1}/5. \label{9.21}
\end{equation}
\item Let $(\ref{9.18})$ and $(\ref{9.20})$ hold and $r \in [\tau r_{1}/4, \tau r_{1}/2]$ be such that $\# \s \cap \partial B_{r}(x_{0})=2$. Then 
\begin{enumerate}[label=(iii-\arabic*)]
\item the two points of $\s \cap \partial B_{r}(x_{0})$ belong to two different connected components of $\partial B_{r}(x_{0})\cap \{y: dist(y,P_{0})\leq \beta_{\s}(x_{0},r)r\}$, where $P_{0}$ is a line realizing the infimum in the definition of $\beta_{\s}(x_{0},r)$.
\item $\s \cap \overline{B}_{r}(x_{0})$ is arcwise connected.
\item If $\{z_{1}, z_{2}\}= \s \cap \partial  B_{r}(x_{0})$, then 
\begin{equation} \label{9.22}
\begin{split}
\mathcal{H}^{1}(\s\cap B_{r}(x_{0}))\leq |z_{2} -z_{1}| +  Cr\bigl(\frac{r}{r_{1}}\bigr)^{b}w^{\tau}_{\Sigma}(x_{0}, r_{1})+Cr^{1+b}.
\end{split}
\end{equation} 
\end{enumerate}
\end{enumerate}
\end{prop}

\begin{rem}\label{rem: 9.9}  Following \cite{opt}, if the situation of item $(iii$-1) occurs, we say that the two points lie ``on both sides''.
\end{rem}

\begin{proof} Step 1. We first prove $(i)$.   By (\ref{9.1}) and (\ref{9.18}), for all $r \in [\frac{\tau r_{1}}{4}, \frac{r_{1}}{10}]$,
\begin{equation}
\beta_{\Sigma}(x_{0},r) \leq \frac{8}{\tau}\beta_{\Sigma}(x_{0}, r_{1})\leq \frac{\tau}{50}. \label{9.23}
\end{equation}
Fix an arbitrary $r \in [\frac{\tau r_{1}}{4}, \frac{r_{1}}{10}]$. Let $P_{0}$ realize the infimum in the definition of $\beta_{\s}(x_{0}, r)$ and let $\xi_{1}$ and $\xi_{2}$ be the two points of $\partial B_{r}(x_{0}) \cap P_{0}$. Define $W$ and $\Sigma^{\prime}$ by 
\[
W:=\partial B_{r}(x_{0}) \cap \{y: dist(y, P_{0})\leq \beta_{\s}(x_{0}, r)r \}, \,\ \,\  \Sigma^{\prime}:=(\s \backslash B_{r}(x_{0}))\cup W \cup [\xi_{1}, \xi_{2}]. 
\] 
Then, $\s^{\prime}\in \mathcal{K}(\Omega),\, \Sigma \Delta \Sigma^{\prime} \subset \overline{B}_{r}(x_{0})$ and from (\ref{9.23}) it follows that  $\beta_{\Sigma^{\prime}}(x_{0}, r) \leq \frac{\tau}{50}.$ Furthermore, since $\Sigma$ is arcwise connected, compact and $r<\diam(\Sigma)/20$, it follows that $\mathcal{H}^{1}(\Sigma)>  20 r$ and then $\mathcal{H}^{1}(\Sigma^{\prime})<100 \mathcal{H}^{1}(\Sigma)$. Since $\Sigma^{\prime}$ is a competitor, 
\begin{equation*}
\mathcal{H}^{1}(\s) \leq \mathcal{H}^{1}(\Sigma^{\prime}) + E_{p}(u_{\s}) - E_{p}(u_{\Sigma^{\prime}})
\end{equation*}
and then, using Proposition~\ref{Defofminimality}, we get
\begin{eqnarray}
\mathcal{H}^{1}(\s \cap B_{r}(x_{0})) &\leq&  2r + \mathcal{H}^{1}(W) + E_{p}(u_{\s}) - E_{p}(u_{\s^{\prime}}) \notag \\
& \leq& 2r +\mathcal{H}^{1}(W) + Cr\bigl(\frac{r}{r_{1}}\bigr)^{b}w^{\tau}_{\Sigma}(x_{0}, r_{1}) + C r^{1+b}. \label{final1}
\end{eqnarray}
On the other hand, since $\arcsin^{\prime}(t)\leq\frac{10}{\sqrt{99}}$ for all $t \in [0,\frac{1}{10}]$ and by (\ref{9.23}), $\beta_{\Sigma}(x_{0}, r)<\frac{1}{10},$
\begin{equation}
\mathcal{H}^{1}(W) \leq 4r\arcsin(\beta_{\Sigma}(x_{0}, r))\leq 5r \beta_{\s}(x_{0}, r). \label{Estimate 7.29}
\end{equation}
Combining \eqref{final1} and (\ref{Estimate 7.29}), we deduce $(i)$. \\
Step 2. We prove now $(ii)$. Let us consider the next three sets
\begin{equation*}
\begin{split}
E_{1} & := \{ s \in (0, \tau r_{1}/2]: \# \s \cap \partial B_{s}(x_{0}) = 1 \}, \,\ E_{2}:= \{ s \in (0, \tau r_{1}/2] : \# \s \cap \partial B_{s}(x_{0}) = 2 \}, \\
E_{3} & := \{s \in (0, \tau r_{1}/2]: \# \s \cap \partial B_{s} (x_{0}) \geq 3 \}.
\end{split}
\end{equation*}
We claim that  either $E_{1}=\emptyset$ or $E_{1}\subset  (0, \tau r_{1}/299).$ For the sake of contradiction, assume that there is  $s \in [\tau r_{1}/299, \tau r_{1}/2]$ such that $\#\s \cap \partial B_{s}(x_{0}) = 1 $. Then the set
\[
\Sigma^{\prime}=\Sigma \backslash B_{s}(x_{0}),
\]
  would be arcwise connected, $\Sigma^{\prime} \Delta \s \subset B_{s}(x_{0})$, $\mathcal{H}^{1}(\Sigma^{\prime})<\mathcal{H}^{1}(\Sigma)$ and
  \begin{equation} \label{conditions}
   \beta_{\Sigma^{\prime}}(x_{0},r_{1})\leq \tau/ 2 + \varepsilon <\tau.
  \end{equation}  
  Since $\Sigma^{\prime}$ is a competitor, $\mathcal{H}^{1}(\Sigma) \leq \mathcal{H}^{1}(\Sigma^{\prime}) + E_{p}(u_{\Sigma})-E_{p}(u_{\Sigma^{\prime}})$. It also holds the estimate $s \leq \mathcal{H}^{1}(\Sigma \cap B_{s}(x_{0}))$, because $s<\diam(\Sigma)/2$, $x_{0}\in \s$ and $\Sigma$ is arcwise connected. Thus
\[
s \leq \mathcal{H}^{1}(\Sigma\cap B_{s}(x_{0})) \leq E_{p}(u_{\Sigma})-E_{p}(u_{\Sigma^{\prime}}) \numberthis.\label{9.26}
\]
Notice that, by assumption, the estimate (\ref{9.16}) holds with $C$, but looking closer at the proof of Proposition~\ref{Defofminimality}, we observe that (\ref{7.5}) in Corollary~\ref{cor: 7.3} also holds with $C$. Then, using (\ref{9.26}), Corollary~\ref{cor: 7.3} and the fact that $r_{1}^{2+p^{\prime}-\frac{2p^{\prime}}{q}} < r_{1}^{1+b}$ (because $r_{1}<1$ and $b<1+p^{\prime}-\frac{2p^{\prime}}{q}$), we obtain the following chain of estimates
\begin{align*}
s \leq \mathcal{H}^{1}(\Sigma\cap B_{s}(x_{0})) &\leq E_{p}(u_{\Sigma})-E_{p}(u_{\Sigma^{\prime}}) \leq C \int_{B_{2s}(x_{0})}|\nabla u_{\Sigma^{\prime}}|^{p}\diff x + C s^{2+p^{\prime}-\frac{2p^{\prime}}{q}}\\
& \leq C \int_{B_{r_{1}}(x_{0})}|\nabla u_{\Sigma^{\prime}}|^{p}\diff x + C r^{1+b}_{1} \\
& \leq C r_{1} w^{\tau}_{\s}(x_{0}, r_{1}) + C r_{1} ^{1+b} \,\ \text{(by (\ref{conditions}) and the definition of $w^{\tau}_{\s}(x_{0}, r_{1})$)}\\
& \leq \tau r_{1}/300\,\ \text{(by (\ref{9.20}))},
\end{align*}
that leads to a contradiction with the fact that  $s\geq\tau r_{1}/299$. Thus, either $E_{1}=\emptyset$, or
\begin{equation}
E_{1}\subset (0,\tau r_{1}/299).\label{9.27}
\end{equation}
Next, using Eilenberg inequality (see \cite[2.10.25]{Federer}), we obtain
\begin{equation}
\mathcal{H}^{1}(\Sigma \cap B_{\tau r_{1}/2}(x_{0})) \geq \int^{\tau r_{1}/2}_{0}  \# \s \cap \partial B_{s}(x_{0}) \diff s.\label{9.28}
\end{equation}
On the other hand, using (\ref{9.19}) with $r=\tau r_{1}/2$, (\ref{9.20}), the fact that 
\[
\beta_{\Sigma}(x_{0}, \tau r_{1}/2) \leq 4\varepsilon /\tau \leq \tau/100\] and the fact that $\tau\leq \varepsilon_{0}/6$, we get
\begin{equation}
\mathcal{H}^{1}(\Sigma \cap B_{\tau r_{1}/2}(x_{0})) < \tau r_{1} + \tau r_{1}/150.\label{9.29}
\end{equation}
Then, using (\ref{9.27})-(\ref{9.29}), we obtain
\begin{align*}
\tau r_{1} + \tau r_{1} /150 & >\mathcal{H}^{1} (E_{1}) + 2 \mathcal{H}^{1} (E_{2}) + 3 \mathcal{H}^{1}(E_{3}) \\
& = \mathcal{H}^{1}(E_{1}) + 2 (\tau r_{1}/2 - \mathcal{H}^{1}(E_{1}) - \mathcal{H}^{1}(E_{3})) + 3 \mathcal{H}^{1}(E_{3})\\
& = - \mathcal{H}^{1}(E_{1}) + \tau r_{1} + \mathcal{H}^{1}(E_{3}) \\
& \geq   -\tau r_{1}/299+ \tau r_{1} + \mathcal{H}^{1}(E_{3}),
\end{align*}
this yields
\begin{equation}
\mathcal{H}^{1}(E_{3}) <  \tau r_{1}/99.\label{9.30}
\end{equation} Using (\ref{9.27}) and (\ref{9.30}), we deduce that
\begin{equation*}
\mathcal{H}^{1}(E_{2}\cap [\tau r_{1}/4, \tau r_{1}/2])  > \tau r_{1}/5,
\end{equation*}
thereby proving $(ii)$.
\\
Step 3. We prove $(iii)$. Let $r \in E_{2} \cap [\tau r_{1}/4, \tau r_{1}/2]$. Assume that $(iii$-1) does not hold for $r$. Then we can take as a competitor the set 
\begin{equation*}
\Sigma^{\prime}=\s \backslash B_{r}(x_{0}) \cup V,
\end{equation*}
where $V$ is the connected component of $\partial B_{r}(x_{0}) \cap \{y:dist(y, P_{0})\leq \beta_{\s}(x_{0}, r)r\}$ such that $\s \cap \partial B_{r}(x_{0}) \subset V$.
So, $\mathcal{H}^{1}(\s \cap B_{r}(x_{0})) \leq \mathcal{H}^{1}(V) + E_{p}(u_{\s}) - E_{p}(u_{\s^{\prime}})$. Arguing as in the proof of the fact that $E_{1}\subset (0, \tau r_{1}/299)$ in Step 2, we deduce that
\begin{equation*}
E_{p}(u_{\s}) - E_{p}(u_{\s^{\prime}})  \leq \tau r_{1}/300. 
\end{equation*}
On the other hand, as in Step 1 we have that 
\[
\mathcal{H}^{1}(V)\leq (5/2)r\beta_{\s}(x_{0},r) \leq \tau^{2} r_{1}/40.
\]
But then 
\[
\mathcal{H}^{1}(\Sigma \cap B_{r}(x_{0})) <  \tau r_{1}/150
\]
that leads to a contradiction because $\mathcal{H}^{1}(\s \cap B_{r}(x_{0})) \geq r \geq \tau r_{1}/4$, therefore $(iii$-1) holds. Next, assume that $\s \cap \overline{B}_{r}(x_{0})$ is not arcwise connected. Then, from \cite[Lemma 5.13]{opt}, it follows that $\s \backslash B_{r}(x_{0})$ is arcwise connected. Thus, taking the set $\s^{\prime}=\s \backslash B_{r}(x_{0})$ as a competitor, we get
\[
\mathcal{H}^{1}(\Sigma \cap B_{r}(x_{0})) \leq \tau r_{1}/300,
\]
that leads to a contradiction with the fact that $\mathcal{H}^{1}(\s \cap B_{r}(x_{0})) \geq \tau r_{1}/4$. So $(iii$-2) holds. Since $\s \cap \partial B_{r}(x_{0})= \{z_{1}, z_{2}\}$, where $z_{1}, z_{2}$ lie ``on both sides" and $[z_{1},z_{2}]$ is sufficiently close, in $\overline{B}_{r}(x_{0})$ and in the Hausdorff distance, to a diameter of $\overline{B}_{r}(x_{0})$, we observe that the set $\s \backslash B_{r}(x_{0}) \cup [z_{1}, z_{2}]$ is a competitor for $\s$ and  (\ref{9.22}) holds. This proves $(iii)$ and concludes the proof.
\end{proof}

\begin{subsection}{Control of the flatness}
We recall the following standard height estimate (see \cite[Lemma 5.14]{opt}), which we shall use so as to establish a control on $\beta_{\s}$ via $w^{\tau}_{\s}$.
\begin{lemma}\label{lem: 9.11} Let $\Gamma$ be an arc in $\overline{B}_{r}(x_{0})$ satisfying $\beta_{\Gamma}(x_{0}, r) \leq 1/10$, and which connects two points $\xi_{1},\, \xi_{2} \in \partial B_{r}(x_{0})$ lying on ``both sides" (as defined in Remark~\ref{rem: 9.9}). Then
\begin{equation}
\max_{y \in \Gamma}\dist(y, [\xi_{1}, \xi_{2}]) \leq (2r (\mathcal{H}^{1}(\Gamma) - |\xi_{2}-\xi_{1}|))^{\frac{1}{2}}.\label{9.31}
\end{equation}
\end{lemma}

In the next proposition we show that if $\beta_{\s}$ and $w_{\s}^{\tau}$ are small enough on some fixed scale, then they stay small on smaller scales.

\begin{prop}\label{prop: 9.13} Let $p \in (1,+\infty)$ and $f \in L^{q}(\Om)$  with $q> q_{1}$, where $q_{1}$ is defined in~(\ref{qrestrict}). Let $\s \subset \overline{\Omega}$ be a solution of Problem~\ref{problemMain}. Then there exist constants $0<\varepsilon_{1}<\varepsilon_{2}$ and $a, b, r_{0}>0$, and a constant $C=C(p, q_{0}, q, \|f\|_{q}, |\Omega|)>0$ such that whenever $x \in \s$ and $0<r<r_{0}$ satisfy ${B}_{r}(x) \subset \Omega$, 
	\begin{equation} 
	w^{\tau}_{\s}(x, r) \leq \varepsilon_{1}\,\ \text{and}\,\ \beta_{\s}(x,r) \leq \varepsilon_{2}\label{9.32}
	\end{equation}
	then
	\begin{enumerate}[label=(\roman*)]
		\item 
		\begin{equation}
				\beta_{\s}(x, ar) \leq C(w^{\tau}_{\s}(x, r))^{\frac{1}{2}} + C r^{\frac{b}{2}}.\label{9.33}
				\end{equation}
				
			\item 
			\begin{equation} \label{Sigma estimate}
			w^{\tau}_{\Sigma}(x, ar)\leq \frac{1}{2} w^{\tau}_{\Sigma}(x,r)+ C(ar)^{b}.
			\end{equation}
		\item For any $n \in \mathbb{N}$,
		\begin{equation}
		w^{\tau}_{\s}(x, a^{n}r)\leq \varepsilon_{1}\,\ \text{and}\,\ \beta_{\s}(x, a^{n}r) \leq \varepsilon_{2}.\label{9.34}
		\end{equation}
	\end{enumerate}
\end{prop}

\begin{proof} Let $\varepsilon_{0} \in (0,1/2),\,b,\, \overline{r}>0$ and $C=C(p, q_{0}, q, \|f\|_{q}, |\Omega|)>0$ be the constants of Lemma~\ref{Good_decay}. Fix $\tau \in (0, \frac{\varepsilon_{0}}{6}]$ and a constant $C_{1}$ such that the estimate (\ref{9.22}) holds with $C_{1}$. Without loss of generality, assume that $C<C_{1}$. We now define 
	\[
	a:=\min\biggr\{\frac{\tau}{4},\Bigl(\frac{1}{2C}\Bigr)^{\frac{1}{b}}\biggl\},\,\ \varepsilon_{2}:=\frac{a \tau}{100},\,\ \varepsilon_{1}:=\biggl(\frac{a \varepsilon_{2}}{50 C_{1}}\biggr)^{2},\,\ C^{\prime}:=\frac{24C_{1}}{a}.
	\]
	Fix $r_{0} \in (0, \min\{\overline{r}, \diam(\Sigma)/2\})$ such that 
	\begin{equation}
			C^{\prime} r_{0}^{\frac{b}{2}} \leq \frac{\varepsilon_{1}}{2} \label{9.35}
			\end{equation}
and hence
	\begin{equation}
			C^{\prime} r^{b}_{0} \leq \frac{\varepsilon_{1}}{2}\label{9.36}
			\end{equation}
because $r_{0}<1$.
	Let us prove $(i)$. Applying Proposition~\ref{prop: 9.8} with $r_{1}=r$ and $\varepsilon=\varepsilon_{2}\leq \frac{\tau^{2}}{400}$, we deduce that there is $s \in [\tau r/4, \tau r/2]$ such that $ \s \cap \partial B_{s}(x)=\{z_{1}, z_{2}\}$, $z_{1}$ and $z_{2}$ lie on ``both sides" (see Remark~\ref{rem: 9.9}). Fix such $s$. Then, by Proposition~\ref{prop: 9.8} $(iii$-3), we get 
			\begin{align*}
			\mathcal{H}^{1}(\s \cap B_{s}(x))& \leq |z_{1}-z_{2}| + C_{1}s\Bigl(\frac{s}{r}\Bigr)^{b} w^{\tau}_{\s}(x, r) + C_{1}s^{1+b} \\
			& := |z_{1} - z_{2}| + L. 
			\end{align*}
	Let $\Gamma \subset \s \cap \overline{B}_{s}(x)$ be an arc connecting $z_{1}$ with $z_{2}$. Then, using Lemma~\ref{lem: 9.11}, we obtain 
	\begin{align*}
	\max_{y \in \Gamma}\dist(y, [z_{1}, z_{2}]) & \leq (2s(\mathcal{H}^{1}(\Gamma)- |z_{1}-z_{2}|))^{\frac{1}{2}} \leq (\tau rL)^{\frac{1}{2}}. 
	\end{align*}
	Since $\s \cap \overline{B}_{s}(x)$ is arcwise connected, $\Sigma \cap \partial B_{s}(x)=\{z_{1},z_{2}\}$ and $\mathcal{H}^{1}(\Gamma)\geq |z_{1}-z_{2}|$, then
	\begin{align*}
	\sup_{y\in \s \cap \overline{B}_{s}(x)\backslash(B_{s}(x) \cap \Gamma)} \dist (y, \Gamma) \leq \mathcal{H}^{1}(\s \cap B_{s}(x) \backslash \Gamma) & \leq \mathcal{H}^{1}(\s \cap B_{s}(x)) - |z_{1} -z_{2}| = L.
	\end{align*}
	Thus
	\begin{equation}
	\max_{y \in \s \cap \overline{B}_{s}(x)} \dist(y, [z_{1}, z_{2}]) \leq (\tau rL)^{\frac{1}{2}}+L.\label{9.37}
	\end{equation}
	Notice that since $\s \cap \overline{B}_{s}(x)$  is arcwise connected and $\s $ escape $\overline{B}_{s}(x)$ either through $z_{1}$ or $z_{2}$, then (\ref{9.37}) yields the following estimate
	\begin{equation}
	d_{H}(\s \cap \overline{B}_{s}(x), [z_{1}, z_{2}]) \leq (\tau rL)^{\frac{1}{2}}+L.\label{9.38}
	\end{equation}
	Let $\widetilde{P}$ be the line passing through $x$ and collinear to $[z_{1},z_{2}]$. Using the fact that $\dist(x, [z_{1},z_{2}])\leq (\tau rL)^{\frac{1}{2}}+L$, we get
	\begin{equation}
	d_{H}([z_{1},z_{2}], \widetilde{P} \cap \overline{B}_{s}(x)) \leq \arcsin(((\tau rL)^{\frac{1}{2}}+L)/s)s < 2((\tau rL)^{\frac{1}{2}}+L), \label{Equation 7.43}
	\end{equation}
	where the latter estimate holds because $((\tau rL)^{\frac{1}{2}}+L)/s<\frac{1}{10}$. Using (\ref{9.38}) together with (\ref{Equation 7.43}), we obtain that
	\begin{equation*}
	d_{H}(\s\cap \overline{B}_{s}(x), \widetilde{P} \cap \overline{B}_{s}(x)) \leq 3((\tau rL)^{\frac{1}{2}}+L)
	\end{equation*}
	and hence $\beta_{\s}(x,s) \leq \frac{3}{s}((\tau rL)^{\frac{1}{2}}+L).$ If $ar=\kappa s$, then $\frac{2}{\kappa}\leq \frac{\tau}{a}$ because $s \leq \frac{\tau r}{2}$ and, thanks to (\ref{9.1}),
			\begin{align*}
			\beta_{\s}(x, ar) & \leq \frac{\tau}{a} \beta_{\s}(x,s) \leq \frac{12}{a r}((\tau rL)^{\frac{1}{2}}+L). \numberthis\label{9.39}
			\end{align*}
	On the other hand,
	\begin{align*}
			(\tau rL)^{\frac{1}{2}} \leq (C_{1} r^{2}w^{\tau}_{\s}(x,r) + C_{1} r^{2+b})^{\frac{1}{2}} \leq C_{1} r (w^{\tau}_{\s}(x, r))^{\frac{1}{2}} + C_{1} r^{1+\frac{b}{2}} \numberthis \label{9.40}
			\end{align*}
	and, moreover, 
	\begin{align*}
	L= C_{1}s\Bigl(\frac{s}{r}\Bigr)^{b} w^{\tau}_{\s}(x, r)+ C_{1}s^{1+b} \leq C_{1} r (w^{\tau}_{\s}(x, r))^{\frac{1}{2}} + C_{1} r^{1+\frac{b}{2}}, \numberthis \label{9.41}
	\end{align*}
	where we have used that $w^{\tau}_{\s}(x,r)<1$, $0<s<r<1$ and that $b>0$.
	By (\ref{9.39})-(\ref{9.41}),
	\[
	\beta_{\s}(x,ar) \leq C^{\prime}(w^{\tau}_{\s}(x, r))^{\frac{1}{2}} + C^{\prime}r^{\frac{b}{2}},
	\]
	with $C^{\prime}=\frac{24  C_{1}}{a}$, that shows $(i)$. Furthermore, using (\ref{9.32}) and (\ref{9.35}), we get 
	\[
	\beta_{\s}(x, ar) \leq C^{\prime}(\varepsilon_{1})^{\frac{1}{2}}+ C^{\prime} r_{0}^{\frac{b}{2}} < \varepsilon_{2}.
	\]  Then, applying Proposition~\ref{prop: 9.6} with $r_{1}=r$, $\varepsilon=\varepsilon_{2}$ and also noting that $\frac{2\varepsilon}{\tau} < a$, namely $ar \in (\frac{2\varepsilon r}{\tau}, \frac{r}{10})$, we deduce that 
	\begin{equation*}
	w^{\tau}_{\s}(x, ar) \leq C a^{b} w^{\tau}_{\s}(x, r) + C (ar)^{b} \leq \frac{1}{2} w^{\tau}_{\s}(x, r) + C^{\prime} (ar)^{b} \leq \frac{\varepsilon_{1}}{2} + \frac{\varepsilon_{1}}{2} = \varepsilon_{1},
	\end{equation*}
	where we have used that $a\leq \Bigl(\frac{1}{2C}\Bigr)^{\frac{1}{b}}$, the fact that $C<C^{\prime}$ and (\ref{9.36}).
	Notice that we have proved $(ii)$ and the fact that $\beta_{\Sigma}(x, ar)\leq \varepsilon_{2}$ and $w^{\tau}_{\Sigma}(x, ar)\leq \varepsilon_{1}$. Next, using (\ref{9.32}) with $ar$ instead of $r$, we get: $\beta_{\s}(x, a^{2} r) \leq \varepsilon_{2}$ and $w^{\tau}_{\Sigma}(x,a^{2}r)\leq \varepsilon_{1}$. Thus, iterating, we observe that for all $n \in \mathbb{N}$ the following holds
	\begin{equation*}
	w^{\tau}_{\s} (x, a^{n} r) \leq \varepsilon_{1} \,\ \text{and}\,\ \beta_{\s}(x, a^{n} r) \leq \varepsilon_{2},
	\end{equation*}
	that shows $(iii)$. This concludes the proof.
\end{proof}
Now we are ready to prove that if $\beta_{\s}(x, r) + w^{\tau}_{\s}(x,r)$ falls below a critical threshold $\delta_{0}>0$ for   $x \in \s \cap \Omega$ and sufficiently small $r>0$, then $\beta_{\s}(x,r) \leq Cr^{\alpha}$ for some $\alpha \in (0,1)$, that leads to a $C^{1,\alpha}$ regularity.


\begin{prop}\label{prop: 9.14} Let $p \in (1,+\infty)$ and $f \in L^{q}(\Om)$  with $q> q_{1}$, where $q_{1}$ is defined in~(\ref{qrestrict}), and let $a\in (0,1/10)$ be the constant of Proposition~\ref{prop: 9.13}. Let $\Sigma$ be a solution of Problem~\ref{problemMain}. Then there exists $\alpha>0$ and $\overline{r}_{0},\, \delta_{0}>0$ such that if $x \in \s$ and $0<r_{0}<\overline{r}_{0}$ satisfy $B_{r_{0}}(x) \subset \Omega$,
	\begin{equation}
	\beta_{\s}(x,r_{0}) + w^{\tau}_{\s}(x,r_{0}) \leq \delta_{0},
	\end{equation}
	then
	\begin{equation} \label{Control beta_sigma}
	\beta_{\s}(x, r) \leq C r^{\alpha} \,\ \text{for all}\,\ r \in (0, ar_{0}),
	\end{equation}
	where $C$ is a positive constant, possibly depending only on $p, q_{0}, q, \|f\|_{q},|\Omega|$ and $r_{0}$.
\end{prop}
\begin{proof} Let $\varepsilon_{1},\, b,\, r_{0}, \, C>0$ be as in Proposition~\ref{prop: 9.13}. Next, we define 
	\[
	\delta_{0}:=\frac{a\varepsilon_{1}}{4},\,\ \gamma:=\min \biggl\{\frac{b}{2}, \frac{\ln(3/4)}{\ln(a)}\biggr\}, \,\ \overline{r}_{0}:=\Biggl\{r_{0}, \biggl(\frac{1}{4}\biggr)^{\frac{1}{\gamma}}\Biggr\}.
	\]
	It is easy to check that for all $t \in (0, \overline{r}_{0}]$,
			\begin{equation} \label{Good inequality}
			\frac{1}{2}t^{\gamma}+t^{b}\leq (at)^{\gamma},
			\end{equation} 
			because since $0<2\gamma \leq b$ and $\overline{r}_{0}<1$,
			\[
			\frac{1}{2}t^{\gamma}+t^{b}\leq \frac{1}{2}t^{\gamma}+\overline{r}_{0}^{\gamma}t^{\gamma} \leq \frac{3}{4}t^{\gamma} \leq (at)^{\gamma}.
			\]
			Now let $r_{0}$ be a radius given in the statement, $r_{0}<\overline{r}_{0}$. Let us show by induction that for all $n \in \mathbb{N}$,
	\begin{equation}
	w^{\tau}_{\s}(x,a^{n}r_{0}) \leq \frac{1}{2^{n}}w^{\tau}_{\s}(x,r_{0}) + C(a^{n+1}r_{0})^{\gamma}.\label{9.46}
	\end{equation}
	Obviously, (\ref{9.46}) holds for $n=0$. Suppose (\ref{9.46}) holds for some $n \in \mathbb{N}$. Notice that by (\ref{9.34}), $\beta_{\Sigma}(x, a^{n}r_{0}) + w^{\tau}_{\Sigma}(x, a^{n}r_{0})\leq \delta_{0}$. Then, applying (\ref{Sigma estimate}) with $r=a^{n}r_{0}$, we get
	\begin{align*}
	w^{\tau}_{\s}(x,a^{n+1}r_{0}) & \leq \frac{1}{2}w^{\tau}_{\s}(x, a^{n}r_{0}) + C(a^{n+1}r_{0})^{b}\\
	& \leq \frac{1}{2^{n+1}}w^{\tau}_{\s}(x, r_{0}) + \frac{C}{2}(a^{n+1}r_{0})^{\gamma} + C(a^{n+1}r_{0})^{b}\,\ \text{(by induction)} \\
	& \leq \frac{1}{2^{n+1}}w^{\tau}_{\s}(x,r_{0}) + C(a^{n+2}r_{0})^{\gamma}\,\ (\text{by (\ref{Good inequality})}),
	\end{align*}
	that shows (\ref{9.46}). Now let $r \in (0, r_{0})$ and let $l \in \mathbb{N}$ be such that $a^{l+1}r_{0}<r \leq a^{l}r_{0}$. Then we deduce that
	\begin{align*}
	w^{\tau}_{\s}(x, r) & \leq \frac{1}{a} w^{\tau}_{\s}(x, a^{l}r_{0}) \leq \frac{1}{a}\frac{1}{2^{l}} w^{\tau}_{\s}(x, r_{0}) + \frac{C}{a}(a^{l+1}r_{0})^{\gamma}\\
	&\leq \frac{2}{a}\frac{1}{2^{l+1}}w^{\tau}_{\Sigma}(x, r_{0}) + \frac{C}{a}r^{\gamma} \\
	& \leq \frac{1}{2^{l+1}}\frac{\varepsilon_{1}}{2} + \frac{C}{a}r^{\gamma} \,\ \text{(since $w^{\tau}_{\s}(x,r_{0}) \leq \frac{a\varepsilon_{1}}{4}$)}\\
	& \leq a^{\gamma(l+1)}\frac{\varepsilon_{1}}{2} + \frac{C}{a}r^{\gamma} \leq \frac{\varepsilon_{1}}{2} \bigl(\frac{r}{r_{0}}\bigr)^{\gamma} +\frac{C}{a}r^{\gamma}.
	\end{align*}
	Thus, for all $r\in (0, r_{0})$,
	\begin{equation}
	w^{\tau}_{\s}(x, r) \leq \frac{\varepsilon_{1}}{2}\bigl(\frac{r}{r_{0}}\bigr)^{\gamma} + \frac{C}{a}r^{\gamma}. \label{9.47}
	\end{equation}
	By (\ref{9.33}) and (\ref{9.47}), for all $r \in (0, r_{0})$, 
	\begin{align*}
	\beta_{\s}(x, ar) \leq C(w^{\tau}_{\s}(x, r))^{\frac{1}{2}} + C r^{\frac{b}{2}}\leq C \bigl(\frac{\varepsilon_{1}}{2}\bigl(\frac{r}{ r_{0}}\bigr)^{\gamma} + \frac{C}{a}r^{\gamma}\bigr)^{\frac{1}{2}} + C r^{\frac{b}{2}} \leq \widetilde{C} r^{\alpha},
	\end{align*}
	where $\alpha=\frac{\gamma}{2}$ and $\widetilde{C}$ is a positive constant, possibly depending only on $p$, $q_{0}$, $q$, $\|f\|_{q}$, $a$, $|\Omega|$ and $r_{0}$. Therefore, $\beta_{\s}(x, r) \leq C^{\prime}r^{\alpha}$ for all $r \in (0, a r_{0})$ with $C^{\prime}=\widetilde{C}/a^{\alpha}$. Notice that although $C^{\prime}$ depends on $a$, however, for any given $p \in (1,+\infty)$ we can fix $a$, and thus, we can assume that $C^{\prime}$ depends only on $p, q_{0}, q, \|f\|_{q}, |\Omega|$ and $r_{0}$.  This concludes the proof.
\end{proof}

\begin{cor} \label{cor: 9.15} \textit{Let $\s$ be a solution of Problem~\ref{problemMain} and $a,\,\alpha,\, \overline{r}_{0},\, \delta_{0}>0$ be as in the statement of Proposition~\ref{prop: 9.14}. Let $x_{0} \in \s$ and $0<r_{0}<\overline{r}_{0}$ be such that $B_{r_{0}}(x_{0}) \subset \Omega$ and 
		\begin{equation*}
		\beta_{\s}(x_{0}, r_{0}) + w^{\tau}_{\s}(x_{0}, r_{0}) \leq \varepsilon_{0}
		\end{equation*}
		with $\varepsilon_{0} := \frac{a\delta_{0}}{120}$. Then for any $y \in \s \cap B_{\frac{a r_{0}}{10}}(x_{0})$ and for any $r \in (0, \frac{ar_{0}}{20})$ we have that $\beta_{\s}(y,r) \leq C r^{\alpha}$, where $C=C(p,q_{0},q,\|f\|_{q},|\Omega|,r_{0})>0$. In particular, there exists $t_{0}\in (0,1)$, only depending on $C,\,a,\, r_{0}$ and $\alpha$, such that $\s \cap \overline{B}_{t_{0}}(x_{0})$ is a $C^{1,\alpha}$ regular curve.}
\end{cor}
\begin{proof}[Proof of Corollary~\ref{cor: 9.15}] Recall that $a \in (0,1/10)$. Let $y \in \s \cap B_{\frac{a r_{0}}{10}}(x_{0})$ and let $P_{0}$ realize the infimum in the definition of $\beta_{\s}(x_{0}, r_{0})$. Since $d_{H}(\s \cap \overline{B}_{\frac{ r_{0}}{20}}(y), P_{0} \cap \overline{B}_{\frac{ r_{0}}{20}}(y))~\leq~3\varepsilon_{0}r_{0}$, $\beta_{\s}(y, \frac{r_{0}}{20}) \leq 200 \varepsilon_{0} <\frac{\delta_{0}}{2}.$ Moreover, if $\s^{\prime}$ realizes the supremum in the definition of $w^{\tau}_{\s}(y, \frac{r_{0}}{20})$, then 
	\begin{align*}
	w^{\tau}_{\s}(y, \tfrac{ r_{0}}{20})  = \frac{20}{ r_{0}} \int_{B_{\frac{r_{0}}{20}}(y)} |\nabla u_{\s^{\prime}}|^{p} \diff x \leq \frac{20}{ r_{0}} \int_{B_{r_{0}}(x_{0})} |\nabla u_{\s^{\prime}}|^{p} \diff x \leq 20 w^{\tau}_{\s}(x_{0}, r_{0}) < \frac{\delta_{0}}{6}
	\end{align*}
	and hence
	\begin{equation*}
	\beta_{\s}(y, \tfrac{r_{0}}{20}) + w^{\tau}_{\s}(y, \tfrac{r_{0}}{20}) < \delta_{0}.
	\end{equation*}
	Then by Proposition~\ref{prop: 9.14}, there exists a constant $C=C(p,q_{0},q,\|f\|_{q},|\Omega|, r_{0})>0$ such that $\beta_{\s}(y, r) \leq Cr^{\alpha}$ for all $r \in (0, \frac{a r_{0}}{20})$. Since $y \in \s \cap B_{\frac{ar_{0}}{10}}(x_{0})$ was arbitrarily chosen in $\s \cap B_{\frac{a r_{0}}{10}}(x_{0})$, there exists $t_{0} \in (0, \frac{a r_{0}}{10})$ such that $\s \cap \overline{B}_{t_{0}}(x_{0})$ is a $C^{1,\alpha}$ regular curve (see e.g. \cite[Lemma 6.4]{Babadjian-Iurlano-Lemenant}).
\end{proof}

Now we prove that locally $\s\cap \Omega$ is a $C^{1,\alpha}$ regular curve outside a set with zero $\mathcal{H}^{1}$-measure.

\begin{proof}[Proof of Theorem~\ref{thm: 9.16}] Let $\varepsilon_{0} \in (0,1/2),\, b,\,  \overline{r},\, C>0$ be the constants of Lemma~\ref{Good_decay} and let $\tau \in (0, \frac{\varepsilon_{0}}{6})$. Since closed connected sets with finite length are rectifiable, then (see e.g. \cite[Proposition 2.2]{Approx}) for $\mathcal{H}^{1}$-a.e. point $x$ in $\s \cap \Om$ there is the affine line $T_{x}$, passing through $x$, such that
	\begin{equation}
	\frac{d_{H}(\s \cap \overline{B}_{r}(x), T_{x} \cap \overline{B}_{r}(x))}{r} \underset{r \to 0+}{\to} 0. \label{9.48}
	\end{equation}
	Let $x \in \s \cap \Om$ be such a point. Then by (\ref{9.48}),
	\begin{equation}
	\beta_{\s}(x, r) \underset{r \to 0+}{\to 0}. \label{9.49}
	\end{equation}
	We claim that $w^{\tau}_{\s}(x, r)$ tends to zero, as $r\to 0+$. Indeed, by (\ref{9.49}), for any $\varepsilon \in (0, \varepsilon_{0})$ there is $t_{\varepsilon} \in (0, \overline{r})$ such that 
	\begin{equation}
	\beta_{\s}(x, r) \leq \varepsilon \,\ \text{for all} \,\ r \in (0, t_{\varepsilon}]. \label{9.50}
	\end{equation}
	We assume that $B_{t_{\varepsilon}}(x) \subset \Omega$ and $t_{\varepsilon}<\diam(\Sigma)/2$. Then by Proposition~\ref{prop: 9.6}, for all  $r\in (0, t_{\varepsilon}/10]$,
	\begin{equation}
	w^{\tau}_{\Sigma}(x, r) \leq C\Bigl(\frac{r}{t_{\varepsilon}}\Bigr)^{b}w^{\tau}_{\Sigma}(x,t_{
	\varepsilon}) + Cr^{b}. \label{9.51}
	\end{equation}
	On the other hand, by Remark~\ref{rem: 9.3} and by Proposition~\ref{prop: 2.12}, $w^{\tau}_{\Sigma}(x,t_{\varepsilon})\in [0, +\infty)$. Then, letting $r$ tend to $0+$ in (\ref{9.51}), we get
	\begin{equation}
	w^{\tau}_{\Sigma}(x,r) \underset{r \to 0+}{\to 0}. \label{9.52}
	\end{equation}
	By (\ref{9.49}) and (\ref{9.52}),
	\[
	\beta_{\s}(x,r) + w^{\tau}_{\s}(x,r) \underset{r \to 0+}{\to} 0.
	\]
	This together with Corollary~\ref{cor: 9.15} concludes the proof.
\end{proof}

\end{subsection}
\section{Remark about singular points}
We shall say that a set $K\subset \mathbb{R}^{2}$ is a cross passing through a point $x \in \mathbb{R}^{2}$ if $K$ consists of two mutually perpendicular affine lines passing through $x$. For convenience, we denote the cross $(\{0\}\times \mathbb{R})\cup (\mathbb{R}\times \{0\})$ passing through the origin by $K_{0}$. 

In this section, we prove that every solution $\Sigma$ of Problem~\ref{problemMain} cannot contain quadruple points inside $\Omega$, namely, there is no point $x \in \Sigma \cap \Omega$ such that for some fairly small radius $r>0$ the set $\Sigma \cap \overline{B}_{r}(x)$ is a union of four distinct $C^{1}$ arcs, each of which meets at point $x$ exactly one of the other three at an angle of $180^{\circ}$ degrees, and each of the other two at an angle of $90^{\circ}$ degrees.

We start by proving the following lemma.
	\begin{lemma}\label{reflection cross} Let $p \in (1,+\infty)$. Then there is a constant $C=C(p)>0$ such that for all $u \in W^{1,p}(B_{1})$, $u=0$ $p$-q.e. on $ K_{0}\cap \overline{B}_{1}$ being a weak solution of the $p$-Laplace equation in $B_{1}\backslash K_{0}$, 
		\[
		\ess_{B_{1/2}} |\nabla u|^{p} \leq C \int_{B_{1}}|\nabla u|^{p}\diff x.
		\]
	\end{lemma} 
	\begin{proof} To simplify the notation, we denote the sets $B_{1}\cap \{x_{2}\geq 0\},\,\ B_{1}\cap \{x_{2}\leq 0\}$, $B_{1}\cap \{x_{1}\geq 0\},\,\ B_{1}\cap \{x_{1}\leq 0\},\,\ B_{1}\cap \{x_{1}\leq 0\}\cap \{x_{2} \geq 0\},\,\ B_{1}\cap \{x_{1}\geq 0\}\cap \{x_{2} \geq 0\}$, $B_{1}\cap \{x_{1} \geq 0\} \cap \{x_{2}\leq 0\}, \,\ B_{1} \cap \{x_{1}\leq 0\}\cap \{x_{2}\leq 0\},$ respectively by $B_{N},$ $B_{S},$ $B_{E},$ $B_{W},$ $B_{NW},$ $B_{NE},$ $B_{SE},$ $B_{SW}$. Next, reproducing the arguments of the proof of Lemma~\ref{lem: 7.4}, we observe that the Sobolev functions $\widetilde{u}, \overline{u} \in W^{1,p}(B_{1})$ defined by
		\[
		\widetilde{u}(x_{1},x_{2})=\begin{cases}                    
		u(x_{1},x_{2}) &\text{if}\,\ (x_{1},x_{2}) \in B_{N}\\
		-u(x_{1},-x_{2}) &\text{if} \,\ (x_{1},x_{2}) \in B_{S},
		\end{cases} \,\ \,\ 
		\overline{u}(x_{1},x_{2})=\begin{cases}
		-u(x_{1}, -x_{2}) &\text{if}\,\ (x_{1},x_{2}) \in B_{N}\\
		u(x_{1},x_{2}) &\text{if}\,\ (x_{1},x_{2}) \in B_{S} 
		\end{cases}
		\] 
		are weak solutions of the $p$-Laplace equations in $B_{1}\backslash (\{0\}\times [-1,1])$ vanishing $p$-q.e. on $\{0\}\times [-1,1]$, and, in addition, the Sobolev functions $\widetilde{v}, \widetilde{w}, \overline{v}, \overline{w} \in W^{1,p}(B_{1})$ defined by
		\[
		\widetilde{v}(x_{1},x_{2})=\begin{cases}
		\widetilde{u}(x_{1},x_{2}) &\text{if}\,\ (x_{1},x_{2}) \in B_{W}\\
		-\widetilde{u}(-x_{1},x_{2}) &\text{if}\,\ (x_{1},x_{2}) \in B_{E},
		\end{cases} \,\ \,\
		\widetilde{w}(x_{1},x_{2})=\begin{cases}
		-\widetilde{u}(-x_{1},x_{2}) &\text{if}\,\ (x_{1},x_{2})\in B_{W}\\
		\widetilde{u}(x_{1},x_{2}) &\text{if}\,\ (x_{1},x_{2})\in B_{E},
		\end{cases}
		\]
		\[
		\overline{v}(x_{1},x_{2})=\begin{cases}
		\overline{u}(x_{1},x_{2}) &\text{if}\,\ (x_{1},x_{2}) \in B_{W}\\
		-\overline{u}(-x_{1},x_{2}) &\text{if}\,\ (x_{1},x_{2}) \in B_{E},
		\end{cases} \,\ \,\
		\overline{w}(x_{1},x_{2})=\begin{cases}
		-\overline{u}(-x_{1},x_{2}) &\text{if}\,\ (x_{1},x_{2})\in B_{W}\\
		\overline{u}(x_{1},x_{2}) &\text{if}\,\ (x_{1},x_{2}) \in B_{E}
		\end{cases}
		\]
		are weak solutions of the $p$-Laplace equations in $B_{1}$. Thus, by \cite[Proposition 3.3]{DIBENEDETTO1983827}, there is $C=C(p)>0$ such that for each $\zeta \in \{\widetilde{v}, \widetilde{w}, \overline{v}, \overline{w}\},$
		\[
		\ess_{B_{1/2}} |\nabla \zeta|^{p}\leq C \int_{B_{1}} |\nabla \zeta|^{p}\diff x,
		\]
		which implies that for each $U \in \{int(B_{NW}), int(B_{NE}), int(B_{SE}), int(B_{SW})\},$
		\[
		\ess_{U\cap B_{1/2}} |\nabla u|^{p} \leq 4C\int_{U} |\nabla u|^{p}\diff x. 
		\]
		Thus, we can conclude that 
		\[
		\ess_{B_{1/2}} |\nabla u|^{p} \leq 4C \int_{B_{1}} |\nabla u|^{p}\diff x,
		\]
		which completes the proof.
	\end{proof}
	The following lemma says that if $\Sigma \subset \overline{\Omega}$ is a closed arcwise connected set, $0<2r_{0}\leq r_{1}$, $r_{1}$ is sufficiently small, $B_{r_{1}}(x_{0})\subset \Omega$ and for each $r \in [r_{0},r_{1}]$ there exists a cross $K$ passing through $x_{0}$ such that $\Sigma$ is close enough, in $\overline{B}_{r}(x_{0})$ and in the Hausdorff distance, to $K\cap \overline{B}_{r}(x_{0})$, then the energy $[r_{0},r_{1}]\ni r \mapsto \frac{1}{r}\int_{B_{r}(x_{0})} |\nabla u_{\Sigma}|^{p}\diff x$ decays no slower than $Cr^{b}$ for some $b \in (0,1)$ and $C>0$.
	\begin{lemma}\label{decay cross}
		Let $p\in (1, +\infty)$, $f \in L^{q}(\Omega)$ with  $q> q_1$, where $q_1$ is defined in (\ref{qrestrict}). Then there exists   $\varepsilon_0 \in (0,1/100)$,  $C=C(p,q_{0},q,\|f\|_q, |\Omega|),\, \overline{r}, \, b>0$ such that the following holds.    Let $\s\subset \overline{\Omega}$ be a closed arcwise connected set. Assume that $0<2r_{0}\leq r_{1}\leq \overline{r}$,  $B_{r_{1}}(x_{0}) \subset \Omega$ and that for all  $r \in [r_{0}, r_{1}]$ there is a cross $K=K(r)$,  passing through $x_{0}$, such that $d_{H}(\s\cap \overline{B}_{r}(x_{0}), K\cap \overline{B}_{r}(x_{0})) \leq \varepsilon_0 r$. Assume also that $\Sigma \setminus B_{r_1}(x_0)\not = \emptyset$. Then for every $r\in[r_{0}, r_{1}]$,
		\begin{equation*}
		\int_{B_{r}(x_{0})}|\nabla u_\Sigma|^{p}\diff x \leq C \bigl(\frac{r}{r_{1}}\big)^{1+b}  \int_{B_{r_{1}}(x_{0})}|\nabla u_\Sigma|^{p}\diff x + Cr^{1+b}. 
		\end{equation*}
	\end{lemma}
	\begin{proof}
		The proof follows by reproducing the proofs of Lemma~\ref{lem: 7.7}, Lemma~\ref{lem: 7.11P} and Lemma~\ref{Good_decay} with a minor modification, namely, replacing the affine line by a cross in the proofs of these lemmas, such a reproduction is possible thanks to Lemma~\ref{reflection cross}.
	\end{proof} 
	\begin{prop}\label{remark about quadruple points} Let $\Omega \subset \mathbb{R}^{2}$ be a bounded open set, $p \in (1,+\infty)$ and $f \in L^{q_{1}}(\Omega)$ with $q>q_{1}$ defined in (\ref{qrestrict}). Then every solution $\Sigma$ of Problem~\ref{problemMain} cannot contain quadruple points in $\Omega$. 
	\end{prop}
	\begin{proof}
		Assume by contradiction that for some $\lambda>0$ a minimizer $\Sigma$ of Problem~\ref{problemMain} contains a quadruple point $x_{0}\in \Sigma \cap \Omega$. Let $\varepsilon_{0}, b ,\overline{r}, C$ be the constants of Lemma~\ref{decay cross} and let $B_{t_{0}}(x_{0})\subset \Omega $ with $t_{0}<\min\{\overline{r}, \diam(\Sigma)/2\}$. Without loss of generality, we can assume that the set $\Sigma \cap \overline{B}_{t_{0}}(x_{0})$ consists of exactly four distinct $C^{1}$ arcs, each of which meets at point $x_{0}$ exactly one of the other three at an angle of $180^{\circ}$ degrees, and each of the other two at an angle of $90^{\circ}$ degrees. Then there exists a cross $K$ passing through $x_{0}$ such that for each $\varepsilon>0$ there exists $\delta=\delta(\varepsilon) \in (0, t_{0}]$ such that for all $r \in (0,\delta]$,
		\begin{equation}\label{distance to cross}
		d_{H}(\Sigma \cap \overline{B}_{r}(x_{0}), K\cap \overline{B}_{r}(x_{0}))\leq \varepsilon r.
		\end{equation}
		We fix $\varepsilon=\frac{\varepsilon_{0}}{2}$ and a sequence of decreasing radii $(r_{n})_{n\in \mathbb{N}}$ with $r_{0}<\delta=\delta\bigl(\frac{\varepsilon_{0}}{2}\bigr)$. Following \cite{BOS}, for each $n\in \mathbb{N}$, we define the set $D_{n}=K\cap \partial B_{r_{n}}(x_{0})$ which consists of exactly four points. Denote by $S_{4}(D_{n})\subset \overline{B}_{r_{n}}(x_{0})$ a closed set of minimum $\mathcal{H}^{1}$-measure in the ball $\overline{B}_{r_{n}}(x_{0})$ which connects the all four points of $D_{n}$ (as in \cite{BOS}, we shall call it \textit{Steiner connection} of these points); see Figure~\ref{steiner connection}.
		\begin{figure}[H]
			\centering
			\includegraphics[width=0.3\textwidth]{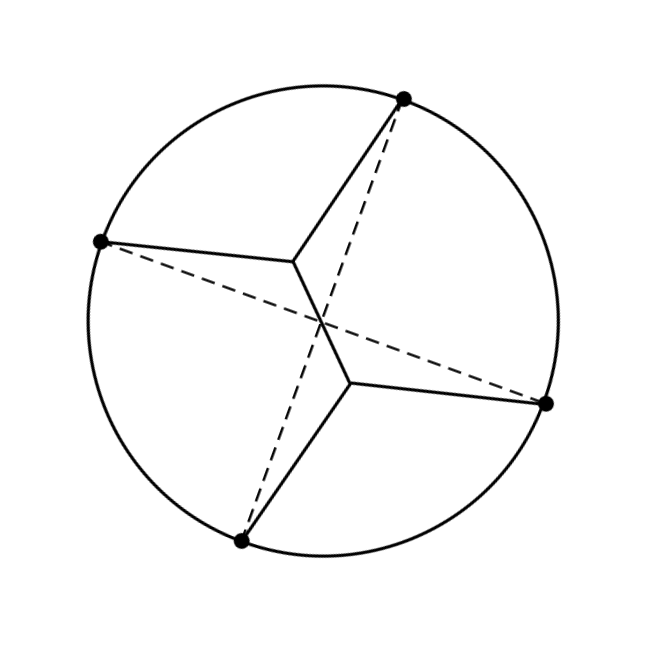}
			\caption{Steiner connection of the vertices of a square.}
			\label{steiner connection}
		\end{figure}
		Next, for each $n\in \mathbb{N}$, we can define a competitor $\Sigma_{n}$ by
		\[
		\Sigma_{n}=(\Sigma \backslash B_{r_{n}}(x_{0})) \cup (\partial B_{r_{n}}(x_{0})\cap \{y: \dist(y, K)\leq d_{H}(\Sigma\cap \overline{B}_{r_{n}}(x_{0}), K\cap \overline{B}_{r_{n}}(x_{0}))\}) \cup S_{4}(D_{n});
		\]
		see Figure~\ref{figure 6.2}.
		\begin{figure}[H]
			\centering
			\includegraphics[width=0.5\textwidth]{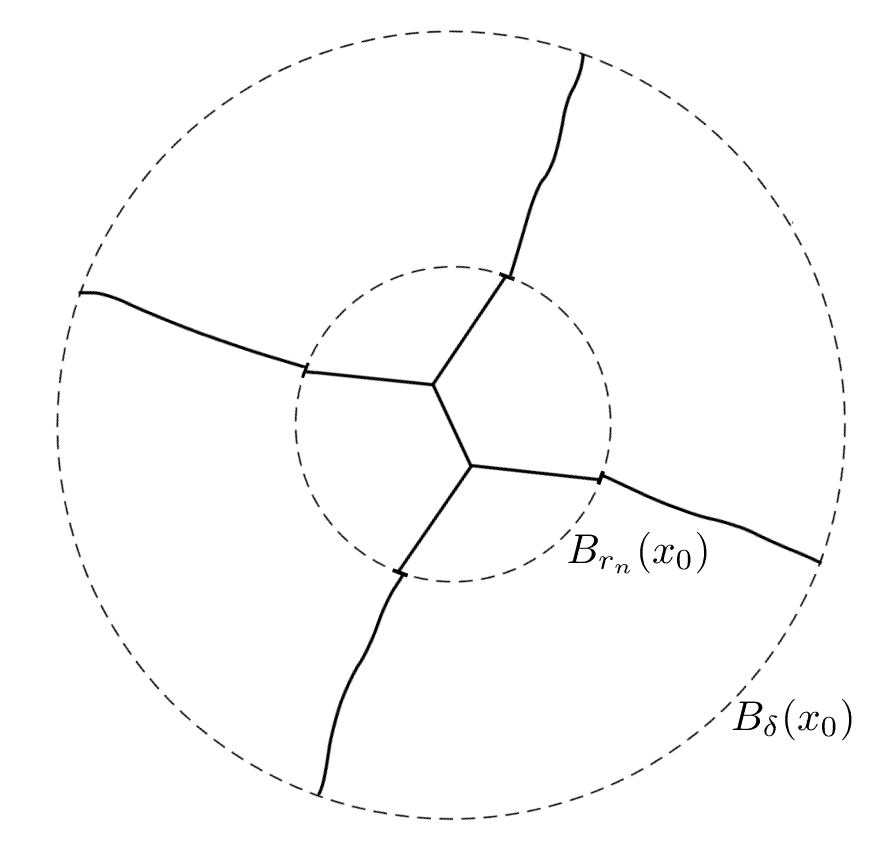}
			\caption{The set $\Sigma_{n}\cap B_{\delta}(x_{0})$.}
			\label{figure 6.2}
		\end{figure}
		Due to the condition (\ref{distance to cross}), each of the four arcs in 
		\[
		\partial B_{r_{n}}(x_{0})\cap \{y: \dist(y, K)\leq d_{H}(\Sigma\cap \overline{B}_{r_{n}}(x_{0}), K\cap \overline{B}_{r_{n}}(x_{0}))\} 
		\]
		has $\mathcal{H}^{1}$-measure less than or equal to $2\arcsin\bigl(\frac{\varepsilon_{0}}{2}\bigr)r_{n}.$ On the other hand,
		\[
		\mathcal{H}^{1}(\Sigma \cap B_{r_{n}}(x_{0}))\geq 4r_{n} \,\ \text{and} \,\ \mathcal{H}^{1}(S_{4}(D_{n}))=\sqrt{2}(\sqrt{3}+1)r_{n},
		\]
		where we have used that $\mathcal{H}^{1}(S_{4}(D_{n}))=\mathcal{H}^{1}(S_{4}(K_{0}\cap \partial B_{1}))r_{n}=\sqrt{2}(\sqrt{3}+1)r_{n}$. Observing that $\varepsilon_{0}\in (0, 1/100)$, $2\arcsin\bigl(\frac{\varepsilon_{0}}{2}\bigr)\leq 2\varepsilon_{0}$ (since $\arcsin(t)\leq 2t$ for all $t\in [0,1/10]$) and $\sqrt{2}(\sqrt{3}+1)\approx 3.86$, we can conclude that there is a constant $\widetilde{C}>0$ independent of $n$ such that for each $n\in \mathbb{N}$,
		\begin{equation}\label{condition}
		\mathcal{H}^{1}(\Sigma \cap \overline{B}_{r_{n}}(x_{0})) -\mathcal{H}^{1}(\Sigma_{n}\cap \overline{B}_{r_{n}}(x_{0})) \geq \widetilde{C} r_{n}.
		\end{equation}
		Now we want to apply Lemma~\ref{decay cross} to $\Sigma_{n}$. If $r_{n}\leq \frac{\varepsilon_{0}r}{2}$ and $r\in (0, \delta]$, then
		\begin{align*}
		& \, d_{H}(\Sigma_{n}\cap \overline{B}_{r}(x_{0}), K\cap \overline{B}_{r}(x_{0})) \\
		\;\, \leq &\, d_{H}(\Sigma_{n}\cap \overline{B}_{r}(x_{0}), \s\cap \overline{B}_{r}(x_{0}))+ d_{H}(\Sigma\cap \overline{B}_{r}(x_{0}), K\cap \overline{B}_{r}(x_{0})) \\
		\;\, \leq &\, r_{n} +  \frac{\varepsilon_{0}r}{2} \leq \frac{\varepsilon_{0}r}{2} +\frac{\varepsilon_{0}r}{2} =  \varepsilon_{0} r,
		\end{align*}
		where we have used (\ref{distance to cross}). So we can apply Lemma~\ref{decay cross} to $\Sigma_{n}$, for the interval $[\frac{2r_{n}}{\varepsilon_{0}}, \delta]$, provided that $\frac{2r_{n}}{\varepsilon_{0}} \leq \frac{\delta}{2}$,  and we obtain that 
		\begin{equation*} 
		\int_{B_{r}(x_{0})}|\nabla u_{\Sigma_n}|^{p}\diff x \leq C \bigl(\frac{r}{\delta}\big)^{1+b}  \int_{B_{\delta}(x_{0})}|\nabla u_{\Sigma_n}|^{p}\diff x + Cr^{1+b}\,\ \text{for every}\,\  r\in\biggl[\frac{2 r_{n}}{\varepsilon_0}, \delta\biggr]. 
		\end{equation*}
		Hereinafter in this proof, $C$ denotes a positive constant that does not depend on $r_{n}$ and can be different from line to line. Thus, applying the above estimate for $r=\frac{2r_{n}}{\varepsilon_{0}}$ and using (\ref{2.4}), we have
		\begin{equation}\label{estimate of the integral}
		\int_{B_{\frac{2r_{n}}{\varepsilon_{0}}}(x_{0})}|\nabla u_{\Sigma_{n}}|^{p}\diff x \leq C r^{1+b}_{n}
		\end{equation}
		for all $n\in \mathbb{N}$ such that $\frac{2r_{n}}{\varepsilon_{0}} \leq \frac{\delta}{2}$. Recall that the exponent $b$ given by Lemma~\ref{decay cross} is positive provided $q>q_{1}$. Now, using the fact that $\Sigma$ is a minimizer and $\Sigma_{n}$ is a competitor for $\Sigma$, the estimate (\ref{condition}), Corollary~\ref{cor: 7.3} and the estimate (\ref{estimate of the integral}),  we deduce the following
		\begin{align*} 
		0&\leq \mathcal{F}_{\lambda,p}({\s}_{n})- \mathcal{F}_{\lambda,p}(\s)  \leq E_{p}(u_{\s})-E_{p}(u_{\s_{n}}) - \lambda \widetilde{C}r_{n}\\
		& \leq C \int_{B_{2r_{n}}(x_{0})}|\nabla u_{\s_{n}}|^{p}\diff x + Cr^{2+p^{\prime}-\frac{2p'}{q}}_{n} -\lambda \widetilde{C} r_{n}  \\
		&\leq C \int_{B_{\frac{2r_{n}}{\varepsilon_{0}}}(x_{0})}|\nabla u_{\Sigma_{n}}|^{p}\diff x  + Cr^{2+p^{\prime}-\frac{2p'}{q}}_{n} -\lambda\widetilde{C} r_{n}  \\
		&\leq  Cr_n^{1+b} +Cr^{2+p^{\prime}-\frac{2p'}{q}}_{n} -\lambda \widetilde{C} r_{n} 
		\end{align*}
		for all $n\in\mathbb{N}$ such that $\frac{2r_{n}}{\varepsilon_{0}} \leq \frac{\delta}{2}$. Notice that $2+p^{\prime}-\frac{2p^{\prime}}{q}>1$ if and only if $q>\frac{2p}{2p-1}$, which is fulfilled under the assumption $q>q_{1}$. Finally, letting $n$ tend to $+\infty$, we arrive to a contradiction. This completes the proof of Proposition~\ref{remark about quadruple points}.
	\end{proof}


\section{Acknowledgments}
We would like to warmly thank the anonymous referees for carefully reading, checking and commenting on our paper. This work was partially supported by the project ANR-18-CE40-0013 SHAPO financed by the French Agence Nationale de la Recherche (ANR).


\appendix
\section{Auxiliary results}
In the next lemma we prove the integration by parts formula for a weak solution of the $p$-Poisson equation.
\begin{lemma} Let $U$ be a bounded open set in ${\color{black}\mathbb{R}^N}, N\geq 2$ and $p \in (1, +\infty)$, and let $f\in L^{q}(U)$ with $q=\frac{Np}{Np-N+p}$ if $1<p<N$,\,\ $q>1$ if $p=N$ and $q=1$ if $p>N$. Let $u$ be the solution of the Dirichlet problem
\begin{equation}
\begin{cases}
-\Delta_{p}u & =\,\ f \,\ \text{in}\,\ U \\  \label{A1}
\qquad u & =\,\ 0 \,\ \text{on}\,\ \partial U,
\end{cases}
\end{equation}
which means $u \in W^{1,p}_{0}(U)$ and 
\begin{equation}
\int_{U} |\nabla u|^{p-2}\nabla u \nabla \varphi \diff x= \int_{U} f \varphi \diff x \,\ \text{for all}\,\ \varphi \in W^{1,p}_{0}(U).  \label{A2}
\end{equation}
Then for every $x_{0} \in {\color{black}\mathbb{R}^2}$ and a.e. $r>0$ we have
\begin{equation*}
\int_{B_{r}(x_{0})} |\nabla u|^{p}\diff x = \int_{\partial B_{r}(x_{0})} u |\nabla u|^{p-2}\nabla u \cdot {\color{black}\nu} \diff \mathcal{H}^{N-1} + \int_{B_{r}(x_{0})} fu \diff x,
\end{equation*}
where $\nu$ stands for the outward pointing unit normal vector to $\partial B_{r}(x_{0})$.
\end{lemma}
\begin{proof} Every ball in this proof is centered at $x_{0}$. We extend $u$ be zero outside $U$ to an element that belongs to $W^{1,p}({\color{black}\mathbb{R}^N})$.  Let us fix an arbitrary $\varepsilon \in (0,r)$ and define
\begin{equation*}
g_{\varepsilon}(t)=
\begin{cases}
1 & \text{if}\,\ \,\  t \in [0,r-\varepsilon]\\
-\frac{1}{\varepsilon}(t-r) & \text{if}\,\ \,\ t \in [r-\varepsilon, r]\\
 0 & \text{if} \,\ \,\   t \in [r, +\infty).
\end{cases}
\end{equation*}
Since $g_{\varepsilon} \in \lip(\mathbb{R}^{+})$, it is clear that the function $\varphi(x):=g_{\varepsilon}(|x-x_{0}|)u$ is an element of $W^{1,p}_{0}(U)$. Thus using the function $\varphi$ as a test function in the weak version of the $p$-Poisson equation which defines $u$, we get
\begin{align*}
\int_{U} |\nabla u|^{p} g_{\varepsilon}(|x-x_{0}|) \diff x + & \int_{U} u |\nabla u|^{p-2}\nabla u \cdot g^{\prime}_{\varepsilon}(|x-x_{0}|) \frac{x-x_{0}}{|x-x_{0}|}\diff x \\ & \qquad= \int_{U} f u g_{\varepsilon}(|x-x_{0}|) \diff x.
\end{align*}
Letting $\varepsilon$ tend  $0+$, we have
\begin{equation} \label{A.3}
\begin{split}
\int_{U} |\nabla u|^{p} g_{\varepsilon}(|x-x_{0}|) \diff x & \to \int_{B_{r}} |\nabla u|^{p}\diff x \\
\int_{U} f u g_{\varepsilon}(|x-x_{0}|) \diff x & \to  \int_{B_{r}} f u \diff x.
\end{split}
\end{equation}
On the other hand, using the integration in the polar coordinates system (see \cite[3.4.4]{Evans}), which is the special case of the coarea formula, we get
\begin{align*}
\int_{U} u |\nabla u|^{p-2}\nabla u \cdot g^{\prime}_{\varepsilon}(|x-x_{0}|)& \frac{x-x_{0}}{|x-x_{0}|}\diff x = -\frac{1}{\varepsilon}\int_{B_{r}\backslash \overline{B}_{r-\varepsilon}} u |\nabla u|^{p-2}\nabla u \cdot \frac{x-x_{0}}{|x-x_{0}|} \diff x \\
& = -\frac{1}{\varepsilon} \int^{r}_{r-\varepsilon} \diff \rho \int_{\partial B_{\rho}} u |\nabla u|^{p-2} \nabla u \cdot \frac{x-x_{0}}{\rho}\diff \mathcal{H}^{N-1}(x) \\
& \to -\int_{\partial B_{r}} u|\nabla u|^{p-2} \nabla u \cdot {\color{black}\nu} \diff \mathcal{H}^{N-1}, \numberthis \label{A.4}
\end{align*}
as $\varepsilon \to 0+$, for a.e. $r>0$, because since $u \in W^{1,p}({\color{black}\mathbb{R}^N})$, the function
\[
r\in (0,+\infty) \mapsto \Psi(r):= \int^{r}_{0} \diff \rho \int_{\partial B_{\rho}} u |\nabla u|^{p-2} \nabla u \cdot \nu \diff \mathcal{H}^{N-1}
\]
is absolutely continuous on every compact subinterval of $(0,+\infty)$ and hence for a.e. $r>0$ there is $\Psi^{\prime}(r)=\int_{\partial B_{r}} u |\nabla u|^{p-2} \nabla u \cdot \nu \diff \mathcal{H}^{N-1}$ and $\Psi^{\prime}\in L^{1}(0, r)$. By (\ref{A.3}) and (\ref{A.4}) we deduce the desired formula.
\end{proof}

The following lemma on the global boundedness of weak solutions of the $p$-Poisson equation, that we prove here for the reader's convenience, is the refined version of the classical result \cite[Theorem 8.15]{PDE}. 
\begin{lemma} \label{lem: A.2} Let $U$ be a bounded open set in $\mathbb{R}^{N},\, N\geq 2$ and $p\in (1,+\infty)$, and let $f\in L^{q}(U)$ with $q>\frac{N}{p}$ if $p \in (1,N]$ and $q=1$ if $p>N$. Let $u$ be the weak solution of the equation (\ref{A1}). Then there exists a constant $C=C(N,p,q,\|f\|_{q},|U|)>0$ such that
\begin{equation}
\|u\|_{L^{\infty}(\mathbb{R}^{N})} \leq C. \label{A.5}
\end{equation}
\end{lemma}

\begin{proof} We assume that $\|f\|_{q}\neq 0$, because otherwise $u=0$ and (\ref{A.5}) holds. Recall that we can extend $u$ by zero outside $U$ to an element that belongs to $W^{1,p}(\mathbb{R}^{N})$ and we shall use the same notation for this extension as for the function $u$. If $p>N$, then by \cite[Theorem 7.10]{PDE} and since $u=0$ on $\mathbb{R}^{N}\backslash U$, there exists $C=C(N,p, |U|)>0$ such that  
\begin{equation}
\|u\|_{L^{\infty}(\mathbb{R}^{N})} \leq C \|\nabla u\|_{L^{p}(\mathbb{R}^{N})}. \label{A.6}
\end{equation}
Using $u$ as the test function in the equation (\ref{A2}), we get
\begin{align*}
\int_{\mathbb{R}^{N}}|\nabla u|^{p}\diff x & = \int_{\mathbb{R}^{N}}fu\diff x \\
& \leq C\int_{U}|f|\diff x \Bigl( \int_{\mathbb{R}^{N}} |\nabla u|^{p}\diff x \Bigr)^{1/p}\,\ (\text{by (\ref{A.6})})
\end{align*}
and then \[ \int_{\mathbb{R}^{N}} |\nabla u|^{p}\diff x \leq C^{p^{\prime}} \|f\|_{1}^{p^{\prime}} \] that together with (\ref{A.6}) implies  (\ref{A.5}).

Now let $p\in (1,N]$ and let $k=\|f\|_{q}$. For $\beta \geq 1$ and $b>k$, define the function $H \in C^{1}([k, +\infty))$ by setting $H(s)=s^{\beta}- k^{\beta}$ if $s\in [k, b]$ and for $s\geq b$ define $H$ to be linear. Next, we set $w=u^{+} + k$ and take 
$$v=G(w)=\int^{w}_{k}|H^{\prime}(s)|^{p}\diff s $$ in the equality (\ref{A2}). By the chain rule, \cite[Theorem 7.8]{PDE}, $v$ is a legitimate test function in (\ref{A2}) and on substitution we obtain
\begin{equation*}
\int_{U}|\nabla w|^{p} G^{\prime}(w)\diff x = \int_{U}f G(w)\diff x.
\end{equation*}
Observing that $|\nabla w|^{p}G^{\prime}(w)=|\nabla H(w)|^{p}$ and $G(t)\leq tG^{\prime}(t)$, and by using H\"{o}lder's inequality, we get
\begin{align*}
\int_{U}|\nabla H(w)|^{p}\diff x & \leq  \int_{U}\frac{k^{p-1}}{k^{p-1}} |f| w G^{\prime}(w) \diff x \leq  \int_{U} \frac{1}{k^{p-1}} |f| w^{p} G^{\prime}(w) \diff x \\
& = \int_{U}\frac{1}{k^{p-1}}|f||w H^{\prime}(w)|^{p} \diff x  \\ &\leq \Bigl(\int_{U}\frac{1}{k^{(p-1)q}} |f|^{q} \diff x \Bigr)^{\frac{1}{q}} \Bigl(\int_{U}|w H^{\prime}(w)|^{\frac{pq}{q-1}} \diff x \Bigr)^{\frac{q-1}{q}}
\end{align*}
and then 
\begin{equation}
\|\nabla H(w)\|_{p} \leq C_{0} \|w H^{\prime}(w) \|_{pq/(q-1)} \label{A.7}
\end{equation}
with $C_{0}=C_{0}(p,\|f\|_{q})>0$. Since $H(w) \in W^{1,p}_{0}(U)$, we may apply the Sobolev inequality \cite[Theorem 7.10]{PDE} to get
\begin{equation}
\|H(w)\|_{\hat{N}p/(\hat{N}-p)} \leq \hat{C} \|\nabla H(w)\|_{p} \label{A.8}
\end{equation}
where $\hat{N}=N$, $\hat{C}=\hat{C}(N,p)>0$ if $N>p$ and $N<\hat{N}<qp$, $\hat{C}=\hat{C}(N,p,|U|)>0$ if $N=p$. By (\ref{A.7}) and (\ref{A.8}),
\begin{equation*}
\|H(w)\|_{\hat{N}p/(\hat{N}-p)}\leq C \|w H^{\prime}(w)\|_{pq/(q-1)} 
\end{equation*}
where $C=C(N, p, |U|)>0$. Recalling the definition of $H$ and letting $b$ tend to $+\infty$ in the latter estimate, we deduce that for all $\beta \geq 1$ the inclusion $w \in L^{\frac{\beta pq}{q-1}}(U)$ implies the stronger inclusion, $w \in L^{\frac{\beta \hat{N}p}{\hat{N}-p}}(U)$ (since $\hat{N}<qp$). Thus, setting
$q^{*}=pq/(q-1)$, and $\chi = \hat{N}(q-1)/q(\hat{N}-p)>1$, we obtain 
\begin{equation}
 \|w\|_{\beta \chi q^{*}} \leq (C \beta)^{\frac{1}{\beta}} \|w\|_{\beta q^{*}}. \label{A.9}
\end{equation}
 Let us take $\beta= \chi^{m},\,\ m\in \mathbb{N},\, m\geq 1,$ so that by (\ref{A.9}),
 \begin{align*}
 \|w\|_{\chi^{m+1} q^{*}} & \leq \prod^{m}_{i=0} (C \chi^{i})^{\chi^{-i}}\|w\|_{q^{*}} \\ 
 & \leq C^{\sigma} \chi^{\tau} \|w\|_{q*}, \,\ \sigma=\chi/(\chi-1), \,\ \tau=\chi/(\chi-1)^{2}.
 \end{align*}
 Letting $m$ tend to $+\infty$, we obtain 
 \begin{equation}
 \|w\|_{\infty} \leq C^{\sigma} \chi^{\tau} \|w\|_{q*}. \label{A.10}
 \end{equation}
Hereinafter in this proof, $C$ denotes a positive constant that can depend only on $N,\, p,\, q$, $|U|$ and can change from line to line. Notice that since $q^{*}<\hat{N}p/(\hat{N}-p)$ and since $u \in W^{1,p}_{0}(U)$, using again the Sobolev inequality \cite[Theorem 7.10]{PDE}, we get 
\begin{equation}
\|u^{+}\|_{q^{*}}\leq C \|\nabla u^{+}\|_{p}. \label{A.11}
 \end{equation}
Thus, observing that $\|w\|_{q^{*}}=\|u^{+}+k\|_{q^{*}} \leq \|u^{+}\|_{q^{*}} + k|U|^{1/q^{*}}$ and using (\ref{A.10}) and (\ref{A.11}), 
\begin{equation} \label{A.12}
 \|u^{+}\|_{\infty} \leq C \|\nabla u^{+}\|_{p} + Ck.
 \end{equation}
 Now, using $u^{+}$ as the test function in equation (\ref{A2}), we get
 \begin{align*}
 \|\nabla u^{+}\|_{p}&= \Bigl(\int_{U} f u^{+}\diff x\Bigr)^{1/p} \leq Ck^{1/p} \|u^{+}\|^{1/p}_{\infty}.
 \end{align*}
 This together with (\ref{A.12}) yields
 \begin{align*}
 \|u^{+}\|_{\infty} \leq C k^{1/p} \|u^{+}\|^{1/p}_{\infty} +Ck 
 \end{align*}
 and then by Young's inequality,
 \begin{align*}
 \|u^{+}\|_{\infty} \leq \frac{1}{p}\|u^{+}\|_{\infty} + Ck^{\frac{1}{p-1}}+ Ck. \numberthis \label{A.13}
 \end{align*}
Therefore $$\|u^{+}\|_{\infty} \leq A$$ where $A=A(N,p,q,\|f\|_{q},|U|)>0$. Observing that the same estimate can be obtained by replacing $u^{+}$ with $u^{-}$, we recover (\ref{A.5}).
\end{proof} 

The next result is classical, however, we could not find a precise reference in the exact
following form and thus we provide the complete proof for the reader’s convenience.
\begin{lemma} \label{lem: A.3} Let $U$ be a bounded open set in $\mathbb{R}^{N},\, N\geq 2$  and $p\in (1,+\infty)$, and let $f\in L^{q}(U)$ with $q=\frac{Np}{Np-N+p}$ if $1<p<N$,\,\ $q>1$ if $p=N$ and $q=1$ if $p>N$. Let $u$ be the weak solution of the equation (\ref{A1}). Then $\sigma=|\nabla u|^{p-2}\nabla u$ solves the problem
\[
\min_{\sigma \in L^{p^{\prime}}(U; \mathbb{R}^{N})} \biggl\{ \frac{1}{p^{\prime}} \int_{U} |\sigma|^{p^{\prime}} \diff x : - div(\sigma)=f \,\ \text{in}\,\ \mathcal{D}^{\prime}(U) \biggr\}.
\]
Moreover, the following equality holds
\begin{equation}
\max_{w\in W^{1,p}_{0}(U)}\biggl\{\int_{U}fw\diff x - \frac{1}{p}\int_{U}|\nabla w|^{p}\diff x\biggr\}=\min_{\sigma \in L^{p^{\prime}}(U; \mathbb{R}^{N})}\biggl\{\frac{1}{p^{\prime}}\int_{U} |\sigma|^{p^{\prime}}\diff x : -div(\sigma)=f \,\ \text{in} \,\ \mathcal{D}^{\prime}(U)\biggr\}. \label{A.14}
\end{equation}
\end{lemma}
\begin{proof} Thanks to the Sobolev inequalities (see \cite[Theorem 7.10]{PDE}), the functional 
\[
W^{1,p}_{0}(U)\ni w \mapsto \int_{U}fw\diff x - \frac{1}{p}\int_{U}|\nabla w|^{p}\diff x
\]
is well defined and it is classical that it admits a unique maximizer which is the weak solution of the equation (\ref{A1}), that is $u$. For a given Sobolev function $v\in W^{1,p}(U)$ let us now show that
\begin{equation}
\frac{1}{p}\int_{U}|\nabla v|^{p}\diff x= \max_{\sigma \in L^{p^{\prime}}(U; \mathbb{R}^{N})}  \label{A.15}
\int_{U} \nabla v\cdot\sigma \diff x-  \frac{1}{p^{\prime}} \int_{U} |\sigma|^{p^{\prime}} \diff x :=\max_{\sigma \in L^{p^{\prime}}(U; \mathbb{R}^{N})}\Psi(v, \sigma)
\end{equation}
and the maximum is reached at $\widetilde{\sigma}=|\nabla v|^{p-2}\nabla v$. By the fact that $\widetilde{\sigma}$ is a competitor,
\begin{equation}
\sup_{\sigma \in L^{p^{\prime}}(U; \mathbb{R}^{N})} \Psi(v,\sigma) \geq \Psi(v,\widetilde{\sigma})= \frac{1}{p}\int_{U} |\nabla v|^{p}\diff x. \label{A.16}
\end{equation}
Since for any $\sigma \in L^{p^{\prime}}(U; \mathbb{R}^{N})$, using H\"{o}lder's inequality, one has
\begin{equation}
\Psi(v,\sigma)  \leq \biggl(\int_{U} |\nabla v|^{p} \diff x \biggr)^{\frac{1}{p}}\biggl(\int_{U}|\sigma|^{p^{\prime}} \diff x \biggr)^{\frac{1}{p^{\prime}}}- \displaystyle \frac{1}{p^{\prime}} \int_{U} |\sigma|^{p^{\prime}} \diff x \label{A.17}
\end{equation}
and since the maximum of the function $g(t)= \|\nabla v\|_{L^{p}(U)}t^{\frac{1}{p^{\prime}}}-\frac{1}{p^{\prime}}t, \,\ t \in [0; +\infty)$ is reached at the point $t_{max}=\|\nabla v\|^{p}_{L^{p}(U)}$, 
\begin{equation}
\sup_{\sigma \in L^{p^{\prime}}(U; \mathbb{R}^{N})} \Psi(v,\sigma) \leq \|\nabla v\|_{L^{p}(U)}\cdot \|\nabla v\|_{L^{p}(U)}^{\frac{p}{p^{\prime}}}-\frac{1}{p^{\prime}}\|\nabla v\|^{p}_{L^{p}(U)} = \frac{1}{p}\int_{U}|\nabla v|^{p} \diff x. \label{A.18}
\end{equation}
By (\ref{A.16}) and (\ref{A.18}), we deduce (\ref{A.15}). Thus  we have that
\begin{equation*}
\max_{w\in W^{1,p}_{0}(U)} \biggl\{\int_{U}fw\diff x - \frac{1}{p}\int_{U}|\nabla w|^{p}\diff x\biggr\} = \max_{w\in W^{1,p}_{0}(U)} \min_{\sigma \in L^{p^{\prime}}(U; \mathbb{R}^{N})}\biggl\{\int_{U}fw\diff x - \Psi(w, \sigma)\biggr\}.
\end{equation*}
Now we want to exchange the max and min in the above formula. Clearly,
\begin{equation*}
\begin{split}
\max_{w\in W^{1,p}_{0}(U)} \min_{\sigma \in L^{p^{\prime}}(U; \mathbb{R}^{N})}\biggl\{\int_{U}fw\diff x - \Psi(w, \sigma)\biggr\} &\leq \inf_{\sigma \in L^{p^{\prime}}(U; \mathbb{R}^{N})}\sup_{w\in W^{1,p}_{0}(U)} \biggl\{ \int_{U}fw \diff x -\Psi(w,\sigma)\biggr\} \\ &=\inf_{\sigma\in D}\frac{1}{p^{\prime}} \int_{U} |\sigma|^{p^{\prime}}\diff x, 
\end{split}
\end{equation*}
where $D$ stands for the space of $\sigma\in  L^{p^{\prime}}(U; \mathbb{R}^{N})$ satisfying 
\begin{equation*}
\int_{U} \sigma \cdot \nabla \phi  \diff x= \int_{U} f\phi \diff x \text{ for all $\phi \in C^{\infty}_{0}(U)$},
\end{equation*}
otherwise the supremum in $w$ would be $+\infty$. This implies that
\[
\max_{w\in W^{1,p}_{0}(U)}\biggl\{\int_{U}fw\diff x - \frac{1}{p}\int_{U}|\nabla w|^{p}\diff x\biggr\} \leq \inf_{\sigma \in D} \frac{1}{p^{\prime}}\int_{U}|\sigma|^{p^{\prime}}\diff x.
\]
We observe that the optimality condition (\ref{A2}) on $ u$ yields $|\nabla u|^{p-2}\nabla u \in D$ and then
 \begin{align*}
 \frac{1}{p^{\prime}} \int_{U} |\nabla u|^{p} \diff x= \max_{w\in W^{1,p}_{0}(U)}\biggl\{\int_{U}fw\diff x - \frac{1}{p}\int_{U}|\nabla w|^{p}\diff x\biggr\} &\leq \inf_{\sigma \in D} \frac{1}{p^{\prime}} \int_{U} |\sigma|^{p^{\prime}} \diff x \\ &\leq  \frac{1}{p^{\prime}} \int_{U} |\nabla u|^{p} \diff x.
 \end{align*}
Therefore (\ref{A.14}) holds and $\sigma=|\nabla u|^{p-2}\nabla u$ is the minimizer.
\end{proof}


\bibliography{bib1}
\bibliographystyle{plain}

\end{document}